\documentclass[10pt,reqno]{amsart}

\usepackage[T1]{fontenc}
\usepackage[letterpaper,textwidth=400pt,textheight=584pt,centering]{geometry}
\usepackage{amssymb,amsfonts}
\usepackage{mathtools}
\usepackage[mathscr]{euscript}      
\usepackage{graphicx}
\usepackage{subcaption}
\usepackage[dvipsnames]{xcolor}
\usepackage[most]{tcolorbox}        
\usepackage{appendix}          
\usepackage{comment}

\definecolor{webgreen}{rgb}{0,.5,0}
\definecolor{webbrown}{rgb}{.6,0,0}
\definecolor{RoyalBlue}{cmyk}{1,0.50,0,0}
\definecolor{BrickRed}{rgb}{0.8,0.25,0.33}

\theoremstyle{plain}
\newtheorem{theorem}{Theorem}[section]

\newtheorem{lemma}[theorem]{Lemma}
\newtheorem{corollary}[theorem]{Corollary}
\newtheorem{conjecture}[theorem]{Conjecture}

\theoremstyle{definition}

\theoremstyle{remark}
\newtheorem{remark}[theorem]{Remark}

\numberwithin{equation}{section}

\tcbset{
  projectstyle/.style={
    colback=gray!7,
    colframe=BrickRed!35,
    fonttitle=\bfseries,
    coltitle=black,
    boxrule=0.4pt,
    arc=2mm, outer arc=2mm,
    breakable, enhanced,
    before skip=10pt, after skip=10pt,
    left=8pt, right=8pt, top=6pt, bottom=6pt
  },
  projecttitle/.style={title={#1}},
}

\newcommand{\R}{{\mathbb R}}

\newcommand{\Q}{{\mathbb Q}}

\newcommand{\N}{{\mathbb N}}

\newcommand{\T}{{\mathbb T}}

\newcommand{\al}{\alpha}

\newcommand{\be}{\beta}

\newcommand{\ga}{\gamma}
\newcommand{\Ga}{\Gamma}

\newcommand{\la}{\lambda}

\newcommand{\de}{\delta}

\newcommand{\di}{\displaystyle}

\newcommand{\dd}{\textrm{d}}          
\newcommand{\ddd}{\partial}           
\AtBeginDocument{}

\newcommand{\qasq}{\quad \text{as} \quad}
\newcommand{\qandq}{\quad \text{and} \quad}

\ExplSyntaxOn
\NewDocumentCommand{\pFq}{O{}mmmmm}
 {
  \group_begin:
  \keys_set:nn { hypergeometric } { #1 }
  \hypergeometric_print:nnnnn { #2 } { #3 } { #4 } { #5 } { #6 }
  \group_end:
 }
\NewDocumentCommand{\hypergeometricsetup}{m}
 {
  \keys_set:nn { hypergeometric } { #1 }
 }
\tl_new:N \l_hypergeometric_divider_tl
\tl_new:N \l_hypergeometric_left_tl
\tl_new:N \l_hypergeometric_right_tl
\keys_define:nn { hypergeometric }
 {
  symbol .tl_set:N = \l_hypergeometric_symbol_tl,
  symbol .initial:n = F,
  separator .tl_set:N = \l_hypergeometric_separator_tl,
  separator .initial:n = {},
  skip .tl_set:N = \l_hypergeometric_skip_tl,
  skip .initial:n = 8,
  divider .choice:,
  divider/semicolon .code:n = \tl_set:Nn \l_hypergeometric_divider_tl { \;; },
  divider/bar .code:n = \tl_set:Nn \l_hypergeometric_divider_tl { \;\middle|\; },
  divider .initial:n = semicolon,
  fences .choice:,
  fences/brack .code:n =
   \tl_set:Nn \l_hypergeometric_left_tl {[}
   \tl_set:Nn \l_hypergeometric_right_tl {]},
  fences/parens .code:n =
   \tl_set:Nn \l_hypergeometric_left_tl {(}
   \tl_set:Nn \l_hypergeometric_right_tl {)},
  fences .initial:n = brack,
 }
\cs_new_protected:Nn \hypergeometric_print:nnnnn
 {
  {} \sb {#1} \l_hypergeometric_symbol_tl \sb { #2 }
  \left\l_hypergeometric_left_tl
  \genfrac .. {0pt} {}
           { \__hypergeometric_process:n { #3 } }
           { \__hypergeometric_process:n { #4 } }
  \l_hypergeometric_divider_tl
  #5
  \right\l_hypergeometric_right_tl
 }
\cs_new_protected:Nn \__hypergeometric_process:n
 {
  \clist_use:nn { #1 }
   {
    {\l_hypergeometric_separator_tl}
    \mspace { \l_hypergeometric_skip_tl mu }
   }
 }
\ExplSyntaxOff

\hypergeometricsetup{
  fences=parens,
  separator={,},
  divider=bar,
}

\usepackage{hyperref}
\hypersetup{
    colorlinks=true,
    linkcolor=RoyalBlue,
    filecolor=red,
    urlcolor=RoyalBlue,
    citecolor=webgreen,
}

\begin{document}

\title[Enumeration of Even-valent maps]{Explicit enumeration and large-valence asymptotics
of even-valent maps}

\author{Roozbeh Gharakhloo}
\address{Mathematics Department, University of California, Santa Cruz,
1156 High Street, Santa Cruz, CA 95064, USA}
\email{roozbeh@ucsc.edu}

\author{Tomas Lasic Latimer}
\address{Mathematics Department, University of California, Santa Cruz,
1156 High Street, Santa Cruz, CA 95064, USA}
\email{tlasicla@ucsc.edu}


\begin{abstract}
Let $\mathscr{N}_g(2\nu,j)$ denote the number of connected labelled
$2\nu$-valent maps of genus $g$ with $j$ vertices. Explicit bivariate
formulae for $\mathscr{N}_g(2\nu,j)$ have been available only in the
planar~\cite{EMP08} and toroidal~\cite{Erc-et-al-23b} cases.  Using a structural formula of \cite{Erc-et-al-23b}, we translate the problem of determining an explicit bivariate formula for $\mathscr{N}_g(2\nu,j)$, $g \geq 2$, to finding finitely many counts with a fixed number of vertices. For $g=2,3$ and $4$ we determine these counts using the
associated orthogonal polynomials, yielding explicit bivariate formulae for $\mathscr{N}_g(2\nu,j)$ in these
genera. Furthermore, the same method applies for every $g\ge5$ at the cost of additional
computation. From these formulae we obtain the leading-order asymptotics of
$\mathscr{N}_g(2\nu,j)$ as $\nu\to\infty$ for $g=2,3,4$, and we conjecture the
structure of these formulae in general genus. In addition, we establish an analogous reduction to finitely many
counts, derive explicit formulae and large-valence asymptotics, and
formulate corresponding conjectures for two-legged even-valent maps.
\end{abstract}

\maketitle



\section{Introduction and Main Results}\label{sect:intro}

Let a \textit{map} be a labeled, connected graph embedded in a compact, oriented, and connected Riemann surface such that the complement of the graph is a disjoint union of cells. The problem of enumerating maps for a fixed choice of
\begin{itemize}
\item[]i) number of vertices,
\item[]ii) valence at each vertex, and
\item[]iii) genus of the underlying surface
\end{itemize}
has inspired a rich body of research, drawing on both purely combinatorial methods and techniques from random matrix theory. The earliest work on map enumeration was carried out by Tutte \cite{Tutte}, who used a combinatorial approach to count maps embedded on the plane. Later efforts extended to maps on Riemann surfaces of low genus \cite{Brown, Arques}, as well as to the study of the asymptotic behaviour of these enumerations as the number of vertices grow \cite{BC0, BC1}. In addition to these purely combinatorial approaches, random matrix models have emerged as a powerful tool for deriving generating functions in map enumeration problems. Motivated by applications in quantum field theory, the connection of random matrix models to map enumeration was first established in the seminal work of Brezin, Itzykson, Parisi and Zuber \cite{BIPZ} (based on the earlier work of 't Hooft \cite{tHooft}). Efforts to develop models of quantum gravity led to a simplified two-dimensional framework, in which collections of distinct geometries on Riemann surfaces are analyzed under a natural probability measure. A key challenge in this context has been understanding the distribution of these geometries, which requires determining their total number, making it a map enumeration problem. The seminal works of theoretical physicists such as David \cite{David85}, Kazakov \cite{Kaz85}, Witten \cite{Witten}, and Bessis, Itzykson, and Zuber \cite{Bes79, BIZ}, built on the connections between two-dimensional quantum gravity and random matrix models and in doing so further developed the field of map enumeration.  We refer to \cite{DiF}, \cite{FGZ}, and \cite{Zvonkin} as excellent reviews on these developments.

\begin{figure}[!htp]
    \centering
    \begin{minipage}{0.24\textwidth}
        \centering
        \includegraphics[width=\textwidth]{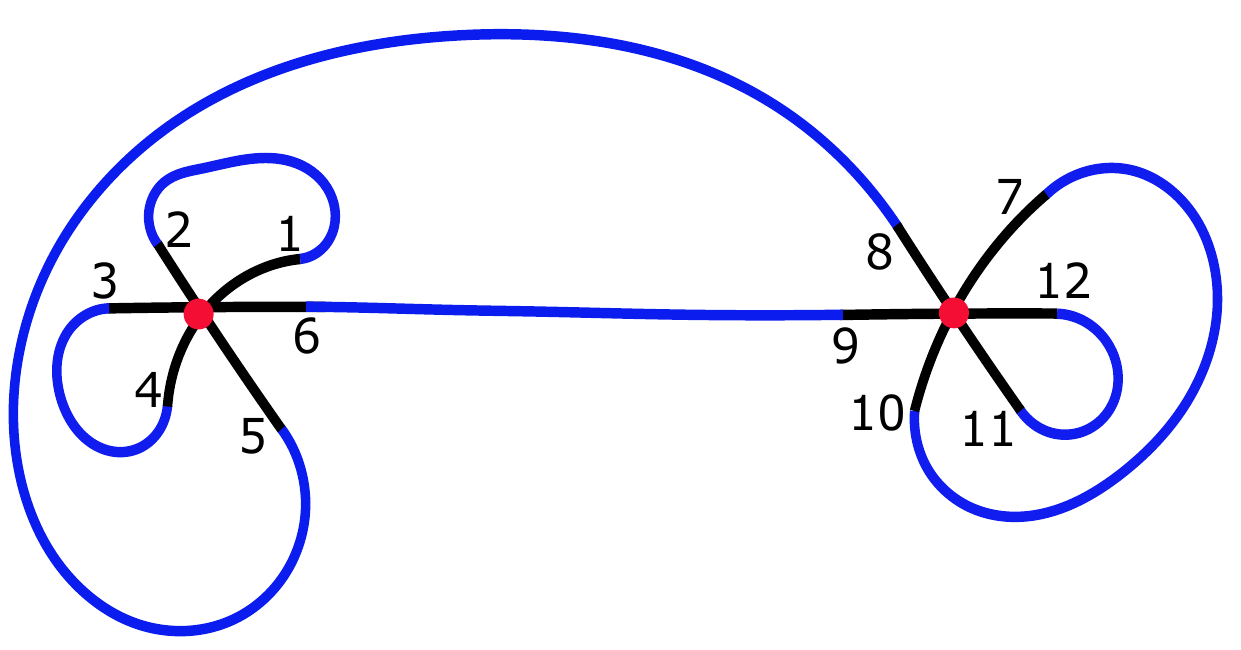}
        \subcaption{}
    \end{minipage}\hfil
    \begin{minipage}{0.24\textwidth}
        \centering
        \includegraphics[width=\textwidth]{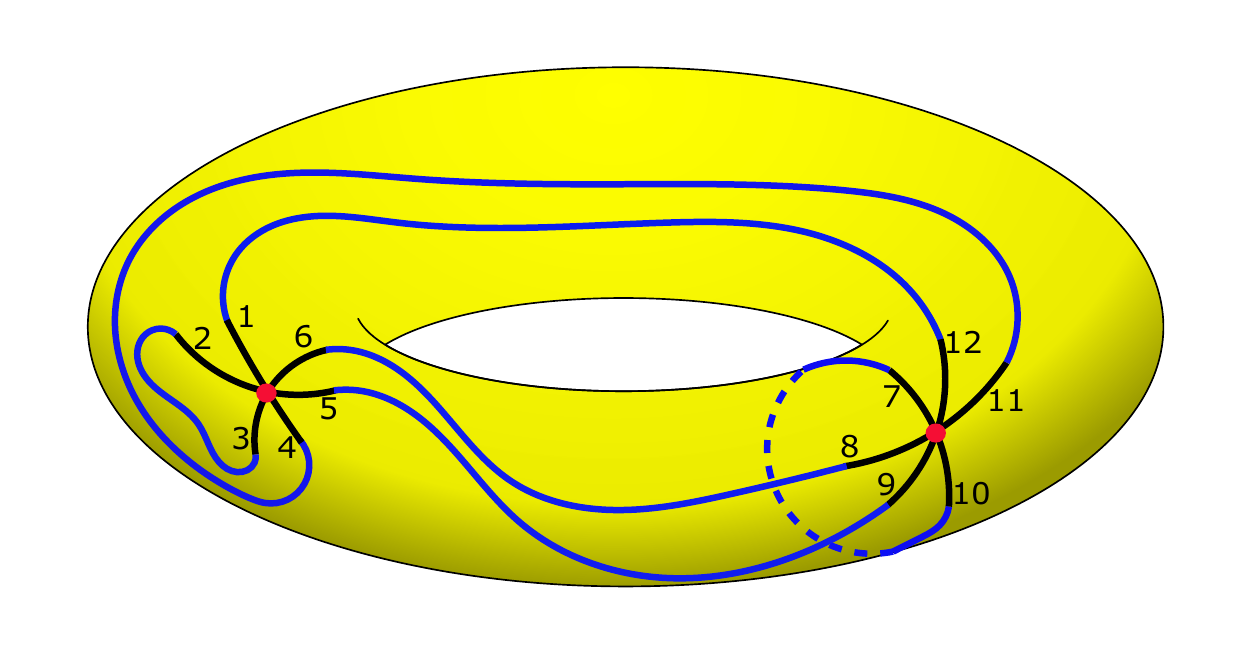}
        \subcaption{}
    \end{minipage}\hfil
    \begin{minipage}{0.24\textwidth}
        \centering
        \includegraphics[width=\textwidth]{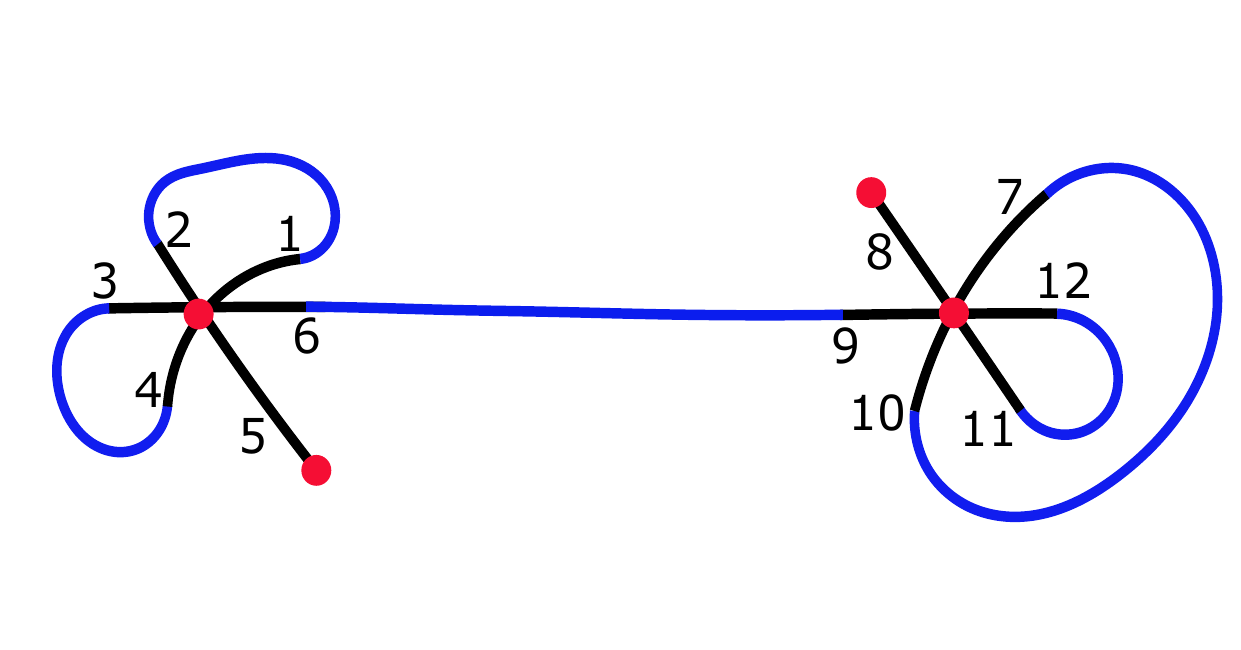}
        \subcaption{}
    \end{minipage}\hfil
    \begin{minipage}{0.24\textwidth}
        \centering
        \includegraphics[width=\textwidth]{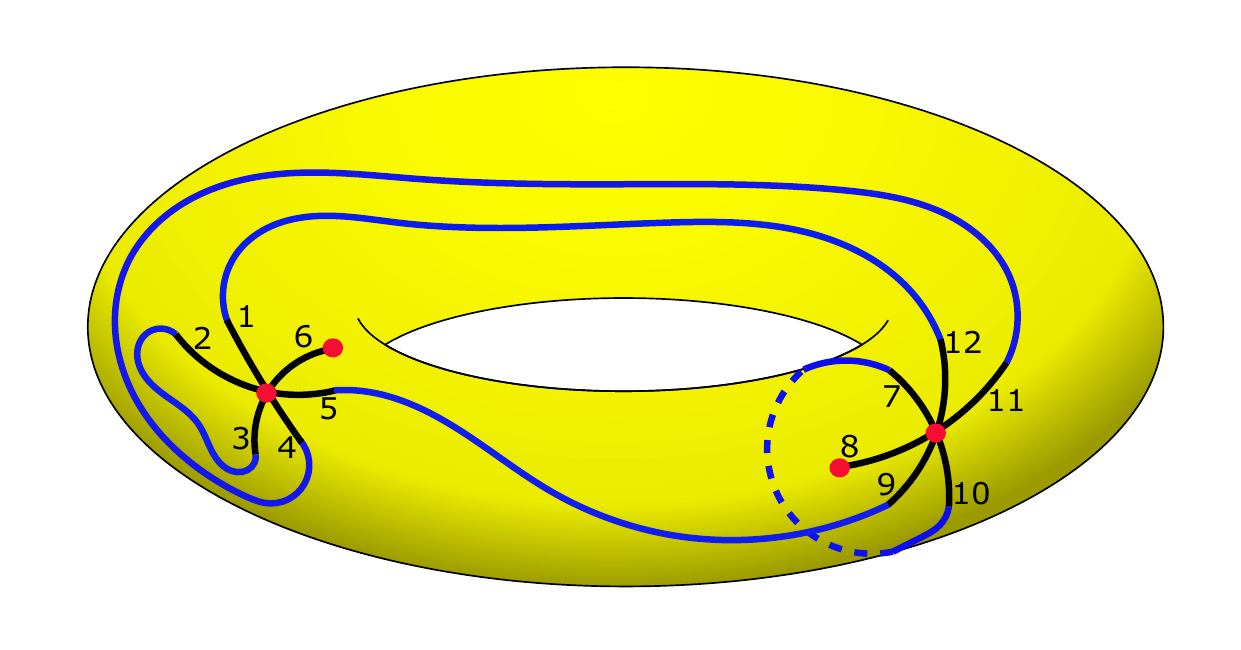}
        \subcaption{}
    \end{minipage}
    \caption{Examples of regular and two-legged graphs.  Subfigures (A) and
    (B) show regular \(6\)-valent graphs with two vertices on the sphere and
    torus, respectively; the graph in (B) has minimal embedding genus one.
    Subfigures (C) and (D) show corresponding two-legged examples.}
    \label{fig:intro}
\end{figure}

\textbf{Notation.} Let \(\mathscr N_g(\mu,j)\) denote the number of connected labeled graphs with
\(j\) \(\mu\)-valent vertices and minimal embedding genus \(g\).  Here
\emph{labeled} means that the \(\mu j\) half-edges are labeled and that a cyclic
order of the half-edges at each vertex has been fixed; the graph is a matching
of these half-edges into edges, so \(\mu j\) is assumed even.  The adjective
``minimal'' is important: a graph that embeds in a surface of genus \(g_0\) also
embeds in any surface of genus \(g\geq g_0\), since the additional handles
provide more flexibility.  Thus, for example, \(\mathscr
N_2(\mu,j)\) does not include graphs already embeddable in the plane or in the
torus.  Similarly, let \(\mathcal N_g(\mu,j)\) denote the number of connected
two-legged labeled graphs with \(j\) \(\mu\)-valent vertices and minimal
embedding genus \(g\).  Here a \emph{leg} means a one-valent vertex together
with its unique incident edge, and the two legs are ordered; see
Figure~\ref{fig:intro}.

\begin{tcolorbox}[projectstyle={Custom}, colback=yellow!6, colframe=orange!70!black]
The purpose of the present paper is to obtain \textit{fully explicit} formulae for
\[
    \mathscr N_g(2\nu,j)
    \qquad\text{and}\qquad
    \mathcal N_g(2\nu,j),
\]
with both the valence parameter \(\nu\) and the number of vertices \(j\) treated
as variables, for fixed genus \(g\).
\end{tcolorbox}

  The first fully explicit result of this
type is due to Ercolani, McLaughlin, and Pierce
\cite{Ercolani-McLaughlin-Pierce}, who proved that, in genus zero,
\begin{equation}\label{numbers sphere}
    \mathscr N_0(2\nu,j)
    =
    c_\nu^j
    \frac{(\nu j-1)!}{((\nu-1)j+2)!},
    \qquad
    c_\nu:=\nu\binom{2\nu}{\nu}.
\end{equation}
More recently, Ercolani, Lega, and Tippings~\cite{Erc-et-al-23b} obtained the
following fully explicit genus-one formula:
\begin{multline}\label{eq:e1_counts}
{\mathscr N}_{1}(2\nu,j)
=
\dfrac{j!\, c_\nu^j}{12}
\Bigg(
(\nu - 1)
{\nu j-1 \choose j-1}
\pFq{3}{2}{1,1,1-j}{2,(\nu-1)j+1}{1-\nu}
\\
-
(\nu - 1)^2
{\nu j-1 \choose j-2}
\pFq{3}{2}{1,1,2-j}{2,(\nu-1)j+2}{1-\nu}
\Bigg).
\end{multline}

The same work also established \textit{structural} formulae in arbitrary genus in terms
of Gauss hypergeometric functions.  For \(g\geq 2\) and \(j\geq 1\), they proved
that
\begin{equation}\label{eq:N-ELT-hypergeom}
{\mathscr N}_{g}(2\nu,j)
=
j! c_\nu^j (\nu - 1)^j
\sum_{\ell = 0}^{3g-3}
b^{(g,\nu)}_{\ell}
{2g + \ell+j -4  \choose j}
\pFq{2}{1}
{-j,1-\nu j}{4-2g-(\ell+j)}
{\frac{1}{1-\nu}},
\end{equation}
and, for two-legged graphs with \(g\geq 1\) and \(j\geq 1\),
\begin{equation}\label{eq:Ncal-ELT-hypergeom}
{\mathcal N}_{g}(2\nu,j)
=
j! c_\nu^j (\nu-1)^{j}
\sum_{\ell = 0}^{3g-1}
a^{(g,\nu)}_{\ell}
{2g + \ell+j -2  \choose j}
\pFq{2}{1}
{-j,-\nu j}{2-2g-(\ell+j)}
{\frac{1}{1-\nu}}.
\end{equation}
These formulae mark a significant advancement in the enumeration of regular
even-valent maps.  However, they do not by themselves give fully explicit
formulae for the counts as for each genus \(g\), one must still determine the
\(3g-2\) coefficients \(b_\ell^{(g,\nu)}\) and the \(3g\) coefficients
\(a_\ell^{(g,\nu)}\) as functions of the valence parameter \(\nu\).

The present work fills this gap in the first nontrivial higher-genus cases,
while also providing a framework for obtaining the sought-after fully explicit
formulae for any prescribed genus, with increasing computational effort.  For
\(g=2,3,4\), we explicitly determine all coefficients 
\(b_{\ell}^{(g,\nu)}\) appearing in \eqref{eq:N-ELT-hypergeom} and find them to be polynomials of degree $3g-3$; and, for
\(g=1,2,3,4\), we explicitly determine all coefficients 
\(a_\ell^{(g,\nu)}\) appearing in \eqref{eq:Ncal-ELT-hypergeom} and find them to be polynomials of degree $3g-1$.  Substitution
of these coefficients into the structural formulae of
Ercolani--Lega--Tippings yields explicit closed formulae for
\(\mathscr N_g(2\nu,j)\) and \(\mathcal N_g(2\nu,j)\), with both \(j\) and
\(\nu\) treated as variables.  These formulae also lead directly to leading-order
large-valence asymptotics as \(\nu\to\infty\) which requires the knowledge of \textit{all} the coefficients \(b_{\ell}^{(g,\nu)}\) and \(a_{\ell}^{(g,\nu)}\) for a fixed $g$. This is in contrast with the leading-order
large-vertex asymptotics which only requires the knowledge of the top coefficients \(b_{3g-3}^{(g,\nu)}\) and
\(a_{3g-1}^{(g,\nu)}\) \cite{Erc-et-al-23b}.

\vspace{.2cm}

We now state the main results. The explicit determination of the coefficient
polynomials \(b_\ell^{(g,\nu)}\) and 
\(a_\ell^{(g,\nu)}\) is the key step that makes the formulae below \textit{fully explicit}.  To
preserve the flow of the introduction, we state the main enumeration
formulae here and provide the lengthy coefficient lists themselves in
Section~\ref{sec:coefficient-tables}.

\begin{theorem}\label{thm:main-N}
Let \(g\in\{2,3,4\}\) and \(j\geq 1\).  Then
\begin{equation}\label{eq:main-N}
    \mathscr{N}_g(2\nu,j)
    =
    c_{\nu}^j j!
    \sum_{r=0}^j
    (\nu-1)^r
    \binom{\nu j-1}{j-r}
    \sum_{\ell=0}^{3g-3}
    b_{\ell}^{(g,\nu)}
    \binom{2g-4+\ell+r}{r}.
\end{equation}
For each \(g\in\{2,3,4\}\), the coefficients
\(b_\ell^{(g,\nu)}\), \(0\leq \ell\leq 3g-3\), are explicit polynomials in
\(\nu\) of degree \(3g-3\) listed in Section \ref{Sec explicit b coeffs}.
\end{theorem}

\begin{theorem}\label{thm:main-Ncal}
Let \(g\in\{1,2,3,4\}\) and \(j\geq 1\).  Then
\begin{equation}\label{eq:main-Ncal}
    \mathcal{N}_g(2\nu,j)
    =
    c_{\nu}^j j!
    \sum_{r=0}^j
    (\nu-1)^r
    \binom{\nu j}{j-r}
    \sum_{\ell=0}^{3g-1}
    a_{\ell}^{(g,\nu)}
    \binom{2g-2+\ell+r}{r}.
\end{equation}
For each \(g\in\{1,2,3,4\}\), the coefficients
\(a_\ell^{(g,\nu)}\), \(0\leq \ell\leq 3g-1\), are explicit polynomials in
\(\nu\) of degree \(3g-1\) listed in Section \ref{Sec explicit a coeffs}.
\end{theorem}

We emphasize that the coefficient polynomials in
\eqref{eq:main-N} and \eqref{eq:main-Ncal} are the same coefficients that appear
in the hypergeometric formulae \eqref{eq:N-ELT-hypergeom} and
\eqref{eq:Ncal-ELT-hypergeom}.  

\vspace{.2cm}

The next two corollaries describe the leading-order large-valence behavior of
these formulae.  For this purpose, define
\begin{equation}\label{eq:B-def}
    B_\ell^{(g)}
    :=
    \lim_{\nu\to\infty}
    \frac{b_\ell^{(g,\nu)}}{\nu^{3g-3}},
    \qquad
    0\leq \ell\leq 3g-3,
\end{equation}
and
\begin{equation}\label{eq:A-def}
    A_\ell^{(g)}
    :=
    \lim_{\nu\to\infty}
    \frac{a_\ell^{(g,\nu)}}{\nu^{3g-1}},
    \qquad
    0\leq \ell\leq 3g-1.
\end{equation}
For the genera appearing in the following corollaries, these constants are
explicit, since the corresponding coefficient polynomials
\(b_\ell^{(g,\nu)}\) and \(a_\ell^{(g,\nu)}\) are explicitly given in
Sections~\ref{Sec explicit b coeffs} and~\ref{Sec explicit a coeffs}.
\begin{corollary}
\label{cor:large-nu-N}
For fixed \(j\geq 1\) and \(2\leq g\leq 4\), as \(\nu\to\infty\),
\begin{equation}\label{eq:large-nu-N}
    \frac{\mathscr{N}_g(2\nu,j)}
    {4^{\nu j}\nu^{3j/2+3g-3}}
    =
    C_{g,j}^{\mathscr N}
    +
    \mathcal O(\nu^{-1}),
\end{equation}
where
\begin{equation}\label{eq:C-N-def}
    C_{g,j}^{\mathscr N}
    =
    \frac{j!}{\pi^{j/2}}
    \sum_{r=0}^j
    \frac{j^{j-r}}{(j-r)!}
    \sum_{\ell=0}^{3g-3}
    B_{\ell}^{(g)}
    \binom{2g-4+\ell+r}{r}.
\end{equation}
\end{corollary}

\begin{corollary}
\label{cor:large-nu-Ncal}
For fixed \(j\geq 1\) and \(1\leq g\leq 4\), as \(\nu\to\infty\),
\begin{equation}\label{eq:large-nu-Ncal}
    \frac{\mathcal{N}_g(2\nu,j)}
    {4^{\nu j}\nu^{3j/2+3g-1}}
    =
    C_{g,j}^{\mathcal N}
    +
    \mathcal O(\nu^{-1}),
\end{equation}
where
\begin{equation}\label{eq:C-Ncal-def}
    C_{g,j}^{\mathcal N}
    =
    \frac{j!}{\pi^{j/2}}
    \sum_{r=0}^j
    \frac{j^{j-r}}{(j-r)!}
    \sum_{\ell=0}^{3g-1}
    A_{\ell}^{(g)}
    \binom{2g-2+\ell+r}{r}.
\end{equation}
\end{corollary}

The genus-zero and genus-one cases fit naturally with this large-valence
asymptotics.  From \eqref{numbers sphere}, for fixed \(j\geq 1\),
\begin{equation}\label{eq:N0-asym}
    \frac{\mathscr N_0(2\nu,j)}
    {4^{\nu j}\nu^{3j/2-3}}
    =
    \frac{j^{j-3}}{\pi^{j/2}}
    +
    \mathcal O(\nu^{-1}),
    \qquad \nu\to\infty.
\end{equation}
Likewise, from \eqref{eq:e1_counts},
\begin{equation}\label{eq:N1-asym}
    \frac{\mathscr N_1(2\nu,j)}
    {4^{\nu j}\nu^{3j/2}}
    =
    \frac{j!}{12j\,\pi^{j/2}}
    \sum_{r=1}^{j}
    \frac{j^{j-r}}{(j-r)!}
    +
    \mathcal O(\nu^{-1}),
    \qquad \nu\to\infty.
\end{equation}
Thus in the regular case, the leading order asymptotic behavior in \(\nu\)  is
\[
4^{\nu j}\nu^{3j/2+3g-3}
\]
for all \(0\leq g\leq 4\).

  An explicit recursive formula in
\(g\), analogous to the recursion arising from the Painlev\'e I hierarchy, was
determined for the top coefficients \(b_{3g-3}^{(g,\nu)}\) and
\(a_{3g-1}^{(g,\nu)}\) in \cite{ErcolaniCaustics}.  This information was
sufficient for the large-\(j\) asymptotic analysis carried out in
\cite{Erc-et-al-23b}.  By contrast, the large-\(\nu\) asymptotics considered
here requires knowledge of \textit{all} coefficients \(b_\ell^{(g,\nu)}\) and
\(a_\ell^{(g,\nu)}\), not only the top ones.  This is why the explicit
determination of the full coefficient families is essential for the large-$\nu$
asymptotics.

The computations for \(g\leq 4\) suggest the following structural conjectures.

\begin{conjecture}\label{conj:b-polynomial}
For \(g\geq 5\) and \(0\leq \ell\leq 3g-3\), the coefficient
\(b_\ell^{(g,\nu)}\) is a polynomial in \(\nu\) of degree \(3g-3\).
\end{conjecture}

\begin{conjecture}\label{conj:a-polynomial}
For \(g\geq 5\) and \(0\leq \ell\leq 3g-1\), the coefficient
\(a_\ell^{(g,\nu)}\) is a polynomial in \(\nu\) of degree \(3g-1\).
\end{conjecture}

\begin{tcolorbox}[projectstyle={Custom}, colback=yellow!6, colframe=orange!70!black]
If Conjectures~\ref{conj:b-polynomial} and \ref{conj:a-polynomial} hold, then
the large-valence asymptotic formulae in
Corollaries~\ref{cor:large-nu-N} and \ref{cor:large-nu-Ncal} structurally extend to all
higher genera, with the constants defined by the leading coefficients of these
polynomials.

More generally, the method developed in this paper applies to any
fixed genus \(g\), although the computational complexity increases with \(g\).
Thus, even in the absence of a closed formula for
\[
    \mathscr N(g,2\nu,j)\equiv \mathscr N_g(2\nu,j)
\]
where \(g\), \(\nu\), and \(j\) are simultaneously treated as variables, the
framework presented in this paper, combined with the structural formulae of
Ercolani--Lega--Tippings, gives an effective procedure for computing the counts
for any prescribed \(g\), and arbitrary \(\nu\), and \(j\).\end{tcolorbox}

The main results of this paper are Theorems~\ref{thm:main-N} and
\ref{thm:main-Ncal}, together with their large-valence corollaries,
Corollaries~\ref{cor:large-nu-N} and~\ref{cor:large-nu-Ncal}.  Their proofs
rely on a collection of auxiliary fixed-genus, fixed-vertex results, which are
stated in Section~\ref{sec:fixed-gj-results}.  Before stating these auxiliary
results, we review the necessary background and recall the relevant ingredients
from random matrix theory and orthogonal polynomials.

\subsection*{Acknowledgements}
We would like to thank P.~Bleher, N.~Ercolani, J.~Lega, K.~McLaughlin, and B.~Tippings for
helpful discussions and suggestions.

\section{Background}

We recall the random matrix formulation of the enumeration problem and review
the map-enumeration results in the literature that are most relevant to the
present paper.
\subsection{Connections to Random Matrix Theory and Orthogonal Polynomials}
Consider the probability distribution 

	\begin{equation} \label{int1}
	\dd \mu_{nN}(M;\boldsymbol{t}) = \frac{1}{\Tilde{\mathcal{Z}}_{nN}(\boldsymbol{t})} e^{-N \mathrm{Tr}\,\mathscr{V}_{\boldsymbol{t}}(M)} \dd M,
	\end{equation}
on the space of $n \times n$ Hermitian matrices with the \textit{external field}
\begin{equation}\label{fancy V def1}
\mathscr{V}_{\boldsymbol{t}}(z) =  \frac{z^2}{2} + \sum_{j=1}^{m} t_j z^j,
\end{equation}
where $m \in 2\N$, $\boldsymbol{t}:= (t_1, \cdots, t_m)^T \in \R^m$, $t_m > 0$. In \eqref{int1} $\Tilde{\mathcal{Z}}_{nN}(\boldsymbol{t})$ is the \textit{partition function} of the matrix model and is given by
	\begin{equation} \label{int3}
	\Tilde{\mathcal{Z}}_{nN}(\boldsymbol{t}) = \int_{\mathscr{H}_n} e^{-N \mathrm{Tr}\,\mathscr{V}_{\boldsymbol{t}}(M)} \dd M.
	\end{equation}
 The eigenvalues of $M$ have the joint probablity distribution function
\begin{equation}
\frac{1}{\mathcal{Z}_{nN}(\boldsymbol{t})}\prod_{1\leq j<k\leq n} (z_j-z_k)^2\prod_{j=1}^n\exp\left[ -N \mathscr{V}_{\boldsymbol{t}}(z_j)\right],
\end{equation}
where $\mathcal{Z}_{nN}(\boldsymbol{t})$ is the eigenvalue partition function and is defined as
\begin{equation}\label{u free energy}
    \mathcal{Z}_{nN}(\boldsymbol{t}) = \int_{-\infty}^\infty \ldots \int_{-\infty}^\infty \prod_{1\leq j<k\leq n} (z_j-z_k)^2\prod_{j=1}^n\exp\left[ -N  \mathscr{V}_{\boldsymbol{t}}(z_j)\right]\dd z_1 \ldots \dd z_n .
\end{equation} 
The connection of matrix models to map enumeration on Riemann surfaces lies in the asymptotic properties of the \textit{free energy}:
\begin{equation}\label{free energy}
    \mathcal{F}_{nN}(\boldsymbol{t}) :=  \frac{1}{n^2} \ln\frac{\mathcal{Z}_{nN}(\boldsymbol{t})}{\mathcal{Z}_{nN}(\boldsymbol{0})}.
\end{equation}
For any given $T>0$ and $\ga>0$, define
\[ \T(T,\ga) := \{ \boldsymbol{t} \in \R^m \ : \ |\boldsymbol{t}| \leq T, t_{m}>\ga \sum^{m-1}_{j=1}|t_j|  \}. \]
Let $x:=n/N$. It turns out that there exist $T>0$ and $\ga>0$ so that for all $\boldsymbol{t} \in \T(T,\ga)$ the free energy $\mathcal{F}_{nN}(\boldsymbol{t})$ admits an asymptotic expansion in powers of $N^{-2}$
\begin{equation}\label{top exp free energy}
    \mathcal{F}_{nN} (\boldsymbol{t}) = \sum_{g=0}^\infty \frac{f_{2g}(x,\boldsymbol{t})}{N^{2g}}, \qasq N \to \infty,
\end{equation}
in some neighborhood of $x=1$. The above expansion was established in \cite{EM02} for $\mathcal{F}_{NN}(\boldsymbol{t})$ (i.e. when $x=1$) and its existence was later generalized to be valid in a neighborhood of $x=1$ in \cite{Ercolani-McLaughlin-Pierce}. The asymptotic expansion \eqref{top exp free energy} is referred to as the \textit{topological expansion} for the associated matrix model, since for each $g \in \N$, the coefficient $f_{2g}(x,\boldsymbol{t})$ is a combinatorial generating function for graphs embedded on a Riemann surface of genus $g$. To this end, we would like to highlight the following result.
\begin{theorem}\label{thmEM02}\cite{EM02}
Let $\mathscr{N}_g(n_1,\cdots,n_{m})$ denote the number of mixed-valence\footnote{This notation should not be confused with $\mathscr{N}_g(\mu,j)$ introduced earlier to denote the number of \textit{regular} \(\mu\)-valent connected labeled graphs with
\(j\)  vertices and minimal embedding genus \(g\). In the notation of Theorem \ref{thmEM02}, $\mathscr{N}_g(\mu,j)$ is denoted as $\mathscr{N}_g(n_{\mu})$, where $n_{\mu}=j$ denotes the number of vertices and the subscript $\mu$ denotes the valence.}
 labeled connected graphs with $n_k$ number of $k$-valent vertices which can be embedded on a Riemann surface of minimal genus $g$. Then
    \begin{equation}
        f_{2g}(1, \boldsymbol{t}) = \sum_{n_k \geq 1} \frac{\mathscr{N}_g(n_1,\cdots,n_{m})}{n_1! \cdots n_m!}(-t_1)^{n_1} \cdots (-t_m)^{n_m}.
    \end{equation}
\end{theorem} 
The $n$-fold integral \eqref{u free energy} is, up to a factor of $n!$, the $n \times n$ Hankel determinant $D_n[w_{\boldsymbol{t}}] \equiv \det \{w_{j+k}\}_{0\leq j,k \leq n-1} $ associated with the weight $w_{\boldsymbol{t}}(x)=\exp(-N \mathscr{V}_{\boldsymbol{t}}(x))$, where $w_{j}$ is the $j$-th moment of the weight $w_{\boldsymbol{t}}(x)$. This is known as the Heine's formula for Hankel determinants \cite{SzegoOP} and relates the partition function, and thus the free energy, to the system of orthogonal polynomials on the real line associated with the weight $\exp(-N \mathscr{V}_{\boldsymbol{t}}(z))$:
\begin{equation}\label{OP}
    \int_{\R} \mathcal{P}_n(z;\boldsymbol{t}) \mathcal{P}_m(z;\boldsymbol{t}) \exp(-N \mathscr{V}_{\boldsymbol{t}}(z)) \dd z= h_n(\boldsymbol{t}) \de_{nm}, 
\end{equation}
where $h_n(\boldsymbol{t}) = D_{n+1}[w_{\boldsymbol{t}}]/D_n[w_{\boldsymbol{t}}]$, and $\de_{nm}$ is the Kronecker delta function. Orthogonal polynomials on the real line satisfy the three-term recurrence equation (see e.g. \cite{BL}):
\begin{equation}\label{OP rec}
    z\mathcal{P}_{n}(z) = \mathcal{P}_{n+1}(z) + \be_n \mathcal{P}_{n}(z) + \mathcal{R}_n\mathcal{P}_{n-1}(z).
\end{equation}

The relationship between this system of orthogonal polynomials and the partition function \(\mathcal{Z}_{nN}(\boldsymbol{t})\) is as follows: An orthogonal polynomial of degree \(n\) exists and is unique if the partition function \(\mathcal{Z}_{nN}(\boldsymbol{t})\), or equivalently, the \(n \times n\) Hankel determinant $D_n[w_{\boldsymbol{t}}]$, is nonzero. The existence of such a polynomial simply follows from the explicit formula:

\begin{equation}
\mathcal{P}_n(z;\boldsymbol{t}) = \frac{1}{D_n[w_{\boldsymbol{t}}]} \det \begin{pmatrix}
w_0 & w_1 & \cdots & w_{n-1} & w_{n} \\
w_1 & w_2 & \cdots  & w_{n} & w_{n+1} \\
\vdots & \vdots & \reflectbox{$\ddots$} & \vdots & \vdots \\
w_{n-1} & w_{n} & \cdots  & w_{2n-2} & w_{2n-1} \\
1 & z & \cdots & z^{n-1} & z^n
\end{pmatrix}.
\end{equation}
Uniqueness of these orthogonal polynomials follows from the fact that the coefficients of \(\mathcal{P}_n(z;\boldsymbol{t})\), expressed in the form 
$
\mathcal{P}_n(z;\boldsymbol{t}) = z^n + \sum_{j=0}^{n-1} a_j(\boldsymbol{t}) z^j,
$
are determined by a linear system 
$
H_n[w_{\boldsymbol{t}}] \boldsymbol{a} = \boldsymbol{b},
$ where $H_n[w_{\boldsymbol{t}}]$ is the $n \times n$ Hankel matrix, and $\boldsymbol{a}=(a_0(\boldsymbol{t}), \cdots,a_{n-1}(\boldsymbol{t}))^T$. Since this system can be inverted when the Hankel determinant is nonzero, the orthogonal polynomial $\mathcal{P}_n(z;\boldsymbol{t})$ is uniquely defined.

The Fokas-Its-Kitaev Riemann-Hilbert problem \cite{FIK} provides an effective analytical framework to obtain precise asymptotic information about orthogonal polynomials $\mathcal{P}_n(z;\boldsymbol{t})$ and thus the associated Hankel determinants. Since the partition function \eqref{u free energy} is equal to the Hankel determinant (up to a factor of $n!$) by the Heine formula, it can be used to obtain the asymptotics for the free energy \eqref{free energy}. For such asymptotic analysis of Hankel determinants see e.g. \cite{Charlier, CharlierGharakhloo} and references therein. 

Alternatively, without obtaining precise asymptotics for the partition function \eqref{u free energy} itself, one can find asymptotics of the free energy by employing the \textit{string} and \textit{Toda} equations which are difference and differential equations involving the recurrence coefficients of the orthogonal polynomials (see e.g. \cite{BGM}). In other words, using Toda and string equations, establishing asymptotic expansions like \eqref{top exp free energy} for the recurrence coefficients from the Riemann-Hilbert method:
\begin{equation}\label{Asymp ga_n^2}
\mathcal{R}_n(x;\boldsymbol{t}) \sim \sum^{\infty}_{g=0} \frac{r_{2g}(x;\boldsymbol{t})}{N^{2g}},
		\end{equation}
will in turn yield the topological expansion for the free energy.

If one allows the vector $\boldsymbol{t}$ to be complex, for the cubic and quartic potentials the validity of the topological expansion in certain subsets of the complex plane have been shown respectively in \cite{BDY} and \cite{BGM}. For other complex potentials the existence of the topological expansion \eqref{top exp free energy} is not known, although aspects of the associated system of orthogonal polynomials and their equilibrium measures have been studied in the literature, e.g. in \cite{BBGMT, DHK, KS, HKL}.

\subsection{Map Enumeration Results in the Literature}To contextualize the findings of this paper, we review key results in the literature on computing the numbers $\mathscr{N}_g(\mu,j)$ and $\mathcal{N}_g(\mu,j)$. To the best of our knowledge, no known results exist for these numbers in the context of mixed-valence graphs. However, several results are available for regular graphs. We outline these results in the following subsections.
\subsubsection{3-valent Graphs.} In \cite{BD}, Bleher and Dea\~{n}o found closed form formulae for $\mathscr{N}_0(3,2j)$ and $\mathscr{N}_1(3,2j)$, respectively for 3-valent graphs embedded on a Riemann surface of genus $0$ and $1$. For the sphere these numbers are described by
      \begin{equation}\label{bleher cubic sphere}
          \mathscr{N}_0(3,2j) = \frac{72^j \Ga(\frac{3j}{2})(2j)!}{2\Ga(j+3)\Ga(\frac{j}{2}+1)},
      \end{equation}
while for the torus the numbers are expressed in terms of a ${}_3F_2$ hypergeometric function:
\hypergeometricsetup{
  fences=parens,
  separator={,},
  divider=bar,
}
 \begin{equation}\label{bleher cubic torus}
          \mathscr{N}_1(3,2j) = \frac{5 \cdot 72^j \Ga(\frac{3j}{2})(2j)!}{48 (3j+2) \Ga(j+1)\Ga(\frac{j}{2}+1)}\pFq{3}{2}{-j+1,2,6}{5,-\frac{3j}{2}+1}{\frac{3}{2}}.
      \end{equation}
 Notice that there are no regular odd-valent graphs with an odd number of vertices. In \cite{Erc-et-al-multiple-scale}, tables provide counts of 3-valent graphs on surfaces of genus $g = 0$, $g = 1$, and $g = 2$, with the number of vertices ranging over even integers from 2 to 30. This was based on the work in \cite{EP2012} which provides a basis to calculate the counts of general odd valence graphs. Similarly, \cite{Dubrovin-Yang} contains numerical tables for the number of $3$-valent graphs embedded on surfaces of genus $g = 0$ through $g = 5$, where the number of vertices range over even integers from 2 to 12.

 \subsubsection{4-valent Graphs.}  The seminal work \cite{BIZ} of Bessis  Itzykson, and Zuber which was the first to discover the profound connection of matrix models and graph enumeration problems, has explicit formulae for the coefficients $f_0$, $f_2$, and $f_4$ for the case $\nu=2$. 

 There are a number of papers in which numerical tables for $\mathscr{N}_g(4,j)$ are calculated for selected choices of $g$ and $j$  \cite{pierce2006} , \cite{Dubrovin-Yang}, and \cite{Erc-et-al-multiple-scale}.

 In \cite{BGM}, among other things, the computations of $f_0$, $f_2$, and $f_4$ from \cite{BIZ} were rigorously verified and a recursive pathway to compute any $f_{2g}$ (and thus any $\mathscr{N}_g(4,j)$) was introduced. In particular, this led to closed form formulae for the numbers $\mathscr{N}_{g}(4,j)$, for genus $g=0,1,2$ and $3$. They found simple formulae for $\mathscr{N}_0(4,j)$ and $\mathscr{N}_1(4,j)$ which are indeed special cases of \eqref{numbers sphere} and \eqref{eq:e1_counts} when $\nu=2$. Formulae for $\mathscr{N}_2(4,j)$ and $\mathscr{N}_3(4,j)$ were also determined in terms of simple factorials, for example, for $\mathscr{N}_2(4,j)$ they found:

	\begin{equation}\label{N j+1 2}
	\mathscr{N}_{2}(4,j+1) = \frac{12^j\left( 2j+2 \right)!(28j+37)}{360(j+1)(j-1)!}-13j(j+1)j!48^{j-1}, \qquad j \in \N,
	\end{equation}
	where $\mathscr{N}_{2}(4,1)=0$ (which is a consequence of the fact that all labeled 4-valent graphs with one vertex are realizable on the sphere and the torus, in fact there are three such graphs). Later in \cite{Erc-et-al-23b}, similar explicit closed form formulae were obtained for $\mathscr{N}_g(4,j)$ for $4 \leq g \leq 7$ (and their two-legged counterparts for $1\leq g \leq 7$). These formulae were also in terms of elementary arithmetic functions and could readily be extended to higher genus using the methods presented in their paper with increasing computational effort. Note that in \cite{Erc-et-al-23b} these formulae were obtained independently of  \eqref{eq:N-ELT-hypergeom} and \eqref{eq:Ncal-ELT-hypergeom}, using a dynamical description based on vector difference equations; see Theorem 4.1 therein.

    \subsubsection{General Even-Valent Graphs} In the case of even-valent potentials \begin{equation}\label{fancy V def 2}
    \mathscr{V}(z;u) =  \frac{z^2}{2} + u  \frac{z^{2\nu}}{2\nu}, \qquad u>0,
\end{equation} Ercolani found structural formulae for $f_{2g}$ and $r_{2g}$ for any $g \geq 2$ and any $\nu$ \cite{ErcolaniCaustics}.  In \cite{ErcolaniCaustics}, it was shown that
    \begin{equation}\label{PQ}
    r_{2g}(r_0) = \frac{r_0(r_0-1)P_{3g-2}(r_0)}{(\nu-(\nu-1)r_0)^{5g-1}},     \qandq f_{2g}(r_0) = \frac{(r_0-1)Q_{d(g)}(r_0)}{(\nu-(\nu-1)r_0)^{o(g)}},\end{equation}
    where $P_{m}$ (and $Q_{m}$) is a polynomial of degree $m$ in $r_0$, the leading (constant) term in the expansion of \eqref{Asymp ga_n^2}. The coefficients of $P_m$ and $Q_m$ are rational functions of $\nu$ over the rational numbers $\Q$. The exponent $o(g)$ and the degree $d(g)$ are non-negative integers to be determined. It turns out, for the potential \eqref{fancy V def 2}, the leading (constant) term $r_0$ in the expansion \eqref{Asymp ga_n^2} is a solution of the algebraic equation 
    \begin{equation}\label{r_0 string}
        r_0+ c_{\nu}x^{\nu-1}t_{2\nu}r_0^{\nu}=1, \qquad c_{\nu} := 2 \nu \binom{2\nu-1}{\nu-1}.
    \end{equation}

    In \cite{Ercolani-McLaughlin-Pierce}, Ercolani, McLaughlin and Pierce found the closed form formula \eqref{numbers sphere} for all planar even-valent graphs, that is, the number of $2\nu$-valent graphs on the sphere, where the number of vertices $ j$ and the valence $\nu$ are general.   In view of Theorem \ref{thmEM02}, this was obtained from the following explicit expression for $f_0$
    \begin{equation}\label{f0 to r0}
f_0(x,t_{2\nu})=\eta\left(r_0-1\right)(r_0-\kappa)+\frac{1}{2}\log(r_0),  
       \end{equation} with \begin{equation} \label{eta and kappa}
           \eta := \frac{(\nu-1)^2}{4\nu(\nu+1)} , \qandq \kappa:= \frac{3(\nu+1)}{\nu-1},
       \end{equation}  where $r_0 \equiv r_0(x,t_{2\nu})$ is the solution of the algbraic equation \eqref{r_0 string}. To obtain \eqref{numbers sphere} from \eqref{f0 to r0}, the authors used residue calculations to compute the Taylor coefficients of $r_0$, $r^2_0$ and $\log(r_0)$ \cite{Ercolani-McLaughlin-Pierce}. Moreover in \cite{Ercolani-McLaughlin-Pierce}, essential calculations for expressing $r_2$, $r_4$, and $r_6$ in terms of $r_0$ were performed, and equations for expressing $f_2$ and $f_4$ in terms of $r_0$ were also derived\footnote{For formulae expressing $r_2$, $r_4$, and $r_6$ in terms of $r_0$, see Sections 5.3, 5.4, and 5.5 of \cite{Ercolani-McLaughlin-Pierce}, respectively. For formulae expressing $f_2$ and $f_4$ in terms of $r_0$, see Sections 5.8 and 5.9 of \cite{Ercolani-McLaughlin-Pierce}, respectively.}. In \cite{Erc-et-al-23b}, Ercolani, Lega, and Tippings derived the torus analogue of \eqref{numbers sphere}, namely \eqref{eq:e1_counts}, using the results of \cite{Ercolani-McLaughlin-Pierce}. Deriving the analogs of this explicit formula for general $\nu$ and $j$ to higher genus, requires a significant amount of algebraic computation and the evaluation of integration constants. These constants are evaluated by calculating combinatorial counts of graphs by another means for fixed $g$, $j$ and $\nu$, and comparing results\footnote{See \cite[Section 5.10]{Ercolani-McLaughlin-Pierce} for such an evaluation for $g\leq 3$.}. This process highlights the difficulty of extending the method used in \cite{Ercolani-McLaughlin-Pierce} to higher genus. Other notable work in this area includes \cite{WatersSoln} who carried out similar calculations to \cite{Ercolani-McLaughlin-Pierce} but for the odd valence case. 

 We would also like to highlight two works which provided numerical tables for $\mathscr{N}_g(\mu,j)$ for valences higher than four, however, closed form formulae were not produced.  In \cite{pierce2006} V. Pierce provided numerical tables for 1-vertex $2\nu$-valent graphs for $0 \leq g \leq 5$ and $2 \leq \nu \leq 10 $ and also numerical tables for 2-vertex $\nu$-valent graphs for $0 \leq g \leq 4$ and $3 \leq \nu \leq 10$. Later, Dubrovin and Yang in \cite{Dubrovin-Yang} for $0 \leq g \leq 5$ provided numerical tables for  a) $\mathscr{N}_g(5,2j)$, $1\leq j\leq 5$, b) $\mathscr{N}_g(6,j)$, $1\leq j\leq 7$,   c) $\mathscr{N}_g(7,2j)$, $1\leq j\leq 4$, d) $\mathscr{N}_g(8,j)$, $1\leq j\leq 5$. As far as we know, the works \cite{pierce2006} and \cite{Dubrovin-Yang} are among the few works that provide actual counts for $g\geq2$ and $\nu>2$.  

For $g \geq 2$, no explicit formulae analogous to \eqref{numbers sphere} and \eqref{eq:e1_counts} exist in the literature. Substantial progress was reported in \cite{Erc-et-al-23b}, where ${\mathscr N}_{g}(2\nu,j)$ and ${\mathcal N}_{g}(2\nu,j)$ were expressed as linear combinations \eqref{eq:N-ELT-hypergeom} and \eqref{eq:Ncal-ELT-hypergeom} of ${}_2F_1$ hypergeometric functions. However, the coefficients in these combinations were left undetermined. Theorems \ref{thm:main-N} and \ref{thm:main-Ncal} of this paper determine these coefficients explicitly for $g=2,3,$ and $4$, while providing a roadmap for obtaining explicit results for all genera $g \geq 5$, requiring only additional computational effort.

\begin{remark}\normalfont
Although the present work is restricted to regular even-valent potentials, we
mention a complementary and more general framework developed in
\cite{Watersthesis,WatersSoln,EW2022}. For general polynomial potentials,
possibly involving several nonzero time parameters, the \(x\)-dependence of
the large-\(N\) expansions of the recurrence coefficients \(\be_n\) and
\(\mathcal{R}_n\) in \eqref{OP rec} can be used to obtain
\emph{valence-independent} representations of the generating functions.
Here, valence independence means that these representations contain no
explicit dependence on the potential or on its time parameters.

Let \(u_0(x;\boldsymbol{t})\) and \(r_0(x;\boldsymbol{t})\) denote,
respectively, the leading-order terms in the large-\(N\) expansions of
\(\be_n\) and \(\mathcal{R}_n\). The valence-independent representations are
expressed in terms of \(u_0\), \(r_0\), and their derivatives with respect to
the 't~Hooft parameter \(x=n/N\). For example, the genus-one generating function
is given by
\begin{equation}\label{F1Formula}
f_{2}(x;\boldsymbol{t})
=
-\frac{1}{12}\log\left(\frac{r_0}{x}\right)
+\frac{1}{24}\log\left(
\left(\frac{\partial r_0}{\partial x}\right)^2
-
r_0\left(\frac{\partial u_0}{\partial x}\right)^2
\right),
\end{equation}
see \cite[Equation~(81)]{WatersSoln}. A corresponding valence-independent
expression for the genus-two generating function is given in
\cite[Section~3.5]{WatersSoln}.

The higher \(x\)-derivatives appearing in these representations can be
recursively eliminated by means of the so-called \emph{unwinding identities}.
 Eliminating
the \(x\)-derivatives makes the dependence on the potential, and hence on the
valences, explicit. Nevertheless, the resulting reduced expressions are
rational functions of \(u_0\) and \(r_0\) at every genus; see
\cite[Theorem~2]{EW2022}. This is the ``general-potential'' analogue of the
rationality in \(r_0\) displayed in \eqref{PQ} for regular even-valent
potentials. In the even-potential setting considered in the present paper,
symmetry gives \(\be_n=0\), and hence \(u_0=0\), so that only \(r_0\)
remains nontrivial.
\end{remark}

\subsubsection{Large-$j$ and Large-$\nu$ Asymptotics} 

There are a number of results concerning the asymptotics of $\mathscr{N}_g(\mu,j)$ as $j\to\infty$. In \cite{BD} and \cite{BGM}, respectively, the leading order asymptotics of $\mathscr{N}_g(3,j)$ and $\mathscr{N}_g(4,j)$ were derived for graphs embedded on a Riemann surface of \textit{arbitrary} genus $g \in \N$, as the number of vertices tends to infinity. For $3$-valent regular graphs it was found in \cite{BD} that
 \begin{equation}\label{3 valent N j g large j asymptotics}
		\mathscr{N}_g(3,2j) = \mathcal{C}_g  \left( \frac{324}{\sqrt{3}}  \right)^{j}  (j)^{\frac{5g-7}{2}} (2j)! \left(1+O(j^{-1/2}) \right), \qquad j \to \infty. 
		\end{equation}
For $4$-valent regular graphs with $j$ vertices it was found in \cite{BGM} that
 \begin{equation}\label{N j g large j asymptotics}
		\mathscr{N}_g(4,j) = \mathcal{K}_g 48^{j}  (j)^{\frac{5g-7}{2}}(j)!\left(1+O(j^{-1/2}) \right), \qquad j \to \infty. 
		\end{equation}
        Additionally, descriptions of the constants $\mathcal{K}_g$ and $\mathcal{C}_g$ in terms of the asymptotics of the solutions
		$u(\tau)$ to the Painlev\'e I equation $	u''(\tau)=6u^2(\tau)+\tau$ were provided in \cite{BD} and \cite{BGM}. 
Recently, Ercolani and Waters~\cite{EW2022} obtained the leading-order
large-vertex asymptotics for graphs of arbitrary fixed genus \(g\in\N\)
and any fixed odd valence \(\mu\). More precisely, they showed that
\begin{equation}\label{odd valent N j g large j asymptotics}
    \mathscr{N}_g(2\nu-1,n)
    \sim
    \mathcal{J}_g
    \bigl(t_c^{1/2}\bigr)^n
    n^{\frac{5g-7}{2}}
    n!,
    \qquad
    n\to\infty,
\end{equation}
where \(n\) is even and \(t_c\) is the radius of convergence of the
Taylor expansion of \(f_{2g}(t_\nu)\).

In~\cite{EW2022}, and with a small correction
in~\cite{Erc-et-al-23b}, the corresponding large-vertex asymptotics for
graphs of fixed even valence \(2\nu\) were also established:
\begin{equation}\label{even valent N j g large j asymptotics}
    \mathscr{N}_g(2\nu,n)
    \sim
    \mathcal{K}_g^{(\nu)}
    \bigl(s_c^{-1}\bigr)^n
    n^{\frac{5g-7}{2}}
    n!,
    \qquad
    n\to\infty,
\end{equation}
where
\begin{equation}\label{def:large-vertex-constant}
    \mathcal{K}_g^{(\nu)}
    :=
    \frac{b_{3g-3}^{(g,\nu)}}
    {\left(\sqrt{\frac{2\nu}{\nu-1}}\right)^{5g-5}
    \Gamma\left(\frac{5g-5}{2}\right)}
\qandq 
    s_c
    :=
    \frac{(\nu-1)^{\nu-1}}
    {c_\nu\nu^\nu}.
\end{equation}
An analogous large-vertex asymptotic formula was also obtained for the
two-legged case.

The results of the present paper address the complementary
\emph{large-valence} regime. Namely, we obtain
leading-order asymptotics for \(\mathscr{N}_g(2\nu,j)\) as
\(\nu\to\infty\), with the number of vertices \(j\) fixed; see
    Corollaries~\ref{cor:large-nu-N} and~\ref{cor:large-nu-Ncal}. To the best of our knowledge, these asymptotic results are new. 
\begin{tcolorbox}[projectstyle={Custom}, colback=yellow!6, colframe=orange!70!black]    
    Note that in contrast to $\lim_{j\to\infty}\mathscr{N}_g(2\nu,j)$, understanding the leading order behavior of $\lim_{\nu\to\infty}\mathscr{N}_g(2\nu,j)$ requires knowledge of \textit{all} the coefficients $\{b_{\ell}^{(g,\nu)}\}_{\ell=0}^{3g-3}$, not just $\ell = 3g-3$ (the case $\ell = 3g-3$ is discussed in \cite{ErcolaniCaustics}). Determining all these coefficients is the task we undertake in this paper, at least for small values of $g$. It would be interesting to investigate whether the constants
\(C_{g,j}^{\mathscr N}\) (or \(C_{g,j}^{\mathcal N}\)) admit generating
functions analogous to the way in which
the constants \(\mathcal{K}_g\) are encoded in the asymptotics of a
Painlev\'e~I solution.
\end{tcolorbox}

\section{Auxiliary Results and Fixed-\texorpdfstring{\((g,j)\)}{(g,j)} Formulae}
\label{sec:fixed-gj-results}

The main results of this paper are Theorems~\ref{thm:main-N} and
\ref{thm:main-Ncal}.  Their proof relies on a collection of auxiliary results,
which we provide in this section. These results are proved in Sections \ref{determining rn section} and \ref{determining free section}.

We first recall the relation between the topological expansion and the
enumeration of regular and two-legged maps.  By Theorem~\ref{thmEM02}, for
regular \(2\nu\)-valent graphs one has
\begin{equation}\label{numbers and free energy}
    (-2\nu)^j
    \frac{\partial^j}{\partial u^j}
    f_{2g}(1,u)\bigg|_{u=0}
    =
    \mathscr{N}_g(2\nu,j).
\end{equation}
While for two-legged \(2\nu\)-valent graphs one has
\begin{equation}\label{2legged counts}
    (-2\nu)^j
    \frac{\partial^j}{\partial u^j}
    r_{2g}(1,u)\bigg|_{u=0}
    =
    \mathcal{N}_g(2\nu,j).
\end{equation}
Here \(f_{2g}\) and \(r_{2g}\) are the genus-\(g\)  coefficients of the free
energy and recurrence coefficient expansions, respectively.

\subsection{Taylor Coefficients of $r_{2g}$ and $f_{2g}$}
\label{new results: general even valence}

In this section we focus on results concerning the potential \eqref{fancy V def 2}, recall that $t_{2\nu} \equiv u/2\nu$ when compared to \eqref{fancy V def1}. The analyticity of $r_{2g}(x;\boldsymbol{t})$ and $f_{2g}(x;\boldsymbol{t})$ in a neighborhood of $(1;\boldsymbol{0})$ was established in \cite{Ercolani-McLaughlin-Pierce} for general even-degree potentials \eqref{fancy V def1}. The following theorems state that the Taylor coefficients of $r_{2g}$ and $f_{2g}$ are in fact monomials in the \textit{'t Hooft parameter} $x=n/N$.

\begin{theorem}
\label{recurrence theorem}
Consider the asymptotic expansion \eqref{Asymp ga_n^2} for the recurrence
coefficients of orthogonal polynomials with respect to the weight
\(e^{-N\mathscr V(z;u)}\), where \(\mathscr V\) is given by
\eqref{fancy V def 2}.  Let \(\beta_{2g,j}(x)\) denote the Taylor coefficients
of \(r_{2g}(x;u)\):
\begin{equation}\label{eq:beta-Taylor}
    r_{2g}(x;u)
    =
    \sum_{j=0}^{\infty}\beta_{2g,j}(x)u^j.
\end{equation}
Then
\begin{equation}\label{eq:beta-monomial}
    \beta_{2g,j}(x)
    =
    c_{2g,j}x^{\mathscr D},
    \qquad
    \mathscr D=j(\nu-1)+1-2g.
\end{equation}
If \(\mathscr D<0\), then \(\beta_{2g,j}(x)=c_{2g,j}=0\).
\end{theorem}

\begin{remark}\normalfont
As described in Section \ref{determining rn section}, we solve a hierarchy of
inhomogeneous differential equations to determine \(\beta_{2G,J}(x)\), in which
the coefficients \(\beta_{2g,j}(x)\) with \(g<G\) and \(j<J\) appear in the
inhomogeneous term. The significance of Theorem \ref{recurrence theorem} is
that it shows that the particular solution to these differential equations is a
simple monomial. This fact allows us to readily determine the coefficients
\(\beta_{2g,j}(x)\) and consequently \(\alpha_{2g,j}(x)\).
\end{remark}

As described in Section \ref{determining free section},
Theorem~\ref{recurrence theorem} leads to the following structural result for
\(f_{2g}(x;u)\).

\begin{theorem}
\label{free energy theorem}
Consider the asymptotic expansion \eqref{top exp free energy} for the free
energy \eqref{free energy} with respect to the weight
\(e^{-N\mathscr V(z;u)}\), where \(\mathscr V\) is given by
\eqref{fancy V def 2}.  Let \(\alpha_{2g,j}(x)\) denote the Taylor coefficients
of \(f_{2g}(x;u)\):
\begin{equation}\label{eq:alpha-Taylor}
    f_{2g}(x;u)
    =
    \sum_{j=0}^{\infty}\alpha_{2g,j}(x)u^j.
\end{equation}
Then
\begin{equation}\label{eq:alpha-monomial}
    \alpha_{2g,j}(x)
    =
    \widetilde c_{2g,j}x^{\widetilde{\mathscr D}},
    \qquad
    \widetilde{\mathscr D}=j(\nu-1)-2g.
\end{equation}
If \(\widetilde{\mathscr D}<-2\), then
\(\alpha_{2g,j}(x)=\widetilde c_{2g,j}=0\).
\end{theorem}

\begin{remark}\normalfont
The process for explicitly determining \(c_{2g,j}\) and
\(\widetilde c_{2g,j}\) using Equation \eqref{eq:volterra_lattice} is detailed
for the first few values of \(g\) and \(j\) at the end of
Sections \ref{determining rn section} and \ref{determining free section}.
Using the arguments presented in this paper, \(c_{2g,j}\) and
\(\widetilde c_{2g,j}\) can be determined for arbitrary \(g\) and \(j\), with
increasing computational effort as \(j\) and \(g\) become large.
\end{remark}

\begin{remark}\normalfont
In this paper, Theorems \ref{recurrence theorem} and
\ref{free energy theorem} are proven by analyzing the string equation for
general \(\nu\). Alternatively, they can also be derived using scaling
relations established in \cite[Lemma 11]{Watersthesis} for general even-degree
potentials; see also \cite[Equations (84)--(85)]{WatersSoln}.
\end{remark}

\subsection{Fixed-\texorpdfstring{\((g,j)\)}{(g,j)} Formulae}
\label{subsec:fixed-gj-combinatorial-formulae}

We now present the first combinatorial results of this paper.  In
Theorems~\ref{thm 2legged combinatorics} and~\ref{thm combinatorics}, we give
explicit formulae in \(\nu\) for \(\mathcal{N}_g(2\nu,j)\) and
\(\mathscr{N}_g(2\nu,j)\), respectively, for \(1\leq j\leq 3\) and
\(0\leq g\leq 5\).  These formulae involve powers of the quantity \(c_\nu\)
defined in \eqref{r_0 string}, which is related to the Catalan number \(C_\nu\) \cite{oeis_A000108}
by\footnote{For the reader's convenience to numerically interpret the results of Theorems \ref{thm 2legged combinatorics} and \ref{thm combinatorics}, we recall that the first ten values of
\(\{C_\nu\}_{\nu=1}^{\infty}\) are
\[
    1,\ 2,\ 5,\ 14,\ 42,\ 132,\ 429,\ 1430,\ 4862,\ 16796,
\]
and the corresponding first ten values of \(\{c_\nu\}_{\nu=1}^{\infty}\) are
\[
    2,\ 12,\ 60,\ 280,\ 1260,\ 5544,\ 24024,\ 102960,\ 437580,\ 1847560.
\]}
\begin{equation}\label{c to C}
    c_{\nu}=\nu(\nu+1)C_{\nu}.
\end{equation}

Note that the factor $c^j_{\nu}$ appears in all enumerative formulae discussed in the introduction and elsewhere in this paper.

The polynomials \(Q_{g,j}(\nu)\) and \(S_{g,j}(\nu)\) in
Theorems~\ref{thm 2legged combinatorics} and~\ref{thm combinatorics} are the
principal ingredients in the fixed-\((g,j)\) formulae: once the
factor \(c_\nu^j\), or the \(j\)-th power of the Catalan number \(C_\nu^j\), is factored out, these polynomials
contain the remaining enumerative information.  To preserve the flow of the
text, we state the resulting fixed-\((g,j)\) enumeration formulae here and
provide the lengthy lists of the \textit{explicit} polynomials \(Q_{g,j}(\nu)\) and
\(S_{g,j}(\nu)\) themselves in Section~\ref{subsec:QS-polynomial-tables}.

\begin{theorem}
\label{thm 2legged combinatorics}
Let \(\mathcal N_g(2\nu,j)\) denote the number of connected two-legged
\(2\nu\)-valent labeled graphs with \(j\) vertices that can be embedded on a
compact Riemann surface of minimal genus \(g\).  For
\(0\leq g\leq 5\) and \(1\leq j\leq 3\), one has
\begin{equation}\label{eq:Qgj-form}
    \mathcal N_g(2\nu,j)
    =
    c_\nu^j Q_{g,j}(\nu),
\end{equation}
where \(Q_{g,j}(\nu)\) is an explicit polynomial in \(\nu\) of degree $3g+j-1$.  The polynomials
\(Q_{g,j}(\nu)\) for \(0\leq g\leq 5\) and \(1\leq j\leq 3\) are explicitly listed in
Section~\ref{subsec:Qgj-polynomials}.
\end{theorem}

Theorem~\ref{thm 2legged combinatorics}, together with
\eqref{eq:Ncal-ELT-hypergeom} and the hypergeometric simplification carried out
in Appendix~\ref{sect: dfi form}, yields Theorem~\ref{thm:main-Ncal}, our main
result concerning \(\mathcal N_g(2\nu,j)\).  The large-valence asymptotics then
follow from Theorem~\ref{thm:main-Ncal}. Figure \ref{fig2} is an illustration of Theorems \ref{thm:main-Ncal} and \ref{thm 2legged combinatorics} for the choices $(\nu,g,j) \in \{(3,0,1),(3,1,1)\}$.

\begin{figure}[!htp]
	\centering
	\begin{minipage}{0.32\textwidth}
		\centering
\includegraphics[scale=0.3]{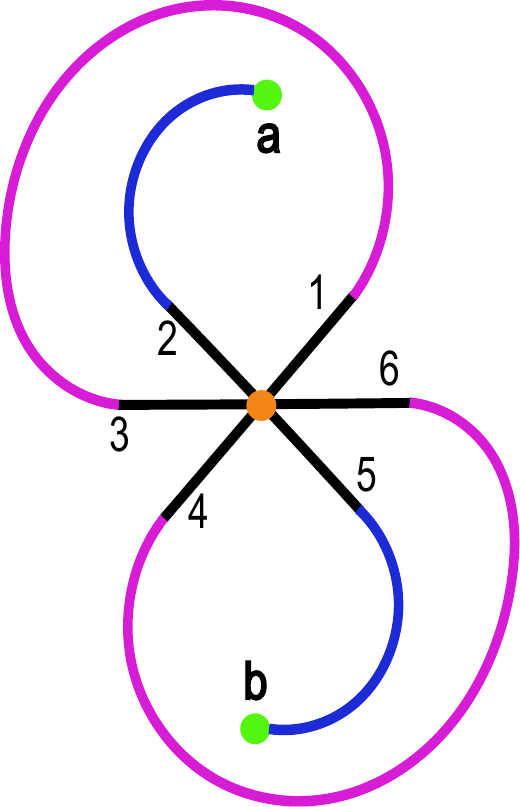}
     \subcaption{}   
	\end{minipage}\hfil 
	\begin{minipage}{0.32\textwidth}
		\centering
		\includegraphics[scale=0.34]{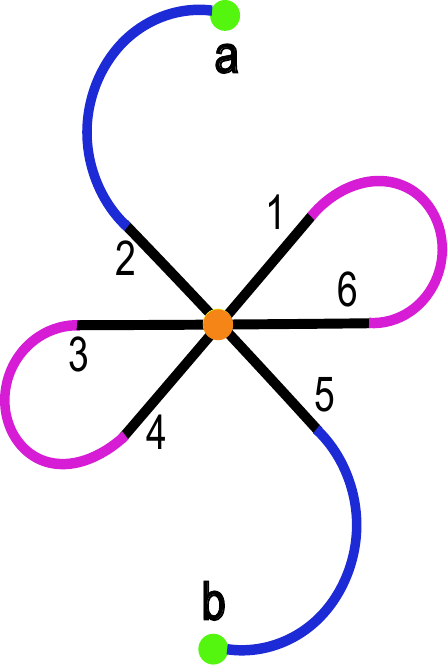}
    \subcaption{}      
	\end{minipage} \hfil
	\begin{minipage}{0.32\textwidth}
		\centering
		
        \includegraphics[width=\textwidth, alt={Three subfigures showing 2-legged graphs with one 6-valent vertex: (a) choice 3↔1, 4↔6 on the sphere; (b) choice 3↔4, 1↔6 on the sphere; (c) choice 3↔6, 4↔1, not on the sphere but on the torus.}]{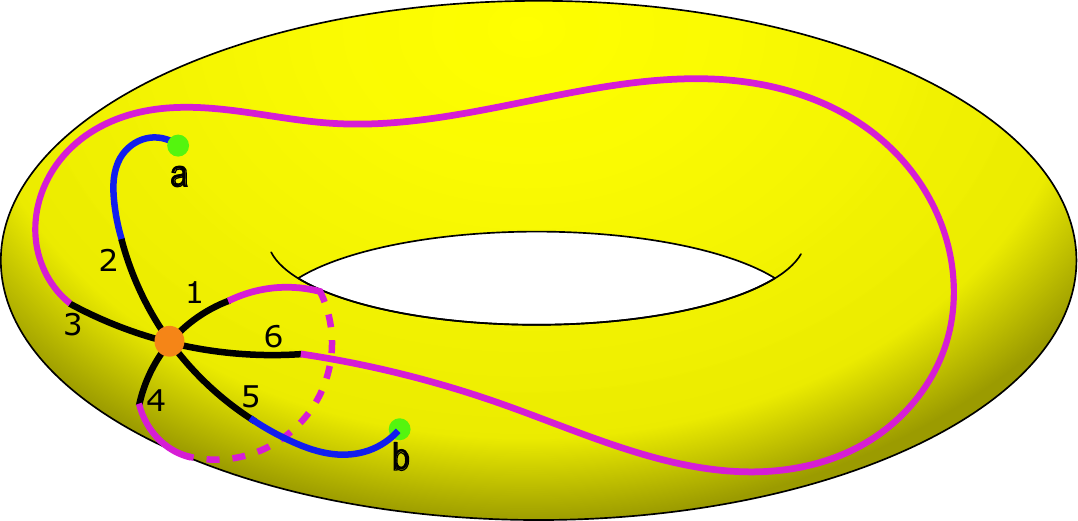}
    \subcaption{} 
	\end{minipage} \hfil
	\caption{The 2-legged graphs with one 6-valent vertex with the choice $a \leftrightarrow 2$ and $b \leftrightarrow 5$. This leaves three destinations for the edge labeled by $3$:  (a) an illustration of the choice $3 \leftrightarrow 1$ and $4 \leftrightarrow 6$ which can be embedded on the sphere,  (b) an illustration of the choice $3 \leftrightarrow 4$ and $1 \leftrightarrow 6$ which can also be embedded on the sphere, and (c) an illustration of the choice $3 \leftrightarrow 6$ and $4 \leftrightarrow 1$ which cannot be embedded on the sphere, but can be embedded on the torus. Given that our initial choice $a \leftrightarrow 2$ and $b \leftrightarrow 5$ is one of the $30$ possible choices, the illustration (c) explains why $\mathcal{N}_1(6,1) =30$ and the two choices illustrated in (a) and (b) explain why $\mathcal{N}_0(6,1)=2\times 30 =60$ as claimed in Theorem \ref{thm hexic 2-legged counts}. With regards to Theorem \ref{thm 2legged combinatorics}, notice that $c_3Q_{1,1}(3)=30$ and $c_3Q_{0,1}(3)=60$.}
	\label{fig2}
\end{figure}

\vspace{.2cm}

The next theorem is the analogous result for regular \(2\nu\)-valent graphs.

\begin{theorem}
\label{thm combinatorics}
Let \(\mathscr N_g(2\nu,j)\) denote the number of connected
\(2\nu\)-valent labeled graphs with \(j\) vertices that can be embedded on a
compact Riemann surface of minimal genus \(g\).  For
\(0\leq g\leq 5\) and \(1\leq j\leq 3\), one has
\begin{equation}\label{eq:Sgj-form}
    \mathscr N_g(2\nu,j)
    =
    C_\nu^j S_{g,j}(\nu),
\end{equation}
where \(S_{g,j}(\nu)\) is an explicit polynomial in \(\nu\) of degree $3(g+j-1)$.  The polynomials
\(S_{g,j}(\nu)\) for \(0\leq g\leq 5\) and \(1\leq j\leq 3\) are explicitly listed in
Section~\ref{subsec:Sgj-polynomials}.
\end{theorem}

Theorem~\ref{thm combinatorics}, together with
\eqref{eq:N-ELT-hypergeom} and the hypergeometric simplification carried out in
Appendix~\ref{sect: dfi form}, yields Theorem~\ref{thm:main-N}, our main result
concerning \(\mathscr N_g(2\nu,j)\).  The large-valence asymptotics then follow
from Theorem~\ref{thm:main-N}.

\begin{remark}\normalfont
For the convenience of the reader, in Appendix \ref{Appendix Number of graphs}
we provide combinatorial interpretations for the formulae in
Theorem \ref{thm combinatorics} when
\[
    (\nu,g,j)\in\{(2,0,1),(2,0,2),(2,1,1),(2,1,2)\}.
\]
\end{remark}

\begin{remark}\label{conjecture remark}\normalfont
We have observed interesting features about the polynomials
\(Q_{g,j}\) and \(S_{g,j}\) which we outline below. For each
\(g\in\mathbb N_0\) and \(j\in\mathbb N\), we conjecture that:
\begin{itemize}
\item
    \(\mathcal{N}_{g}(2\nu,j)=c_\nu^j Q_{g,j}(\nu)\), where
    \(Q_{g,j}(\nu)\) is a polynomial in \(\nu\) of degree \(3g+j-1\)\footnote{This statement is in fact no longer a conjecture as it was recently proved in \cite{ELBA2026}.}. Note that \(Q_{g,j}(\nu)\) is a polynomial in \(\nu\) of degree \textit{at most} \(3g+j-1\) is implied by Conjecture \ref{conj:a-polynomial}. Moreover, we conjecture that \(Q_{g,j}(\nu)\) has
    simple roots with non-negative real parts.
   \item
    \(\mathscr{N}_{g}(2\nu,j)=C_\nu^j S_{g,j}(\nu)\), where
    \(S_{g,j}(\nu)\) is a polynomial in \(\nu\) of degree \(3(g+j-1)\). We expect that the method used in \cite{ELBA2026} will also be able to prove this conjecture. Note that \(S_{g,j}(\nu)\) is a polynomial in \(\nu\) of degree \textit{at most} \(3(g+j-1)\) is implied by Conjecture \ref{conj:b-polynomial}. Moreover, we conjecture that
    besides repeated roots of order \(j\) at \(\nu=0\) and \(\nu=-1\), all
    other roots of \(S_{g,j}(\nu)\) are simple with positive real parts.
\end{itemize}
\end{remark}

We have also observed numerically that, for fixed \(g\), as \(j\) increases,
the zeros of \(Q_{g,j}(\nu)\) and \(S_{g,j}(\nu)\) tend to move toward the line
\(\Re \nu=0\), even as their number increases. In particular, the maximum real
part of the roots of \(S_{g,j_1}\) (resp. \(Q_{g,j_1}\)) is strictly smaller
than that of \(S_{g,j_2}\) (resp. \(Q_{g,j_2}\)) whenever \(j_1>j_2\). We have
also observed that, for fixed \(g\), the number of complex-conjugate roots
increases with \(j\). Investigating the structure of these roots and
understanding the behavior of the associated polynomials remains an intriguing
direction for future research.

\begin{remark}\normalfont
We note that by using the same arguments as in
Sections \ref{determining rn section} and \ref{determining free section}, which
were used to generate the explicit formulae in
Theorems \ref{thm 2legged combinatorics} and \ref{thm combinatorics}, one can
readily extend these theorems to larger values of \(g\) and \(j\), with
additional computational effort.
\end{remark}

\subsection{Outline of the Proof of Theorems \ref{thm:main-N}, \ref{thm:main-Ncal}, \ref{thm 2legged combinatorics}
and \ref{thm combinatorics}}\label{Sec Roadmap}
A brief overview of the approach leading to our main results is as follows. We start with the unitary ensemble $\mathscr{H}_n$ of $n\times n$ Hermitian random matrices with the distribution
\begin{equation} 
	\dd \mu_{nN}(M;u,\nu) = \frac{1}{\Tilde{\mathcal{Z}}_{nN}(u,\nu)} \exp\Big(-N \mathrm{Tr}\left(\frac{M^2}{2}+u \frac{M^{2\nu}}{2\nu}\right)\Big) \dd M.
	\end{equation}
 These random matrices are in turn connected to orthogonal polynomials on the real line with orthogonality weight $\exp(-N(\frac{z^2}{2}+u \frac{z^{2\nu}}{2\nu}))$.   We establish our first set of results by combining two key identities (which have previously appeared separately in the literature) concerning the recurrence coefficients, \( \mathcal{R}_n \), of these orthogonal polynomials. First, through a change of variables in Section~\ref{variable change section}, we derive the differential-difference equation for \( \mathcal{R}_n \):
\begin{equation}\label{eq:volterra_lattice}
    \frac{\ddd \mathcal{R}_n}{\ddd u} = \frac{\mathcal{R}_n}{2\nu u} \left( N(\mathcal{R}_{n+1} - \mathcal{R}_{n-1}) - 2 \right),
\end{equation}
sometimes referred to as the Volterra lattice equation \cite{WaltersBook}. Equation \eqref{eq:volterra_lattice} was also derived as a special case of the more general ``edge equations'' introduced in \cite[Section~3.3]{Watersthesis}, which apply to all pure-valence cases, both even and odd; see also \cite[Equations~(D.3) and~(D.4)]{EW2022}. The second identity concerns the topological expansion \begin{equation}\label{Asymp ga_n^2 1}
\mathcal{R}_n(x;u) \sim \sum^{\infty}_{g=0} \frac{r_{2g}(x;u)}{N^{2g}},
		\end{equation}
        of the recurrence coefficients, where $x$ denotes the \textit{'t Hooft parameter} $n/N$. In particular, it turns out that the Taylor coefficients of $r_{2g}(x;u)$, when expanded in $u$ near $u=0$,
        depend monomialy on $x$ \cite{Watersthesis} \footnote{In a correspondence with N.Ercolani after posting the first preprint of this article, we learned that this property was originally proved in~\cite{Watersthesis}. We provide our alternative proof in Section~\ref{determining rn section}.}. Using these two ingredients, we prove  Theorems~\ref{thm 2legged combinatorics} and~\ref{thm combinatorics}, in which we obtain explicit formulae in the variable \( \nu \) describing \( \mathscr{N}_g(2\nu,j) \) and \( \mathcal{N}_g(2\nu,j) \) for \( 0 \leq g \leq 5 \) and finitely many fixed values of $j$. Our results reveal several intriguing structural patterns satisfied by the polynomials in $\nu$ describing \( \mathscr{N}_g(2\nu,j) \) and \( \mathcal{N}_g(2\nu,j) \), which lead to Remark~\ref{conjecture remark}, where we formulate conjectures and suggest possible directions for future research.

        \begin{tcolorbox}[projectstyle={Custom}, colback=yellow!6, colframe=orange!70!black]
For each fixed \(g\), the structural formulae
\eqref{eq:N-ELT-hypergeom}--\eqref{eq:Ncal-ELT-hypergeom}
reduce the general problem of finding explicit formulae for
\(\mathscr{N}_g(2\nu,j)\) and \(\mathcal{N}_g(2\nu,j)\), valid for \textit{all}
\(j,\nu\in\mathbb{N}\), to the determination of
\(3g-2\) and \(3g\) \(\nu\)-dependent expressions, respectively. These are
precisely the expressions determined in
Theorems~\ref{thm 2legged combinatorics} and~\ref{thm combinatorics}.
Consequently, these two theorems, together with
\eqref{eq:N-ELT-hypergeom}--\eqref{eq:Ncal-ELT-hypergeom}, yield fully
explicit formulae for \(\mathscr{N}_g(2\nu,j)\) and
\(\mathcal{N}_g(2\nu,j)\) valid for all
\(j,\nu\in\mathbb{N}\).
 \end{tcolorbox}

Let us illustrate this approach: fix $g \geq 2$ and consider ${\mathscr N}_{g}(2\nu,j)$. Since the coefficients $b^{(g,\nu)}_{\ell}$ are independent of $j$, one obtains a system of $3g-2$ linear equations for $\{b^{(g,\nu)}_{\ell}\}_{\ell=0}^{3g-3}$ provided that $3g-2$ formulae are known for
\[
{\mathscr N}_{g}(2\nu,j_0), \qquad j_0 = 1,2,\dots,3g-2,
\]
in the variable $\nu$. Similarly, for fixed $g \geq 1$, a linear system for $\{a^{(g,\nu)}_{\ell}\}_{\ell=0}^{3g-1}$ arises once the $3g$ formulae in $\nu$ are known for
\[
{\mathcal N}_{g}(2\nu,j_0), \qquad j_0 = 1,2,\dots,3g.
\]

\medskip
Solving these linear systems allows one to readily determine $\{b^{(g,\nu)}_{\ell}\}_{\ell=0}^{3g-3}$ and $\{a^{(g,\nu)}_{\ell}\}_{\ell=0}^{3g-1}$. As an example, we now provide the details
for the determination of
\[
    b_\ell^{(2,\nu)},
    \qquad
    \ell=0,\ldots,3.
\]
For \(g=2\), equation~\eqref{eq:N-ELT-hypergeom} becomes

\begin{equation}\label{non 2 general form g=2}
{\mathscr N}_{2}(2\nu,j) = j!\, c_\nu^j (\nu - 1)^j  \sum_{\ell = 0}^{3} \left(b^{(2,\nu)}_{\ell}
{ \ell+j  \choose j}
\ \pFq{2}{1}{-j,1-\nu j}{-(\ell+j)}{\frac{1}{1-\nu}}\right).
\end{equation} Setting $j=1,\ldots,4$ in
\eqref{non 2 general form g=2} yields a system of four equations for the
four \(j\)-independent unknowns
\(
    b_\ell^{(2,\nu)},
    \ell=0,\ldots,3.
\) By the method developed in
Section~\ref{determining free section}, the left-hand sides of these four
equations are obtained explicitly. More precisely, $\mathscr{N}_{2}(2\nu,j) = C_\nu^jS_{2,j}(\nu)$, where
\begin{eqnarray*}
S_{2,1}(\nu) &=& \frac{1}{1440}(5\nu-2)\prod_{i=-1}^3(\nu-i) ,\\
S_{2,2}(\nu) &=& \frac{1}{1440 }(\nu+1)^2\nu^2 (2\nu-3)(49\nu^2-43\nu+6)\prod_{i=1}^2(\nu-i),\\
S_{2,3}(\nu) &=& \frac{1}{480} (\nu+1)^3\nu^3(\nu-1) \bigg(539\nu^5-2356\nu^4+3677\nu^3-2460\nu^2  +660\nu-48\bigg),\\
S_{2,4}(\nu) &=& \frac{1}{360}(1+\nu)^4\nu^4(\nu - 1)\bigg(7148\nu^6 - 32946\nu^5 + 57857\nu^4 - 48477\nu^3  + 19778\nu^2 \\ &&\, - 3504\nu + 180\bigg).
\end{eqnarray*}
Solving the system of four equations given by \eqref{non 2 general form g=2} for $j=1,2,3,4$ we obtain
\begin{flalign*}
    b^{(2,\nu)}_0 &= -(\frac{\nu^3}{360} + \frac{71\nu^2}{2880} + \frac{\nu}{36} + \frac{1}{240}),&\\
b^{(2,\nu)}_1 &= \frac{\nu(31\nu^2 + 98\nu + 40)}{1440},&\\
b^{(2,\nu)}_2 &= -\frac{\nu^2(22\nu + 25)}{576},&\\
b^{(2,\nu)}_3 &= \frac{7\nu^3}{360}.&
\end{flalign*}
The same procedure can, in principle, be applied to determine
\(a_\ell^{(g,\nu)}\) and \(b_\ell^{(g,\nu)}\) for any fixed genus \(g\).

Outside the specific cases of low genus \( 0 \leq g \leq 2 \)\footnote{See \cite[Table 2]{ELTMe}}, the vast majority of previous attempts to determine \( \mathscr{N}_g(\mu,j) \) and \( \mathcal{N}_g(\mu,j) \) have focused on obtaining an expression for fixed $\mu=\mu_0$ and general $j$. Unfortunately, this approach does not allow for one to build a system of linear equations to solve for the undetermined coefficients in \eqref{eq:N-ELT-hypergeom}--\eqref{eq:Ncal-ELT-hypergeom}. Extending these fixed-$\mu$ results to arbitrary \( \mu \), in a way that simultaneously holds for general \( j \), is a difficult task and often requires an understanding of the Freud equations (see Remark~\ref{Freud remark}) as their order becomes large (see \cite{BGM,Erc-et-al-multiple-scale,Alfy}).

In the appendices we contrast the general-$\nu$ method of our main text with fixed-$\nu$ approaches of \cite{BGM} and \cite{BD}. We fix the valency to $2\nu = 6$ and derive closed-form expressions in $j$ for $\mathscr{N}_g(6,j)$ and $\mathcal{N}_g(6,j)$, when $g = 0, 1,$ and $2$. The case $g=2$ is, to our knowledge, new. This extends the work of \cite{BGM} on the 4-valent case, and that of \cite{BD} on the 3-valent case. Finding explicit results with the fixed-$\nu$ approach for $\mathscr{N}_g(6,j)$ and $\mathcal{N}_g(6,j)$ at higher genera $g \geq 3$ is possible but computationally more demanding\footnote{However, they are of limited mathematical interest in view of Theorems \ref{thm:main-Ncal} and \ref{thm:main-N} for $g=3$ and $g=4$.}. These results, though narrower in scope, serve to illustrate the comparative advantages of the general-$\nu$ framework.

Our main results close the problem of finding $\mathscr{N}_g(\mu,j)$ and $\mathcal{N}_g(\mu,j)$ for fixed $g$ and general $\mu=2\nu$ and $j$. We explicitly derive the counts for $g\leq4$ and our results can readily be extended to larger $g$ with increasing computational effort. The case of odd $\mu$ remains open, as do the more challenging problems of obtaining expressions for general $g$, $j$, and $\mu$, and mixed valence combinatorics (see Theorem \ref{thmEM02}).

\section{Differential Difference Equations}\label{variable change section} \normalfont
We begin our main arguments of this paper by proving Lemma \ref{sigma lemma} below. Let us begin by rescaling \eqref{fancy V def 2}. First, we make the transformation
\begin{equation}\label{uv relation}
    z = \sigma^{1/2}\zeta \quad \text{and} \quad u=\sigma^{-\nu }.
\end{equation} 
Recalling \eqref{fancy V def 2}, under this transformation we find that $\mathscr{V}(z) = V(\zeta)$, where
\begin{equation}\label{V def}
    V(\zeta) = \frac{\zeta^{2\nu}}{2\nu} + \sigma \frac{\zeta^2}{2}.
\end{equation}
We define the corresponding $\sigma$-partition function as
\begin{equation}\label{sigma partition function}
    Z_{nN}(\sigma) = \int_{-\infty}^\infty ... \int_{-\infty}^\infty \prod_{1\leq j<k\leq n} (\zeta_j-\zeta_k)^2\prod_{j=1}^n\exp\left[ -N\left(\frac{\zeta_j^{2\nu}}{2\nu}+\sigma\frac{\zeta_j^2}{2}\right)\right]\dd \zeta_1...\dd \zeta_n .
\end{equation} 
We define the $\sigma$-free energy as
\begin{equation}\label{sigma free energy}
    F_{nN}(\sigma) = \frac{1}{n^2} \ln\frac{Z_{nN}(\sigma)}{\mathcal{Z}_{nN}(0)}.
\end{equation}
Note that by equations \eqref{u free energy}, \eqref{free energy}, and \eqref{sigma free energy} we obtain the following relation between the free energies $\mathcal{F}_{nN}(u)$ and $F_{nN}(\sigma)$:
\begin{equation}\label{free relation}
    \mathcal{F}_{nN}(u) = \frac{\ln \sigma}{2}+F_{nN}(\sigma), \quad \sigma = u^{-1/\nu}.
\end{equation} 

We now introduce the class of monic polynomials $\{P_n(\zeta)\}_{n=0}^\infty$ which satisfy the orthogonality condition
\begin{equation}\label{h_n def}
    \int_{\mathbb{R}} P_n(\zeta)P_m(\zeta) e^{-NV(\zeta)} \dd \zeta  = h^{(\sigma)}_n\delta_{nm},
\end{equation} 
where $V(\zeta)$ is as defined in \eqref{V def}. As a consequence of their orthogonality condition these polynomials also satisfy the three term recurrence relation \cite{SzegoOP}
\begin{equation}\label{three term}
    \zeta P_n(\zeta) = P_{n+1}(\zeta) + R_nP_{n-1}(\zeta),
\end{equation}
where, 
\begin{equation}\label{R def}
    R_n = \frac{h^{(\sigma)}_n}{h^{(\sigma)}_{n-1}}.
\end{equation}
Using the definition of the recurrence coefficients and
\eqref{uv relation}, a direct computation gives
\begin{equation}\label{r relation}
    R_n = u^{\frac{1}{\nu}}\mathcal{R}_n,
\end{equation}
where $\mathcal{R}_n := \ga^2_n$ are the recurrence coefficients corresponding to polynomials orthogonal with respect to the weight $\mathscr{V}(z)$, see \eqref{OP rec}. We prove differential difference equations for $R_n$\footnote{Equation \eqref{dRndsigma} is sometimes referred to as the Volterra lattice equation \cite{Suris2003}.} and $F_{nN}$\footnote{Differential difference equations for $F_n$ are referred to as Toda equations in the literature \cite{BD, BGM}.} which are valid for all $\nu$.

\begin{lemma}\label{sigma lemma}
The recurrence coefficient $R_n$ and the free energy $F_{nN}$ satisfy the following differential difference equations independent of $\nu$,
\begin{equation}\label{dRndsigma}
    \frac{\ddd  R_n}{\ddd \sigma} = \frac{-N}{2} R_n \left( R_{n+1} - R_{n-1} \right),
\end{equation}
\begin{equation}\label{the second f}
    \frac{\ddd ^2F_{nN}}{\ddd \sigma^2}  = \frac{N^2}{4n^2}R_n\left( R_{n+1} + R_{n-1} \right).
\end{equation} 
\end{lemma}

\begin{proof}
We first derive equation \eqref{dRndsigma} which will in turn be used to prove equation \eqref{the second f}. 
Differentiating \eqref{h_n def} with respect to $\sigma$ and using the orthogonality of $P_n(\zeta )$ combined with \eqref{three term}, we find
\begin{eqnarray}
    \frac{\ddd  h^{(\sigma)}_n}{\ddd \sigma} &=& \frac{\ddd }{\ddd \sigma}\int_\Gamma P_n(\zeta )P_n(\zeta ) e^{-NV(\zeta )} d\zeta  \nonumber, \\
    &=& 2\int_\Gamma \left( \frac{\ddd }{\ddd \sigma} P_n(\zeta )\right) P_n(\zeta ) e^{-NV(\zeta )} d\zeta  + \int_\Gamma  P_n(\zeta ) P_n(\zeta ) \frac{\ddd }{\ddd \sigma} e^{-NV(\zeta )} d\zeta  ,\nonumber \\
    &=& 0 + \int_\Gamma  P_n(\zeta ) P_n(\zeta ) \left( \frac{-N\zeta ^2}{2}\right) e^{-NV(\zeta )} d\zeta  \nonumber, \\
    &=&\frac{-N}{2}\int_\Gamma  \left(P_{n+1}(\zeta )+R_n P_{n-1}(\zeta )\right)^2 e^{-NV(\zeta )} d\zeta \nonumber, \\
    &=&\frac{-N}{2} \left( h^{(\sigma)}_{n+1} + R_n^2h^{(\sigma)}_{n-1} \right)\nonumber, \\
    &=&\frac{-N}{2} h^{(\sigma)}_n \left( R_{n+1} + R_n \right).\label{dhn}
\end{eqnarray}
Thus,
\begin{eqnarray}
    \frac{\ddd  R_n}{\ddd \sigma} &=& \frac{\ddd }{\ddd \sigma} \left( \frac{h^{(\sigma)}_n}{h^{(\sigma)}_{n-1}} \right)\nonumber ,\\
    &=& \frac{\left( \frac{\ddd }{\ddd \sigma}h^{(\sigma)}_n \right)h^{(\sigma)}_{n-1} - h^{(\sigma)}_n \left(\frac{\ddd }{\ddd \sigma}h^{(\sigma)}_{n-1}\right)}{(h^{(\sigma)}_{n-1})^2}\nonumber, \\
    &=& \frac{-N}{2} \frac{ h^{(\sigma)}_n \left( R_{n+1} + R_n \right)h^{(\sigma)}_{n-1} - h^{(\sigma)}_n h^{(\sigma)}_{n-1} \left( R_{n} + R_{n-1} \right)}{(h^{(\sigma)}_{n-1})^2}\nonumber ,\\
    &=& \frac{-N}{2} R_n \left( R_{n+1} - R_{n-1} \right).\label{d Rn ds}
\end{eqnarray}
Hence, we have proved equation \eqref{dRndsigma}. By the Heine's identity for Hankel determinants we can re-write the free energy $F_{nN}(\sigma)$ as
\begin{equation}\label{F def}
   F_{nN}(\sigma) = \frac{\ln n!}{n^2} +\frac{1}{n^2}\sum^{n-1}_{k=0} \ln h^{(\sigma)}_k.
\end{equation}
As an immediate consequence of equation \eqref{dhn} we determine that
\begin{equation}\label{d ln hn}
    \frac{\ddd  \ln h^{(\sigma)}_n}{\ddd  \sigma} = \frac{-N}{2} \left( R_{n+1} + R_n \right).
\end{equation}
Thus, taking the second derivative of \eqref{F def} and applying equations \eqref{d ln hn} and \eqref{d Rn ds} we find that,
\begin{eqnarray*}
    \frac{\ddd ^2 F_{nN}(\sigma)}{\ddd  \sigma^2} &=& \frac{-N}{2n^2}\frac{\ddd }{\ddd \sigma} \left(\sum^{n-1}_{k=0}  \left(R_{k+1} + R_k \right) \right),\\
    &=& \frac{-N}{2n^2} \sum^{n-1}_{k=0} \left( \frac{\ddd }{\ddd \sigma}R_{k+1} + \frac{\ddd }{\ddd \sigma}R_k \right),\\
    &=& \frac{N^2}{4n^2} \sum^{n-1}_{k=0} \Big( R_{k+1}  \left( R_{k+2} - R_{k} \right) + R_k \left( R_{k+1} - R_{k-1} \right) \Big),\\
    &=& \frac{N^2}{4n^2} \sum^{n-1}_{k=0}  R_{k+1} R_{k+2} - R_kR_{k-1} , \\
        &=& \frac{N^2}{4n^2}\left( R_{n+1}R_n + R_nR_{n-1} + \sum^{n-3}_{k=0}  R_{k+1}R_{k+2} - \sum^{n-1}_{k=0} R_{k}R_{k-1} \right), \\
    &=& \frac{N^2}{4n^2}\left( R_{n+1}R_n + R_nR_{n-1} + \sum^{n-1}_{j=2}  R_{j-1}R_{j} - \sum^{n-1}_{k=0} R_{k}R_{k-1} \right), \\
    &=& \frac{N^2}{4n^2}R_n\left( R_{n+1} + R_{n-1} \right),
\end{eqnarray*}
where to arrive at the final equality we have used the fact that $R_0=0$ which follows from equation \eqref{three term}.
\end{proof}
We will use Lemma \ref{sigma lemma} to prove the main results of this paper.

\section{The Asymptotic Expansion of $\mathcal{R}_n$}\label{determining rn section}

 In this section we use equation \eqref{dRndsigma} to prove Theorem \ref{recurrence theorem}. Theorem \ref{recurrence theorem} then allows us to prove Theorem \ref{free energy theorem} in Section \ref{determining free section}. To begin, let us determine $\di \frac{\ddd \mathcal{R}_n}{\ddd u}$ in terms of $\mathcal{R}_{n+1}$ and $\mathcal{R}_{n-1}$ using equations \eqref{uv relation}, \eqref{r relation} and \eqref{dRndsigma},
\begin{eqnarray}
    \frac{\ddd\mathcal{R}_n}{\ddd u} &=& \frac{\ddd}{\ddd u}(u^{-\frac{1}{\nu}}R_n),\nonumber \\
    &=&-\frac{1}{u\nu}\mathcal{R}_n - \frac{u^{-\frac{2}{\nu}}}{u \nu}\frac{\ddd R_n}{\ddd \sigma},\nonumber \\
    &=&\frac{\mathcal{R}_n}{2\nu u}\left( N(\mathcal{R}_{n+1} - \mathcal{R}_{n-1}) -2 \right).\label{dRdu}
\end{eqnarray}
Note that we have just recovered \eqref{eq:volterra_lattice}. To prove Theorem \ref{recurrence theorem} we are going to need \eqref{dRdu} and some properties of the Freud equations (sometimes referred to as the string equations). The Freud equations (see e.g. \cite{BL, Maggy}) are given by 
\begin{equation}
    \ga_n[\mathscr{V}'(\mathcal{Q})]_{n,n-1}=\frac{n}{N},
\end{equation}
 where in the case of even potentials the infinite matrix $\mathcal{Q}$ is given by
\begin{equation}
    \mathcal{Q}= \begin{pmatrix}
    0 & \ga_1 & 0 & 0 & \cdots \\
    \ga_1 & 0 & \ga_2 & 0 & \cdots \\
    0 & \ga_2 & 0 & \ga_3 & \cdots \\
    \vdots & \ddots & \ddots & \ddots & \cdots\\
    \end{pmatrix}.
\end{equation}
It is straightforward to show that for the weight $\mathscr{V}(z) = \frac{z^2}{2} + u\frac{z^{2\nu}}{2\nu}$ the Freud equation can be written as
\begin{equation}\label{freud equation}
    \mathcal{R}_n + uF_\nu = x,
\end{equation}
where we refer to $F_\nu$ as the \textit{Freud function} and $x=\frac{n}{N}$. Note that the $\mathcal{R}_n$ term on the LHS of \eqref{freud equation} arises from the $\frac{z^2}{2}$ component of the weight $\mathscr{V}(z)$ and $F_\nu$ arises from $\frac{z^{2\nu}}{2\nu}$.

\begin{remark}\label{Freud remark}
        The first few Freud functions are:
    \begin{eqnarray*}
        \nu =1 &:& F_1 = \mathcal{R}_n,\\
        \nu =2 &:& F_2 = \mathcal{R}_n(\mathcal{R}_{n+1} +\mathcal{R}_{n}+\mathcal{R}_{n-1}),\\
        \nu =3 &:& F_3 = \mathcal{R}_n(\mathcal{R}_{n+2}\mathcal{R}_{n+1} + \mathcal{R}_{n+1}^2 + 2\mathcal{R}_n\mathcal{R}_{n+1}
      +\mathcal{R}_n^2 + 2\mathcal{R}_n\mathcal{R}_{n-1} \\ &&  \qquad+  \mathcal{R}_{n+1}\mathcal{R}_{n-1} + \mathcal{R}_{n-1}^2 + \mathcal{R}_{n-1}\mathcal{R}_{n-2}).
    \end{eqnarray*}
\end{remark}
\noindent As part of our work studying hexic weights, we provide a direct derivation of equation \eqref{freud equation} for $\nu=3$ in Appendix \ref{Hexic section}. We now recall a few well known facts about $F_\nu$. 

\begin{lemma}\label{Freudian}
    The Freud functions $F_{\nu}$ for weights of the form $e^{-N(\frac{z^{2}}{2}+u\frac{z^{2\nu}}{2\nu})}$ have the following properties:
    \begin{enumerate}
        \item There are ${2\nu-1 \choose \nu}$ number of terms in $F_{\nu}$, which are not necessarily distinct.
        \item Each term is the product of $\nu$ recurrence coefficients from the set $\{\mathcal{R}_{n+\ell}: -\nu+1 \leq \ell \leq \nu-1 \}$.
    \end{enumerate}

\end{lemma}

\noindent Lemma \ref{Freudian} can be seen as a consequence of the work \cite{Maggy}. However, we also include a short proof in Appendix \ref{Freud section} for completeness.

Let us recall the asymptotic expansion \eqref{Asymp ga_n^2} 
\begin{equation}\label{R_n expansion}
    \mathcal{R}_{n}(x;u) = \sum_{g=0}^\infty \frac{r_{2g}(x;u)}{N^{2g}},
\end{equation} 
where $r_{2g}(x;u)$ can also be written as a power series, this time in terms of $u$. Furthermore, 
evaluation of the Taylor expansion of $r_j$, centered at $x = n/N$, at $x \pm k/N$ yields
\begin{equation}\label{R_np1 expansion}
    \mathcal{R}_{n\pm k}(x;u) \sim \sum_{m=0}^\infty \frac{1}{N^{2m}}\sum_{l=0}^\infty \frac{(\pm k)^lr_{2m}^{(l)}(x;u)}{l!N^{l}},\quad \text{as} \quad N\to\infty,
\end{equation} 
where the derivatives of $r_j$ are taken with respect to $x$. 

\begin{remark}\label{string remark} \normalfont
    Using Lemma \ref{Freudian} and \eqref{R_np1 expansion} we can deduce that the $N^{0}$ order of the Freud equation for the recurrence coefficients of polynomials with weight $e^{-N\mathscr{V}(z)}$ is given by
        \begin{eqnarray}
            r_0 + u{2\nu-1 \choose \nu}\left(r_0\right)^\nu &=& x, \label{Order zero String} \end{eqnarray}
        which is equivalent to \eqref{r_0 string} proven in \cite{ErcolaniCaustics} up to a simple change of variables: $r_0 \mapsto xr_0$.
\end{remark}
\begin{theorem}\label{theorem r structure} It holds that
\begin{equation}\label{r_2g series}
    r_{2g}(x;u) = \sum_{j=0}^\infty \beta_{2g,j}(x)u^j,
\end{equation}
where $\beta_{2g,j}(x) = c_{2g,j}x^{\mathscr{D}}$ and $\mathscr{D} = j(\nu-1)+1-2g$.  If $\mathscr{D}<0$ then $\beta_{2g,j}(x)=c_{2g,j} = 0$. Note that for $\mathscr{D}\geq 0$ one may still find the trivial solution $\beta_{2g,j}(x)=0$.
\end{theorem}
\begin{proof}
We will prove Theorem \ref{theorem r structure} by induction. First, as shown in Appendix \ref{0 expansion} we find that
\begin{equation}\label{induction step 1}
    \beta_{0,j}(x) = c_{0,j}x^{j(\nu-1)+1},
\end{equation}
where 
\[ c_{0,j} = \left(-{2\nu-1 \choose \nu}\right)^j\frac{(j\nu)!}{j!(j(\nu-1)+1)!} .\]
Thus, Theorem \ref{theorem r structure} holds for all $j\in \N_0$ when $g=0$. Furthermore, comparing the $N^{-2g}$ coefficients in equation \eqref{freud equation} it readily follows that $\beta_{2g,0}=0$ for all $g>0$. Thus, Theorem \ref{theorem r structure} also holds for all $g\in \N_0$ when $j=0$. These two identities constitute our base case for the inductive argument.

Assume Theorem \ref{theorem r structure} holds true for all $j\leq J$ when $g<G$ and for all $j<J$ when $g=G$. We will prove that  that $\beta_{2G,J}(x) = c_{2G,J}x^{J(\nu-1)+1-2G}$. 

Let us recall the Freud equation \begin{equation}\label{freud equation1}
    \mathcal{R}_n   = x - uF_\nu.
\end{equation} The first statement of Lemma \ref{Freudian} suggests expressing the Freud function $F_{\nu}$ as 
\begin{equation}
    F_{\nu} \equiv \sum^{M_{\nu}}_{m=1} F_{\nu,m}, \qquad M_{\nu}:= {2\nu-1 \choose \nu},
\end{equation}
where by the second statement of Lemma \ref{Freudian} we have
\begin{equation}
    F_{\nu,m} = \prod_{s\in I_{\nu,m}}\mathcal{R}_{n+s},
\end{equation}
with the index set $I_{\nu,m} \subset I_{\nu} := \{-\nu+1, -\nu+2, \ldots, \nu-2, \nu-1\}$ and $|I_{\nu,m}|=\nu$. The members of $I_{\nu,m}$ may not be necessarily distinct.

By equation \eqref{freud equation1}, in order to find an expression for $\be_{2G,J}$ we need to find the $u^{J-1}$ Taylor coefficient of the $N^{-2G}$ coefficient in the large $N$ asymptotic expansion of $F_{\nu}$. To this end, fix $m \in \{1,\ldots,M\}$, set $I_{\nu,m} = \{a_1,\ldots,a_{\nu}\}$, and choose the vectors of indices $(j_1,\ldots,j_{\nu})^T \in \N^{\nu}_0$ and $(k_1,\ldots,k_{\nu})^T \in \N^{\nu}_0$ with the property that \begin{equation}\label{k's and j's}
    k_1 + \cdots + k_{\nu}=2G, \qquad \mbox{and} \qquad j_1 + \cdots + j_{\nu}= J-1.
\end{equation}
For each $p \in \{1,\ldots,\nu\}$, we define $\Xi_{m}(x;k_p,j_p)$ as being the $u^{j_p}$ Taylor coefficient of the $N^{-k_{p}}$ coefficient in the large $N$ asymptotic expansion of $\mathcal{R}_{n+a_p}$. Then $\Upsilon_{m}(x;k_1,\ldots,k_p, j_1, \ldots, j_p):=\prod^{\nu}_{p=1} \Xi_{m}(x;k_p,j_p)$ is the contribution of the particular choice $(j_1,\ldots,j_{\nu})^T \in \N^{\nu}_0$ and $(k_1,\ldots,k_{\nu})^T \in \N^{\nu}_0$ to the desired the $u^{J-1}$ Taylor coefficient of the $N^{-2G}$ coefficient in the large $N$ asymptotic expansion of $F_{\nu,m}$.  Recalling \eqref{R_np1 expansion} we have
\begin{equation}\label{R_np1 expansion1}
    \mathcal{R}_{n+ a_p}(x;u) \sim \sum_{m=0}^\infty \frac{1}{N^{2m}}\sum_{l=0}^\infty \frac{(a_p)^lr_{2m}^{(l)}(x;u)}{l!N^{l}},\quad \text{as} \quad N\to\infty,
\end{equation} where we recall that the derivatives in the inner summation are with respect to $x$. The coefficient of $N^{-k_p}$ in the asymptotic expansion of \eqref{R_np1 expansion1} is
\begin{equation}\label{sum1}
\sum_{\substack{m , \ell \in \N_0 \\ 2m+\ell=k_p} } \frac{(a_p)^l}{l!}r_{2m}^{(l)}(x;u), \qquad \mbox{and thus} \qquad \Xi_{m}(x;k_p,j_p)=\sum_{\substack{m , \ell \in \N_0 \\ 2m+\ell=k_p} } \frac{(a_p)^l}{l!}\be_{2m,j_p}^{(l)}(x).
\end{equation}
Using the induction hypothesis, for each $m , \ell \in \N_0$ with $2m+\ell=k_p$ we can write 
\begin{equation}
\be_{2m,j_p}^{(l)}(x)=\tilde{c}_{2m,j_p} x^{j_p(\nu-1)+1-2m-\ell} = \tilde{c}_{2m,j_p} x^{j_p(\nu-1)+1-k_p}.
\end{equation}
So $\Xi_{m}(x;k_p,j_p)$ given by \eqref{sum1} must be of the same form as well. Therefore 
\begin{equation} \begin{split}
       \Upsilon_{m}(x;k_1,\ldots,k_p, j_1, \ldots, j_p) & = \prod^{\nu}_{p=1} \Xi_{m}(x;k_p,j_p) \\ & =  A_{m}(k_1,\ldots,k_p, j_1, \ldots, j_p) x^{(J-1)(\nu-1)+\nu-2G},
\end{split}
\end{equation}
for some constant $A_{m}(k_1,\ldots,k_p, j_1, \ldots, j_p)$, where we have used \eqref{k's and j's}. Thus, the $u^{J-1}$ Taylor coefficient of the $N^{-2G}$ coefficient in the large $N$ asymptotic expansion of $F_{\nu}$ is \begin{equation} \begin{split}
     \sum^{M_{\nu}}_{m=1}  \sum_{\substack{k_1, k_2, \ldots, k_{\nu} \in \N_0 \\ k_1 + k_2 + \ldots + k_{\nu} = 2G} } \sum_{\substack{j_1, j_2, \ldots, j_{\nu} \in \N_0 \\ j_1 + j_2 + \ldots + j_{\nu} = J-1} } & \Upsilon_{m}(x;k_1,\ldots,k_p, j_1, \ldots, j_p) = A x^{(J-1)(\nu-1)+\nu-2G}.
\end{split}
\end{equation}
Now, by applying \eqref{freud equation1} we obtain the desired result
\[\beta_{2G,J}(x) = c_{2G,J}x^{J(\nu-1)+1-2G},\] which holds for an arbitrary choice of $(G,J)\in \N_0 \times \N_0$.

Since our inductive argument is on a two-dimensional lattice, some care is needed to complete our inductive argument. We will justify why
our inductive reasoning described above can be used to fill out the finite set of points $(j,G)$  with $j<J$ and the points $(j,g)$ with $g<G$, $j\leq J$ from our base case. Recalling our base case we can immediately apply our inductive step to conclude that $\beta_{2,1}$ satisfies Theorem \ref{theorem r structure}. This will then imply that $\beta_{2,2}$ satisfies Theorem \ref{theorem r structure}. We then repeatedly apply our induction step until we arrive at $\beta_{2,J}$. Furthermore, given our base case and the fact that Theorem \ref{theorem r structure} holds for $\beta_{2,1}$, we immediately apply our inductive step to conclude that $\beta_{4,1}$ also satisfies Theorem \ref{theorem r structure}. This will then imply that $\beta_{4,2}$ satisfies Theorem \ref{theorem r structure}. We then repeatedly apply our induction step until we arrive at $\beta_{4,J}$. There are finitely many iterations of this process until we reach $g=G$ and $j=J$ as desired.

\end{proof}

\begin{remark} \normalfont As an alternative attempt to prove that
\begin{equation}\label{be 2g j}
\be_{2g,j}(x)=c_{2g,j}x^{j(\nu-1)+1-2g},    
\end{equation}
 one can try to directly derive from \eqref{dRdu} the inhomogeneous differential equation satisfied by $\be_{2g,j}$:
\begin{equation}\label{eq 4132}
    \nu J \beta_{2G,J} - x\frac{\dd \beta_{2G,J}}{\dd x} = \lambda_{G,J}x^{J(\nu-1)+1-2G},
\end{equation} where $\lambda_{G,J}$ is a constant. This differential equation provides a  convenient way to compute $\be_{2g,j}$'s in a recursive way and this is what we use to derive all the formulae in Theorem \ref{thm 2legged combinatorics}. However this differential equation in itself does not prove \eqref{be 2g j}, as it suggests that  $$\be_{2g,j}(x)=c_{2g,j}x^{j(\nu-1)+1-2g}+Ax^{\nu J},$$ where $x^{\nu J}$ is the homogeneous solution of the differential equation. The proof of Theorem \ref{theorem r structure} shows that $A=0$.

Here, for completeness, we provide the details of deriving the differential equation \eqref{eq 4132}. Rearranging Equation \eqref{dRdu} we find,
\begin{equation}\label{yolo}
    2\left( \nu u \frac{\ddd \mathcal{R}_n}{\ddd u} + \mathcal{R}_n \right)= N\mathcal{R}_n\left(\mathcal{R}_{n+1} - \mathcal{R}_{n-1}\right).
\end{equation}
Next we equate the $N^{-2G}$ coefficient in Equation \eqref{yolo}. Substituting Equation \eqref{R_np1 expansion} into Equation \eqref{yolo} we find,
\begin{equation}
     \nu u \frac{\ddd r_{2G}(x;u)}{\ddd u} + r_{2G}(x;u)  = \left( \sum_{g=0}^{G} r_{2g}(x;u) \left(  \sum_{h=0}^{G-g}\frac{r^{(2l+1)}_{2h}(x;u)}{(2l+1)!} \right)  \right),
\end{equation}
where $l = G-g-h$. After some re-arranging of terms we are left with,
\begin{multline}\label{induction eq}
     \nu u \frac{\ddd r_{2G}(x;u)}{\ddd u} + r_{2G}(x;u)  - r_0(x;u)\frac{\ddd r_{2G}(x;u)}{\ddd x} - r_{2G}(x;u)\frac{\ddd r_{0}(x;u)}{\ddd x} \\
    =  \sum_{g=1}^{G-1} r_{2g}(x;u)  \left(  \sum_{h=0}^{G-g} \frac{r^{(2l+1)}_{2h}(x;u)}{(2l+1)!} \right) + r_{0}(x;u)  \left(  \sum_{h=0}^{G-1} \frac{r^{(2(G-h)+1)}_{2h}(x;u)}{(2(G-h)+1)!} \right),
\end{multline}
Note that the RHS now only contains the term $r_{2k}(x;u)$, where $k<G$. Thus, all terms that contribute to the $u^J$ power of the RHS satisfy our induction assumption. Notice that  the second term on the RHS of \eqref{induction eq} has the following coefficient of $u^J$
\begin{equation}\label{coeff u^J 2nd}
\sum^{G-1}_{h=0} \sum^J_{k=0} \be_{0,J-k}(x) \frac{\dd^{2(G-h)+1}}{\dd x^{2(G-h)+1}} \be_{2h,k}(x).    
\end{equation} 
In what follows we use the notation
\[ f(x) \overset{\sim}{=} g(x) \] to denote the equation $f(x)=cg(x)$ for some constant $c$ (which may or may not be zero).
For a fixed $0 \leq h \leq G-1$ and $0 \leq k \leq J$, from the induction hypothesis we have 
\[ \be_{0,J-k}(x) \overset{\sim}{=} x^{(J-k)(\nu-1)+1},  \] and
\begin{equation}
\frac{\dd^{2(G-h)+1}}{\dd x^{2(G-h)+1}} \be_{2h,k}(x) \overset{\sim}{=} \begin{cases}
  x^{k(\nu-1)-2G}, & k(\nu-1)-2G \geq 0, \\
 0, & k(\nu-1)-2G < 0.
\end{cases}
\end{equation}
So we have
\begin{equation}\label{coeff u^J 2nd'}
    \be_{0,J-k}(x) \frac{\dd^{2(G-h)+1}}{\dd x^{2(G-h)+1}} \be_{2h,k}(x) \overset{\sim}{=}  \begin{cases}
   x^{J(\nu-1)-2G+1}, & J(\nu-1)-2G \geq 0, \\
 0, & J(\nu-1)-2G < 0.
\end{cases}
\end{equation}

Now we focus on the first term on the RHS of \eqref{induction eq} which has the following coefficient of $u^J$
\begin{equation}\label{coeff u^J 1st} \sum^{G-1}_{g=1}
\sum^{G-g}_{h=0} \sum^J_{k=0} \be_{2g,J-k}(x) \frac{\dd^{2\ell+1}}{\dd x^{2\ell+1}} \be_{2h,k}(x), \qquad \ell = G-g-h.    
\end{equation} For a fixed $1\leq g \leq G-1,$ $0\leq h \leq G-g,$  and $0\leq k \leq J,$ we have 
\begin{equation}
    \be_{2g,J-k}(x) \overset{\sim}{=} \begin{cases}
 x^{(J-k)(\nu-1)+1-2g}, & (J-k)(\nu-1)+1-2g \geq 0, \\
 0, & (J-k)(\nu-1)+1-2g < 0,
\end{cases}
\end{equation}
and
\begin{equation}
\frac{\dd^{2\ell+1}}{\dd x^{2\ell+1}} \be_{2h,k}(x) \overset{\sim}{=} \begin{cases}
 x^{k(\nu-1)-2(G-g)}, & k(\nu-1)-2(G-g) \geq 0, \\
 0, & k(\nu-1)-2(G-g) < 0,
\end{cases},
\end{equation}
 where again $\ell = G-g - h$. We get nonzero terms simultaneously in the last two expressions if $J(\nu-1)+1-2G \geq 0$. Therefore we have
\begin{equation}\label{coeff u^J 1st'}
    \be_{2g,J-k}(x)\frac{\dd^{2\ell+1}}{\dd x^{2\ell+1}} \be_{2h,k}(x) \overset{\sim}{=} \begin{cases}
 x^{J(\nu-1)+1-2G}, & J(\nu-1)+1-2G \geq 0, \\
 0, & J(\nu-1)+1-2G < 0,
\end{cases} 
\end{equation}
Combining  \eqref{coeff u^J 2nd}, \eqref{coeff u^J 2nd'}, \eqref{coeff u^J 1st}, and \eqref{coeff u^J 1st'} we conclude that the coefficient of $u^J$ on the RHS of \eqref{induction eq} is equal to 
\begin{equation}\label{u^J RHS}
 Cx^{J(\nu-1)+1-2G}   
\end{equation}
for some constant $C$, if $J(\nu-1)+1-2G \geq 0$ \footnote{Compare with the condition on $J(\nu-1)-2G$ in \eqref{coeff u^J 2nd'}.
}, and equals zero otherwise.

Now, we focus on the LHS of \eqref{induction eq}. The coefficient of $u^J$ from the terms $\nu u \frac{\ddd r_{2G}(x;u)}{\ddd u} + r_{2G}(x;u)$ can be easily seen to be equal to
\begin{equation}\label{u^J LHS1}
    (\nu J+1) \be_{2G,J}.
\end{equation}
The coefficient of $u^J$ from the term $- r_0(x;u)\frac{\ddd r_{2G}(x;u)}{\ddd x}$ is \begin{equation}\label{u^J LHS2}
    -x \frac{\dd}{\dd x} \be_{2G,J}(x) + A x^{J(\nu-1)+1-2G},
\end{equation}
for some constant $A$, where we have used the fact that $\be_{0,0}(x)=x$. Finally, the coefficient of $u^J$ from the term $ - r_{2G}(x;u)\frac{\partial r_{0}(x;u)}{\partial x}$ equals
\begin{equation}\label{u^J LHS3}
    -\be_{2G,J}+B x^{J(\nu-1)+1-2G},
\end{equation}
for some constant $B$, where again we have used the fact that $\be_{0,0}(x)=x$.

Combining \eqref{u^J RHS}, \eqref{u^J LHS1}, \eqref{u^J LHS2}, and \eqref{u^J LHS3} we obtain the following differential equation for $\beta_{2G,J}(x)$,
\begin{equation}\label{eq 413}
    \nu J \beta_{2G,J} - x\frac{\dd \beta_{2G,J}}{\dd x} = \lambda_{G,J}x^{J(\nu-1)+1-2G},
\end{equation} for some constant $\la_{G,J}$. This is the desired differential equation \eqref{eq 4132}.
    
\end{remark}

Having proved Theorem \ref{theorem r structure} we will now show how to use equation \eqref{dRdu} to recursively derive differential equations \eqref{eq 413} with explicit $\la_{G,J}$. Solving these allows us to explicitly find $\be_{2G,J}(x)$ and thus the numbers $\mathcal{N}_{G}(2\nu,J)$ via \eqref{2legged counts}. From \eqref{beta 0 def} we have

\begin{subequations}\label{alpha 0s}
    \begin{eqnarray}
        \beta_{0,0}(x) &=& x,\\
        \beta_{0,1}(x) &=& -{2\nu-1 \choose \nu}x^\nu,\\
        \beta_{0,2}(x) &=& \frac{ (2\nu-1!)^2}{(\nu-1!)^3\nu!}x^{2\nu-1}.
    \end{eqnarray}
\end{subequations}
In order to derive $\beta_{2g,j}(x)$ for $g>0$ we will use \eqref{dRdu}. We show how to derive $r_2(x;u)$ from $r_0(x;u)$, larger values of $g$ can then be determined recursively. Evaluating \eqref{dRdu} at order $N^{-2}$ we find,
\begin{equation}\label{N^{-2} matched}
           \nu u\frac{\ddd r_2}{\ddd u} = r_0\Big(\frac{\ddd ^3r_0}{3!\ddd x^3} + \frac{\ddd r_2}{\ddd x}\Big) + r_2\Big(\frac{\ddd r_0}{\ddd x}-1\Big).
\end{equation}
Substituting in $\beta_{0,0}=x$ (found in equation \eqref{alpha 0s}) and evaluating the above equation at powers of $u^1$ and $u^2$ we find
\begin{eqnarray}
    \nu \beta_{2,1}-x\frac{\dd \beta_{2,1}}{\dd x} &=& \frac{x}{3!}\frac{\dd ^3\beta_{0,1}}{\dd x^3},\label{el3}\\
    2\nu \beta_{2,2} - x\frac{\dd \beta_{2,2}}{\dd x} &=& \frac{\dd}{\dd x}\left(\beta_{2,1}\beta_{0,1}\right) + \frac{1}{6}\left( \beta_{0,1}\frac{\dd ^3 \beta_{0,1}}{\dd x^3} + x \frac{\dd ^3 \beta_{0,2}}{\dd x^3}\right). \label{el4}
\end{eqnarray}

We can solve equation \eqref{el3} to find
 \begin{equation}
    \beta_{2,1}(x) = \lambda x^{\nu} - \frac{(\nu-2)(2\nu -1)!}{2(\nu -2)!(\nu -1)!3!}x^{\nu -2}, 
\end{equation} 
where it remains to find the constant $\lambda$. But from Theorem \ref{theorem r structure} it follows $\beta_{2,1}$ is of degree $x^{\nu-2}$. Hence, $\la=0$ and we conclude
\begin{equation}
    \beta_{2,1}(x) = -\frac{(\nu -2)(2\nu -1)!}{2(\nu -2)!(\nu -1)!3!}x^{\nu -2}.
\end{equation}
We can then use this information to solve \eqref{el4} to find
\[ \beta_{2,2}(x) = \frac{(2+3\nu (\nu -2))(2\nu -1!)^2}{6(\nu -2!)(\nu -1!)^2\nu !}x^{2\nu -3}.\]
Formulae for $\be_{2g,j}$ for larger values of $g$ and $j$ can then be  evaluated recursively using \eqref{dRdu}. See Section \ref{subsec:Qgj-polynomials} for explicit values for $\be_{2g,j}$, when $j=1,2,3$ and $g=0,1,2,3,4,5$ (recall that  $\be_{2g,j}(x)|_{x=1}$ is related to  $\mathcal N_g(2\nu,j)$  by \eqref{2legged counts} and 
    $\mathcal N_g(2\nu,j)
    =
    c_\nu^j Q_{g,j}(\nu)$ by \eqref{eq:Qgj-form}).
\section{The Asymptotic Expansion of $\mathcal{F}_{nN}$}\label{determining free section}
In this section we prove Theorem \ref{free energy theorem} using Theorem \ref{recurrence theorem} and Lemma \ref{sigma lemma}. Combining equations \eqref{uv relation}, \eqref{free relation}, \eqref{r relation} and \eqref{the second f} we find
\begin{equation}\label{second F u}
    (\nu ^2u^2)\frac{\ddd  ^2\mathcal{F}_{nN}}{\ddd  u^2} + \nu (\nu +1)u\frac{\ddd \mathcal{F}_{nN}}{\ddd u}+1/2 = \frac{1}{4x^2}\mathcal{R}_n(\mathcal{R}_{n+1}+\mathcal{R}_{n-1}).
\end{equation}
Recalling \eqref{top exp free energy}, we know that $\mathcal{F}_{nN}(x;u)$ has the topological expansion 
\begin{equation}
    \mathcal{F}_{nN}(x;u) = \sum_{g=0}^\infty \frac{f_{2g}(x;u)}{N^{2g}},
\end{equation} 
where $x = \frac{n}{N}$. Furthermore, $f_{2g}(x;u)$ can be written as a power series in $u$ as
\begin{equation}\nonumber
    f_{2g}(x;u) = \sum_{j=0}^\infty \alpha_{2g,j}(x)u^j.
\end{equation}
The following theorem about the structure of $\alpha_{2g,j}(x)$ follows from the same arguments used in the proof of Theorem \ref{theorem r structure} and the details are left to the reader.
\begin{theorem}\label{theorem f structure} It holds that
\begin{equation}
     f_{2g}(x;u) = \sum_{j=0}^\infty \alpha_{2g,j}(x)u^j,
\end{equation}
where $\alpha_{2g,j}(x) = \tilde{c}_{2g,j}x^{\widetilde{\mathscr{D}}}$ and $\widetilde{\mathscr{D}} = j(\nu-1)-2g$. That is, $\tilde{c}_{2g,j}(x)$ is a monomial in $x$ of degree $\widetilde{\mathscr{D}}$. If $\widetilde{\mathscr{D}}<-2$ then $\alpha_{2g,j}(x)=\tilde{c}_{2g,j} = 0$.
\end{theorem}

Through equation \eqref{second F u} we can relate $f_{2g}$ and $r_{2g}$ by the equations
\begin{eqnarray}
     (\nu^2u^2)\frac{\ddd^2f_0}{\ddd u^2} + \nu(\nu+1)u\frac{\ddd f_0}{\ddd u}+1/2 &=& \frac{r_0^2}{2x^2}, \label{fr rel 1}\\
     (\nu^2u^2)\frac{\ddd ^2f_2}{\ddd u^2} + \nu(\nu+1)u\frac{\ddd f_2}{\ddd u} &=& \frac{r_0}{4x^2}(4r_2+\frac{\ddd ^2r_0}{\ddd x^2}),\label{fr rel 2}\\
     &\vdots&\nonumber
\end{eqnarray}
By comparing coefficients of $u$ in \eqref{fr rel 1} we find, 
\begin{subequations}\label{ab alpha relation rel 1}
    \begin{eqnarray}
    \beta_{0,0}^2 &=& x^2,\\
    \beta_{0,0}\beta_{0,1}, &=& x^2\nu(\nu+1)\alpha_{0,1},\\
    &\vdots&\nonumber
\end{eqnarray}
\end{subequations}
Similarly, comparing coefficients of $u$ in \eqref{fr rel 2} we find,
\begin{subequations}\label{ab alpha relation rel 2}
\begin{align}
   \beta_{0,0}\!\left(2\beta_{2,0}+\frac{\dd^2\beta_{0,0}}{\dd x^2}\right) &= 0, \\
   2(\beta_{0,1}\beta_{2,0}+\beta_{2,1}\beta_{0,0}) 
   + \left(\beta_{0,1}\frac{\dd^2\beta_{0,0}}{\dd x^2} + \beta_{0,0}\frac{\dd^2\beta_{0,1}}{\dd x^2}\right) 
   &= 4x^2\nu(\nu+1)\alpha_{2,1}, \\
   &\vdots \nonumber
\end{align}
\end{subequations}

Using these equations we can deduce a relation between the $\beta_{2g,j}(x)$'s and the $\alpha_{2g,j}(x)$'s. By solving equations \eqref{ab alpha relation rel 1} and \eqref{ab alpha relation rel 2} we can determine $\alpha_{0,1}$ and $\alpha_{2,1}$,
\begin{eqnarray}
    \alpha_{0,1} &=& -\frac{(2\nu-1)!}{\nu!(\nu+1)!}x^{\nu-1},\\
 \alpha_{2,1} &=& -\frac{(2\nu-1)!}{12\nu!(\nu-2)!}x^{\nu-3}.    
\end{eqnarray}
One can then iteratively repeat the arguments to find higher and higher powers of $u$ and $g$ in the free energy expansion. For example:
\begin{eqnarray}
    \alpha_{0,2} &=& \frac{((2\nu-1)!)^2}{4(\nu!)^3(\nu-1)!}x^{2\nu-2},\\
    \alpha_{2,2} &=& \frac{(3\nu-1)((2\nu-1)!)^2}{24(\nu-2)!(\nu-1)!(\nu!)^2}x^{2\nu-4}.
\end{eqnarray}

See Section \ref{subsec:Sgj-polynomials} for explicit values for $\alpha_{2g,j}$, when $j=1,2,3$ and $g=0,1,2,3,4,5$ (recall that  $\alpha_{2g,j}(x)|_{x=1}$ is related to  $\mathscr N_g(2\nu,j)$  by \eqref{numbers and free energy} and 
    $\mathscr N_g(2\nu,j)
    =
    C_\nu^j S_{g,j}(\nu)$ by \eqref{eq:Sgj-form}).

\section{Explicit Polynomials in Theorems \ref{thm:main-N}, \ref{thm:main-Ncal}, \ref{thm 2legged combinatorics}
and \ref{thm combinatorics}}\label{sec:coefficient-tables}
We first explicitly state the polynomials $Q_{g,j}(\nu)$ and $S_{g,j}(\nu)$ introduced in Theorems \ref{thm 2legged combinatorics} and \ref{thm combinatorics}. We use these polynomials, and also further graph counts provided in Appendix \ref{sec further eqns}, to deduce $a_{\ell}^{(g,\nu)}$ and $b_{\ell}^{(g,\nu)}$ which are explicitly stated in Section \ref{sect:a and b state}.

\subsection{The Fixed-\texorpdfstring{\((g,j)\)}{(g,j)} Polynomials
\texorpdfstring{\(Q_{g,j}\)}{Qgj} and
\texorpdfstring{\(S_{g,j}\)}{Sgj}}
\label{subsec:QS-polynomial-tables}

In this section we list the explicit polynomials \(Q_{g,j}\) and
\(S_{g,j}\) appearing in Theorems~\ref{thm 2legged combinatorics}
and \ref{thm combinatorics}. These were derived in Sections \ref{determining rn section} and \ref{determining free section} respectively.

\subsubsection{The Polynomials \texorpdfstring{\(Q_{g,j}\)}{Qgj}}
\label{subsec:Qgj-polynomials}

For \(0\leq g\leq 5\) and \(1\leq j\leq 3\), the polynomials
\(Q_{g,j}\) in \eqref{eq:Qgj-form} are as follows:
{\small
\begin{eqnarray*}
   Q_{0,1}(\nu) &=& 1, \\  
   Q_{0,2}(\nu) &=& 2\nu, \\ 
   Q_{0,3}(\nu) &=& 3\nu(3\nu-1),  \\   
   Q_{1,1}(\nu) &=& \frac{1}{12}\prod_{i=0}^2(\nu-i),  \\
   Q_{1,2}(\nu) &=& \frac{1}{3}\left(3\nu^2-6\nu+2\right)\prod_{i=0}^1(\nu-i),\\
     Q_{1,3}(\nu) &=&\frac{3}{4}\left(17\nu^3-39\nu^2+24\nu-4\right)\prod_{i=0}^1(\nu-i),\\ 
    Q_{2,1}(\nu) &=& \frac{(5\nu-7)}{1440}\prod_{i=0}^4(\nu-i),\\
    Q_{2,2}(\nu) &=& \frac{1 }{360}(2\nu-3)(49\nu^3 -201\nu^2+220\nu-56)\prod_{i=0}^2(\nu-i),\\
    Q_{2,3}(\nu) &=& \frac{1}{160}(3\nu-5)(539\nu^4-1788\nu^3+2005\nu^2-856\nu+112)\prod_{i=0}^2(\nu-i),\\
    Q_{3,1}(\nu) &=& \frac{1}{362880}\left(35\nu^2-147\nu+124\right)\prod_{i=0}^6(\nu-i),\\
    Q_{3,2}(\nu) &=& \frac{1}{45360}(2\nu-5)\bigg(1181\nu^5 -9883\nu^4+29848\nu^3-40538\nu^2+23976\nu -4464\bigg) \prod_{i=0}^3(\nu-i),\\
    Q_{3,3}(\nu) &=& \frac{1}{4480}(3\nu-7)\bigg(8621\nu^7 - 78417\nu^6 + 288943\nu^5-555499\nu^4+594372\nu^3  -346452\nu^2  \\&&\,+98272\nu - 9920\bigg)  \prod_{i=0}^2(\nu-i), \\
    Q_{4,1}(\nu) &=& \frac{1}{87091200}\left(175\nu^3-1470\nu^2+3509\nu-2286\right)\prod_{i=0}^8(\nu-i),\\
    Q_{4,2}(\nu) &=& \frac{1}{10886400}\bigg(21015\nu^6 - 248463\nu^5 + 1108499\nu^4 - 2386617\nu^3 + 2597902\nu^2 - 1313808\nu  \\&&\,+ 219456\bigg) (2\nu - 5)(2\nu - 7) \prod_{i=0}^4(\nu-i) ,
            \end{eqnarray*} }
{\small
\begin{eqnarray*} 
 Q_{4,3}(\nu) &=& \frac{1}{1075200} \bigg( 2805887\nu^{10} - 46719825\nu^9 + \\&&\,338126378\nu^8 - 1396332194\nu^7  + 3628412663\nu^6  - 6163425041\nu^5 + 6874078128\nu^4 \\&& - 4909790588\nu^3 + 2108489904\nu^2 - 476570112\nu    + 40965120 \bigg)\prod_{i=0}^3(\nu-i),\\
    Q_{5,1}(\nu) &=& \frac{1}{11496038400} \left(385\nu^4 - 5390\nu^3 + 24959\nu^2 - 44242\nu + 24528\right) \prod_{i=0}^{10}(\nu-i),\\[1pt]
        Q_{5,2}(\nu) &=& \frac{1}{718502400}(2\nu - 7)(2\nu - 9)\bigg(168155\nu^8 - 3106577\nu^7 + 23488479\nu^6  - 94884829\nu^5 \\&&\, +  223426562\nu^4 - 312172674\nu^3 + 249503444\nu^2 \\[1pt]
    &&\, - 101165280\nu + 14716800\bigg)\prod_{i=0}^5(\nu-i),\\[1pt]
    Q_{5,3}(\nu) &=& \frac{1}{141926400}(3\nu- 11)\bigg(46360603\nu^{11} - 880543553\nu^{10} + 7377406270\nu^9 - 35895463278\nu^8  \\&&\, + 112326954267\nu^7  - 236357283609\nu^6  + 339283640108\nu^5  - 329560955560\nu^4 \\&&\, + 209749893152\nu^3  - 81769381200\nu^2 + 17052537600\nu  - 1373568000\bigg)\prod_{i=0}^4(\nu-i).
\end{eqnarray*}}

\subsubsection{The Polynomials \texorpdfstring{\(S_{g,j}\)}{Sgj}}
\label{subsec:Sgj-polynomials}

For \(0\leq g\leq 5\) and \(1\leq j\leq 3\), the polynomials
\(S_{g,j}\)  in \eqref{eq:Sgj-form} are as follows: 
{\small \begin{flalign*}
  S_{0,1}(\nu) &= 1, & \\
S_{0,2}(\nu) &= \frac{1}{2}(\nu+1)^2\nu, & \\
S_{0,3}(\nu) &= (\nu+1)^3\nu^3, & \\
S_{1,1}(\nu) &= \frac{1}{12}(\nu+1)\nu(\nu-1), & \\
S_{1,2}(\nu) &= \frac{1}{12}(\nu+1)^2\nu^2(3\nu-1)(\nu-1), & \\
S_{1,3}(\nu) &= \frac{1}{12}\left(17\nu^2-13\nu+2\right) (\nu+1)^3\nu^3(\nu-1), & \\
S_{2,1}(\nu) &= \frac{1}{1440}(5\nu-2)\prod_{i=-1}^3(\nu-i) , & \\
S_{2,2}(\nu) &= \frac{1}{1440 }(\nu+1)^2\nu^2 (2\nu-3)(49\nu^2-43\nu+6)\prod_{i=1}^2(\nu-i), & \\
S_{2,3}(\nu) &= \frac{1}{480} (\nu+1)^3\nu^3(\nu-1) \left(539\nu^5-2356\nu^4+3677\nu^3-2460\nu^2+660\nu-48\right), & \\
S_{3,1}(\nu) &= \frac{1}{362880} \left(35\nu^2-77\nu+12\right)\prod_{i=-1}^5(\nu-i), & \\ 
S_{3,2}(\nu) &= \frac{1}{181440}(\nu+1)^2\nu^2 (2\nu-5)(1181\nu^4 - 4282\nu^3+4969\nu^2-1868\nu+120) \times \prod_{i=1}^3(\nu-i),& 
\end{flalign*}
\begin{flalign*}
S_{3,3}(\nu) &= \frac{1}{13440}\bigg(8621\nu^7 -67098\nu^6+207750\nu^5-326324\nu^4 \\&+273029\nu^3-115578\nu^2  +20560\nu-800\bigg)  (\nu+1)^3\nu^3\prod_{i=1}^2(\nu-i) , & \\ 
S_{4,1}(\nu) &= \frac{1}{87091200} \left(175\nu^3 - 945\nu^2 + 1094\nu - 72\right)\prod_{i=-1}^7(\nu-i), & \\
S_{4,2}(\nu) &= \frac{1}{43545600}(2\nu - 5)(2\nu - 7)\bigg(21015\nu^5 - 117163\nu^4 + 228063\nu^3  \\ &- 182453\nu^2  + 50034\nu - 1512\bigg)  (\nu+1)^2\nu^2\prod_{i=1}^4(\nu-i), & \\
S_{4,3}(\nu) &= \frac{1}{9676800}\bigg(2805887\nu^9 - 33646824\nu^8 + 170341574\nu^7 - 473605544\nu^6  \\ & + 786759767\nu^5 - 794026448\nu^4  + 471186660\nu^3 - 149071904\nu^2 \\
& + 19693632\nu  - 376320\bigg) (\nu+1)^3\nu^3\prod_{i=1}^3(\nu-i), & \\
    S_{5,1}(\nu) &= \frac{1}{11496038400} \left(385\nu^4 - 3850\nu^3 + 11099\nu^2 - 8954\nu + 240\right)\prod_{i=-1}^9(\nu-i), & \\
S_{5,2}(\nu) &= \frac{1}{2874009600}(2\nu - 7)(2\nu - 9)\bigg(168155\nu^7 - 1803472\nu^6 + 7641252\nu^5 \\ & - 16263590\nu^4 + 18157345\nu^3  - 9913818\nu^2 + 2014128\nu & \\ & - 25920\bigg) (\nu+1)^2\nu^2 \times \prod_{i=1}^5(\nu-i), & \\ 
S_{5,3}(\nu) &= \frac{1}{425779200}\bigg( 46360603\nu^{12} - 973391694\nu^{11} + 9018453443\nu^{10} - 48560689270\nu^9  \\ & + 168394080893\nu^8  - 393534106562\nu^7 + 629719954801\nu^6 \\
&\,- 686021525378\nu^5  + 494760354900\nu^4 - 222565585336\nu^3 + 55430820000\nu^2 \\ &- 5767948800\nu  + 48384000 \bigg)(\nu+1)^3\nu^3\prod_{i=1}^3(\nu-i). &
\end{flalign*}
}

\subsection{The Coefficient Polynomials $b_\ell^{(g,\nu)}$ and \(a_\ell^{(g,\nu)}\) }\label{sect:a and b state}

In this section we list the coefficient polynomials appearing in
Theorems~\ref{thm:main-N} and \ref{thm:main-Ncal}.  These are the
coefficients \(b_\ell^{(g,\nu)}\) and \(a_\ell^{(g,\nu)}\) appearing in
\eqref{eq:main-N} and \eqref{eq:main-Ncal}, respectively. The method for determining these coefficients explicitly was outlined in
Section~\ref{Sec Roadmap}.

\subsubsection{Coefficients for \(\mathscr N_g(2\nu,j)\)}\label{Sec explicit b coeffs}

We first list the coefficients \(b_\ell^{(g,\nu)}\) for
\(g=2,3,4\).  For each such \(g\), the coefficients
\(b_\ell^{(g,\nu)}\), \(0\leq \ell\leq 3g-3\), are polynomials in
\(\nu\) of degree \(3g-3\).

\begin{itemize}
\item The 4 genus-2 coefficients in~\eqref{eq:main-N} are all cubic polynomials in $\nu$ and are given by:
{\small
\begin{flalign*}
b^{(2,\nu)}_0 &= \frac{-1}{2880}\big( 12+ 80\nu+ 71\nu^2+8\nu^3\big),&\\
b^{(2,\nu)}_1 &= \frac{\nu}{1440}\big( 40+ 98\nu+31\nu^2\big),& \\
b^{(2,\nu)}_2 &= \frac{-\nu^2}{576}\big(25+22\nu\big),& \\
b^{(2,\nu)}_3 &= \frac{7}{360}\nu^3. &   
\end{flalign*}
}
\item The 7 genus-3 coefficients in~\eqref{eq:main-N} are all polynomials of degree $6$ in $\nu$ and are given by:
{\small
\begin{flalign*}
b^{(3,\nu)}_0 &= \frac{1}{725760}\big( 720 +  22176 \nu +  103996 \nu^2+  148106 \nu^3 +  70537 \nu^4 + 9168 \nu^5 + 32\nu^6\big),& \\
b^{(3,\nu)}_1 &= \frac{-\nu}{120960}\big(3696 + 40302 \nu + 105063 \nu^2+ 88751 \nu^3+ 23726\nu^4 + 1352\nu^5\big),& \\
b^{(3,\nu)}_2 &= \frac{\nu^2}{51840}\big(9844 + 59892 \nu + 92779 \nu^2 + 43983 \nu^3 + 5137\nu^4 \big),& \\
b^{(3,\nu)}_3 &= \frac{-\nu^3}{362880} \big(178108 + 644796 \nu + 560697 \nu^2 + 115989\nu^3\big),& \\
b^{(3,\nu)}_4 &= \frac{\nu^4}{6912}\big(4311 + 8764\nu + 3324\nu^2\big),& \\
b^{(3,\nu)}_5 &= \frac{-\nu^5}{864}\big(335 + 297\nu\big),& \\
b^{(3,\nu)}_6 &= \frac{245}{2592}\nu^6.  &  
\end{flalign*}
}
\item The 10 genus-4 coefficients in~\eqref{eq:main-N} are all polynomials of degree $9$ in $\nu$ and are given by:
{\small \begin{flalign*}
     b^{(4,\nu)}_0 &= \frac{-1}{87091200}\bigg(60480 + 6091776 \nu + 69138396 \nu^2 + 271690872 \nu^3 + 
  465121035 \nu^4  \\ & + 369591027 \nu^5 + 131702178 \nu^6 + 17530000 \nu^7 + 
  298048 \nu^8 - 27008 \nu^9\bigg), & \\
b^{(4,\nu)}_1 &= \frac{\nu}{87091200}\bigg(6091776 + 152189712 \nu + 1008188924 \nu^2 + 
  2656587008 \nu^3  + 3172645503 \nu^4\\ & + 1753874888 \nu^5 + 
  417930588 \nu^6 ,+ 33675968 \nu^7 + 264640 \nu^8\bigg),& \\
b^{(4,\nu)}_2 &= \frac{-\nu^2}{87091200}  \bigg(83051316 + 1215477348 \nu + 5407498229 \nu^2 + 
  9947570376 \nu^3\\& + 8266710762 \nu^4 + 3054981038 \nu^5 + 
  440225473 \nu^6 + 16310128 \nu^7\bigg),& \\
b^{(4,\nu)}_3 &= \frac{\nu^3}{87091200}\bigg(   478979296 + 4713790504 \nu + 14650381372 \nu^2 + 
  18758897792 \nu^3  \\& + 10420470078 \nu^4 + 2327628744 \nu^5 + 
  154051144 \nu^6\bigg),& \\
  b^{(4,\nu)}_4 &= \frac{-\nu^4}{87091200}\bigg(1497758248 + 10303469792 \nu + 22295920990 \nu^2 + 
  19083694728 \nu^3 \\ &+ 6406095591 \nu^4 + 656850999 \nu^5\bigg),& \end{flalign*}
  \begin{flalign*}
  b^{(4,\nu)}_5 &= \frac{\nu^5}{87091200}\bigg(2797604320 + 13403430040 \nu + 19389912360 \nu^2 + \\&
  10028168640 \nu^3  + 1544787615 \nu^4\bigg),&\\
  b^{(4,\nu)}_6 &= \frac{-\nu^6}{87091200}\big(3221868790 + 10320667420 \nu + 9021248320 \nu^2 + 
  2140698280 \nu^3\big),& \\[6pt]
  b^{(4,\nu)}_7 &= \frac{\nu^7}{87091200}\big(2248560160 + 4352240480 \nu + 1745323720 \nu^2\big),&\\[6pt]
b^{(4,\nu)}_8 &= \frac{-\nu^8}{87091200}\big(873846400  + 775944400 \nu\big),& \\[6pt]
b^{(4,\nu)}_9 &= \frac{259553}{155520}\nu^9.&
\end{flalign*}
}
\end{itemize}

\subsubsection{Coefficients for \(\mathcal N_g(2\nu,j)\)}\label{Sec explicit a coeffs}

We next list the coefficients \(a_\ell^{(g,\nu)}\) for
\(g=1,2,3,4\).  For each such \(g\), the coefficients
\(a_\ell^{(g,\nu)}\), \(0\leq \ell\leq 3g-1\), are polynomials in
\(\nu\) of degree \(3g-1\).

 \begin{itemize}
 \item  The 3 genus-1 coefficients in~\eqref{eq:main-Ncal} are all quadratic polynomials in $\nu$ and are given by: 
{\small \begin{flalign*}
a^{(1,\nu)}_0 &= \frac{\nu}{12}\big(2+\nu\big), & \\
a^{(1,\nu)}_1 &= \frac{-\nu}{12}\big(2+3\nu\big), & \\
a^{(1,\nu)}_2 &= \frac{1}{6}\nu^2, &
\end{flalign*}}
    \item  The 6 genus-2 coefficients in~\eqref{eq:main-Ncal} are all polynomials of degree $5$ in $\nu$ and are given by:
{\small \begin{flalign*}
    a^{(2,\nu)}_0 &= \frac{-\nu}{480}\big(56 + 302 \nu + 383 \nu^2 + 130 \nu^3 + 8 \nu^4\big), & \\
    a^{(2,\nu)}_1 &= \frac{\nu}{1440}\big(168 + 2114 \nu + 4985 \nu^2 + 3102 \nu^3 + 428 \nu^4\big), & \\
    a^{(2,\nu)}_2 &= \frac{-\nu^2}{1440}\big(1208 + 6716 \nu + 7802 \nu^2 + 1969 \nu^3\big), & \\
a^{(2,\nu)}_3 &= \frac{\nu^3}{288}\big(576 + 1582 \nu + 745 \nu^2\big), &\\
a^{(2,\nu)}_4 &= \frac{-\nu^4}{72}\big(141 + 157 \nu\big), &\\
a^{(2,\nu)}_5 &= \frac{49}{72}\nu^5. &
\end{flalign*}}
\item The 9 genus-3 coefficients in~\eqref{eq:main-Ncal} are all polynomials of degree $8$ in $\nu$ and are given by:
{\small \begin{flalign*}
a^{(3,\nu)}_0 &= \frac{\nu}{72576} \bigg(17856 + 235296 \nu + 939236 \nu^2 + 1505064 \nu^3 + 
   1032603 \nu^4 + 285860 \nu^5 \\ & + 24472 \nu^6 + 64 \nu^7\bigg), &\\
a^{(3,\nu)}_1 &= \frac{-\nu}{362880}\bigg(89280 + 2588256 \nu + 17470540 \nu^2 + 43350840 \nu^3 + 
   45171237 \nu^4 \\&+ 19790842 \nu^5 + 3202640 \nu^6 + 122384 \nu^7\bigg), 
   \end{flalign*}}

{\small \begin{flalign*}
a^{(3,\nu)}_2 &= \frac{\nu^2}{120960}\bigg(470592 + 7034376 \nu + 29599732 \nu^2 + 
   47839718 \nu^3 + 31864925 \nu^4 \\ & + 8166599 \nu^5 + 591434 \nu^6\bigg), &\\
   a^{(3,\nu)}_3 &= \frac{-\nu^3}{51840}\bigg(1189824 + 11112596 \nu + 30497468 \nu^2+ 
   31533303 \nu^3 + 12291699 \nu^4 \\ &+ 1410522 \nu^5\bigg), &\\
a^{(3,\nu)}_4 &= \frac{\nu^4}{51840}\big(3544928 + 21617504 \nu + 37979568 \nu^2 + 
   22989726 \nu^3 + 4013349 \nu^4\big), &\\
a^{(3,\nu)}_5 &= \frac{-\nu^5}{17280}\big(1969104 + 7691608 \nu + 7913786 \nu^2 + 2145687 \nu^3\big),\\
a^{(3,\nu)}_6 &= \frac{\nu^6}{2592}\big(279762 + 640168 \nu + 295069 \nu^2\big), &\\
a^{(3,\nu)}_7 &= \frac{-\nu^7}{2592}\big(140998 + 144559 \nu\big), &\\
a^{(3,\nu)}_8 &= \frac{1225}{108}\nu^8. &
\end{flalign*}}

\item The 12 genus-4 coefficients in~\eqref{eq:main-Ncal} are all polynomials of degree $11$ in $\nu$ and are given by:
{\small \begin{flalign*}
    a^{(4,\nu)}_0 &=\frac{-\nu}{87091200} \bigg(92171520 + 2098742688 \nu + 16092283032 \nu^2 + 
   56367784900 \nu^3  \\ & + 100912028042 \nu^4  + 95941872033 \nu^5  + 
   47857995514 \nu^6 + 11645825128 \nu^7  \\ &++ 1123745952 \nu^8  
   13504960 \nu^9 - 1134336 \nu^{10}\bigg), & \\
a^{(4,\nu)}_1 &= \frac{\nu}{87091200} \bigg(92171520 + 4497305760 \nu + 56186738136 \nu^2 + 
   289168376484 \nu^3 \\ & + 726203633242 \nu^4 
   + 953850310201 \nu^5  + 
   664891212498 \nu^6  + 237899049736 \nu^7  \\&+ 39460597200 \nu^8 + 
   2326824640 \nu^9 + 12389120 \nu^{10}\bigg), & \\
   a^{(4,\nu)}_2 &=\frac{-\nu^2}{87091200} \bigg(2398563072 + 64510638912 \nu + 
   542085293504 \nu^2  + 2002789878500 \nu^3  \\ &+ 3693545234356 \nu^4 + 
   3559327488365 \nu^5  + 1780743266214 \nu^6  \\
   &+ 434681953116 \nu^7  + 
   44174048544 \nu^8 + 1204524992 \nu^{9}\bigg), &\\ 
   a^{(4,\nu)}_3 &=\frac{\nu^3}{87091200}\bigg(24416183808 + 442459038240 \nu + 
   2671980151700 \nu^2  + 7254888306748 \nu^3  \\ & + 9821647137541 \nu^4 
   + 
   6796178238906 \nu^5 + 2320687841608 \nu^6 \\ & + 347315456984 \nu^7 + 
   16360414736 \nu^{8}\bigg),& 
   \end{flalign*}}

{\small \begin{flalign*}
a^{(4,\nu)}_4 &=\frac{-\nu^4}{87091200} \bigg(133174336320 + 1736002857024 \nu + 
   7709497511976 \nu^2  + 15373668288148 \nu^3 \\ & + 14950112833628 \nu^4 + 
   7062752648152 \nu^5 + 1479132827021 \nu^6 + 102674086346 \nu^{7}\bigg),&\\
a^{(4,\nu)}_5 &=\frac{\nu^5}{87091200} \bigg(441520978624 + 4234961647976 \nu + 
   13818569757108 \nu^2  + 19799971981928 \nu^3 \\ &  +13148780770332 \nu^4 + 
   3810551222121 \nu^5 + 370238758206 \nu^{6}\bigg),& \\
a^{(4,\nu)}_6 &=\frac{-\nu^6}{87091200}\bigg(944715646560 + 6660438078560 \nu + 
   15620288203240 \nu^2  + 15287477780820 \nu^3 \\ & + 6228745467280 \nu^4 + 
   837587645685 \nu^{5}\bigg),&\\
a^{(4,\nu)}_7 &=\frac{\nu^7}{87091200} \bigg(1336183743440 + 6772555031480 \nu + 
   10850932667540 \nu^2  + 6516572243020 \nu^3 \\ &+ 1232139788705 \nu^{4}\bigg),&\\
a^{(4,\nu)}_8 &=\frac{-\nu^8}{87091200}\bigg(1243814173840  + 4307716419440 \nu + 
   4235681481360 \nu^2  + 1180572677480 \nu^{3}\bigg),&\\
a^{(4,\nu)}_9 &=\frac{\nu^9}{87091200} \big(733890670800 + 1560046779200 \nu + 
   711981918200 \nu^{2}\big),&\\[6pt]
a^{(4,\nu)}_{10} &= \frac{-\nu^{10}}{87091200} (249065196800  + 245759637200 \nu),& \\[6pt]
a^{(4,\nu)}_{11} &= \frac{4412401}{10368}\nu^{11}.&
\end{flalign*}}
\end{itemize}

\section{Conclusion}

In this paper, we determine explicit formulae for $\mathscr{N}_g(2\nu,j)$ and $\mathcal{N}_g(2\nu,j)$ which hold for general $j$ and $\nu$, and for fixed $g\leq 4$. This then leads to formulae for the asymptotic behavior of $\mathscr{N}_g(2\nu,j)$ and $\mathcal{N}_g(2\nu,j)$ as $\nu\to\infty$. To the best of the authors' knowledge, the results for counts and asymptotics for genus $g\geq 2$ are new to the literature. The method presented here can readily be extended to higher genus, only demanding additional computational cost. 

Proving the Conjectures \ref{conj:b-polynomial} and \ref{conj:a-polynomial} as well as the conjectures raised in Remark \ref{conjecture remark} concerning the polynomials $Q_{g,j}(\nu)$ and $S_{g,j}(\nu)$ would be an interesting direction of future research and would extend many of the results presented here to higher genus (at least structurally). Another interesting avenue of future research would be to determine a formula for $\mathscr{N}_g(2\nu,j)$ (and $\mathcal{N}_g(2\nu,j)$) which holds for general $g$ and fixed $\nu$ and $j$. 

As we have seen, the case of general \(j\), with \(\nu\) and \(g\)
fixed (see ~\eqref{bleher cubic sphere}--\eqref{N j+1 2} and
Theorem~\ref{thm hexic counts}), and the case of general \(\nu\), with
\(j\) and \(g\) fixed (see ~Theorem~\ref{thm combinatorics}), both lead
to nice closed-form expressions. It is therefore natural to ask whether the
remaining case of general \(g\), with \(j\) and \(\nu\) fixed, also
admits a nice closed-form expression. However, if such an expression is found, it must be fundamentally different from those found for general $j$ and $\nu$. This can be seen by recognizing that $\mathscr{N}_g(2\nu,j)$ denotes the number of connected labeled $2\nu$-valent graphs with $j$ vertices on a compact Riemann surface of genus $g$ that cannot be realized on Riemann surfaces of smaller genus. Thus, if one fixes $j$ and $\nu$, there will be a critical value $g_c$ where $\mathscr{N}_g(2\nu,j)=0$ for $g>g_c$. Hence, there will be infinitely many $g$'s for which $\mathscr{N}_g(2\nu,j)=0$, which means that $\mathscr{N}_g(2\nu,j)$, for general $g$ with fixed $j$ and $\nu$ cannot be a rational expression of $g$ (compare with Theorems \ref{thm combinatorics} and \ref{thm hexic counts}).

\begin{appendices} 

\renewcommand{\theHequation}{A.\arabic{equation}} 
\setcounter{equation}{0}                        

\section{Proof of Equations \eqref{eq:main-N} and \eqref{eq:main-Ncal}}\label{sect: dfi form}

When $\nu=2$ the hypergeometric function in \eqref{eq:Ncal-ELT-hypergeom} reduces to
${}_2F_1(-j,-2j;2-2g-\ell-j;-1)$.
It was conjectured in \cite[Conjecture 4.2]{Erc-et-al-23b} that at this value of $\nu$ the summands can be written without hypergeometric functions, which was
proved by Yattselev in \cite{Yat}. In this section we perform a reduction of this type for general
$\nu$, turning the hypergeometric expressions in \eqref{eq:N-ELT-hypergeom} and \eqref{eq:Ncal-ELT-hypergeom} into
the explicit binomial sums of Theorems \ref{thm:main-N} and \ref{thm:main-Ncal}.

 First, observe that for positive $j$ the following identity holds,
\begin{equation}
    (-j)_k =(-1)^k\frac{j!}{(j-k)!},
\end{equation}
where $(a)_k$ is the rising factorial \cite[Section 5.2]{NIST:DLMF}. Thus,
\begin{eqnarray}
    \pFq{2}{1}{-j,-\lambda_j}{-G-\ell-j}{\tfrac{1}{1-\nu}} &=& \sum^{j}_{s=0} \frac{(-j)_s(-\lambda_j)_s}{(-G-\ell-j)_s}\frac{(1-\nu)^{-s}}{s!},\nonumber \\
    &=& \frac{j!(\lambda_j)!}{(G+\ell+j)!} \sum^{j}_{s=0} \frac{(G+\ell + j-s)!}{(j-s)!(\lambda_j-s)!}\frac{(\nu-1)^{-s}}{s!},\nonumber \\
    &=& \frac{j!(\lambda_j)!}{(G+\ell+j)!} \sum^{j}_{r=0} \frac{(G+\ell + r)!}{(r)!(\lambda_j-(j-r))!}\frac{(\nu-1)^{r-j}}{(j-r)!}.\nonumber
\end{eqnarray}
The first equality holds form the definition of the hypergeometric function \cite[Section 15.2]{NIST:DLMF} and the last inequality holds by making the change of variables $r=j-s$. Hence, we find that
\begin{multline}
    (\nu - 1)^j  \sum_{\ell = 0}^{L} \left(c^{(g,\nu)}_{\ell} {G + \ell+j   \choose j}
\ \pFq{2}{1}{-j,\lambda_j}{-G-\ell-j}{\tfrac{1}{1-\nu}}\right)\\ =    \sum_{\ell = 0}^{L} \left(c^{(g,\nu)}_{\ell} \frac{(\lambda_j)!}{(G+\ell)!}
\sum^{j}_{r=0} \frac{(G+\ell + r)!}{(r)!(\lambda_j-(j-r))!}\frac{(\nu-1)^{r}}{(j-r)!}\right).\nonumber
\end{multline}
Interchanging the sum and gathering terms on the right hand side of the above equation we conclude that
\begin{multline}
    (\nu - 1)^j  \sum_{\ell = 0}^{L} \left(c^{(g,\nu)}_{\ell} {G + \ell+j   \choose j}
\ \pFq{2}{1}{-j,\lambda_j}{-G-\ell-j}{\tfrac{1}{1-\nu}}\right)\\ =    \sum^{j}_{r=0} (\nu-1)^{r}\binom{\lambda_j}{j-r} \sum_{\ell = 0}^{L}   c^{(g,\nu)}_{\ell} \binom{G+\ell+r}{r}
.\label{eq:gen diff form}
\end{multline}
Equations \eqref{eq:main-N} and \eqref{eq:main-Ncal} now follow immediately from \eqref{eq:gen diff form}, and equations \eqref{eq:N-ELT-hypergeom} and \eqref{eq:Ncal-ELT-hypergeom}.

\section{Combinatorics of $6$-valent Graphs with Arbitrary Number of Vertices}\label{Hexic section}
In this section we derive the topological expansion of $\mathcal{R}_n$ and $\mathcal{F}_{nN}$ for the hexic weight
\[\mathscr{V}(z;u) = \frac{z^2}{2}+u\frac{z^6}{6},\]
by extending the method presented in \cite{BGM}. We begin with proving the hexic Freud equation
\begin{equation}\label{sextic string} \begin{split}
      x & =  \mathcal{R}_n\bigg(1 + u( \mathcal{R}_{n+2}\mathcal{R}_{n+1} + \mathcal{R}_{n+1}^2 + 2\mathcal{R}_n\mathcal{R}_{n+1} 
      +\mathcal{R}_n^2 + 2\mathcal{R}_n\mathcal{R}_{n-1}\\ &   \qquad  + \mathcal{R}_{n+1}\mathcal{R}_{n-1} + \mathcal{R}_{n-1}^2 + \mathcal{R}_{n-1}\mathcal{R}_{n-2}) \bigg).
\end{split}
\end{equation}

This is a standard proof in orthogonal polynomial theory which we include for completeness. Using integration by parts we find that in the case $\nu=3$,
\begin{eqnarray*}
     nh_{n-1} &=& \int_\Gamma \left( \frac{\dd}{\dd z}\mathcal{P}_n(z)\right) \mathcal{P}_{n-1}(z) e^{-N\mathscr{V}(z)} \dd z,\\
     &=&  -\int_\Gamma \mathcal{P}_n(z) \left( \frac{\dd}{\dd z}\mathcal{P}_{n-1}(z) e^{-N\mathscr{V}(z)}\right) \dd z,\\
     &=&  N \int_\Gamma \left(\frac{\dd}{\dd z}V(z)\right)\mathcal{P}_n(z) \mathcal{P}_{n-1}(z) e^{-N\mathscr{V}(z)} \dd z,\\
      &=&  N \int_\Gamma \left(uz^5+z\right)\mathcal{P}_n(z) \mathcal{P}_{n-1}(z) e^{-N\mathscr{V}(z)} \dd z.
\end{eqnarray*}
We can use equation \eqref{OP rec} (recalling the notation $\ga^2_n \equiv \mathcal{R}_n$ and the fact that $\be_n=0$ for the hexic weight) to calculate the last equality to arrive at
Equation \eqref{sextic string} by letting $x = \frac{n}{N}$.

Note that the form of the RHS of \eqref{sextic string} is dependent on the choice $\nu=3$, and the RHS will become increasingly complicated as $\nu$ becomes larger. Substituting equation \eqref{R_np1 expansion} into \eqref{sextic string} one can determine $r_{2g}$ for as large a $g$ as desired, albeit with increasing effort. Comparing the first two coefficients of $N$ ($N^0$ and $N^{-2}$) yields the equations:
\begin{eqnarray}
    r_0+10ur_0^3 &=& x, \label{r0 cubic}\\
    r_2(1+30ur_0^2) &=& -5ur_0((r_0')^2 + 2r_0r_0'') \label{r2 OG},
\end{eqnarray}
where we remind the reader that the derivative is respect to $x$. Solving \eqref{r0 cubic} we find
\begin{equation}\label{r0 sol}
    r_0 = u^{-\frac{1}{3}}\left( \Big( \frac{x}{20} + \left[\frac{x^2}{400} + \frac{1}{30^3u}\right]^{1/2}\Big)^{1/3} + \Big( \frac{x}{20} - \left[ \frac{x^2}{400} + \frac{1}{30^3u}\right]^{1/2}\Big)^{1/3}  \right).
\end{equation}
As was the case in \cite{BGM} we can derive an explicit expression for $r_0$. We note that obtaining an explicit expression is possible for $\nu=2$ and $\nu=3$. However, since finding an explicit expression for $r_0$ is equivalent to solving an algebraic equation of degree $\nu$, this problem becomes intractable as $\nu$ becomes larger (see equation \eqref{Order zero String}).\\

Solving equation \eqref{r2 OG} for $r_2$ we find
\begin{equation}\label{r2 sol}
    r_2 = \frac{-5ur_0((r_0')^2 + 2r_0r_0'') }{1+30ur_0^2}.
\end{equation}
A formula for $r_{2g}$ can be found inductively for any $g$ by comparing coefficients of $N^{-2g}$ in \eqref{sextic string} for larger and larger $g$. For example, comparing the $N^{-4}$ coefficient we find
\begin{multline}
    r_4 = \frac{1}{12(1+30u(r_0)^2}\big( -360 u r_0 (r_2)^2 - 60 u r_2 (r_0')^2 - 
  120 u r_0 r_0' r_2' - 
  240 u r_0 r_2 r_0'' - \\
  33 u r_0 (r_0'')^2 - 
  120 u (r_0)^2 r_2'' - 
  44 u r_0 r_0'r_0^{(3)} - 22u(r_0)^2r_0^{(4)} \big).
\end{multline}

It remains to find a nice expression for the coefficients of $u^j$ of $r_2(u)$ and $r_4(u)$. We will detail the process for $r_2(u)$ which can then be generalized to $r_4(u)$. Our approach involves the same techniques as was used in \cite{BD} and \cite{Ercolani-McLaughlin-Pierce}.

\begin{lemma}\label{explicit r2}
For the hexic weight, the $N^{-2}$ coefficient of $\mathcal{R}_n$ has the series expansion,
\begin{equation}
    r_{2} = \sum_{j=1}^\infty c_{2,j} u_j, 
\end{equation}
where 
\begin{eqnarray*}
c_{2,1} &=& \frac{x}{2},\\
    c_{2,j\geq2} &=& \frac{(-10)^j}{2}x^{2j-1}\bigg( 10 {3j \choose j-2} {}_2F_1(3,2-j,3+2j,-2)\\&& \, +{3j \choose j-1} {}_2F_1(3,1-j,2+2j,-2)\bigg).
\end{eqnarray*}
\end{lemma}
\begin{proof}
    Following the arguments presented in Appendix \ref{0 expansion} we find that
    \begin{eqnarray}
        c_{2,j} &=& \frac{1}{2\pi i}\oint \frac{r_2}{u^{j+1}} \dd u,\\
        &=& \left(-10\right)^j\frac{1}{2\pi i} \oint \frac{r_2}{r_0}\frac{\left(x+z \right)^{3j}}{z^{j+1}} \left( x  -2z\right) \dd z,\label{cj sub1}
    \end{eqnarray}
    where $z = r_0-x$ and the integral is about $z=0$. Taking the $x$ derivative of \eqref{r0 cubic} allows us to write $r_0'$ and $r_0''$ in terms of $r_0$ which then allows us to express equation \eqref{r2 sol} as
    \begin{equation}
        r_2 = r_0\frac{(9x-10r_0)(x-r_0)}{2(3x-2r_0)^4}.
    \end{equation}
    We substitute this expression for $r_2$ into \eqref{cj sub1} and change variables from $r_0$ to $z$ to find
    \begin{equation}
        c_{2,j} = \left(-10\right)^j\frac{1}{2\pi i} \oint \frac{z(x+z)^{3j}(x+10z)}{2(x-2z)^3z^{j+1}} \dd z.\label{cj sub2}
    \end{equation}
    Equation \eqref{cj sub2} can be explicitly evaluated using \cite[Equation 15.6.2]{NIST:DLMF}. This provides an explicit expression for $c_{2,j}$.
\end{proof}
We can repeat the arguments presented in Lemma \ref{explicit r2} to determine the series expansion of $r_4$. The steps are identical but with more algebra involved. We used Mathematica to deal with the increasing algebraic steps required. The results are presented in the following theorem.

\begin{theorem}\label{hexic rec theorem} Consider the system of orthogonal polynomials \eqref{OP} with respect to the weight
\[ \exp\left(-N \left( \frac{z^2}{2} + \frac{u z^6}{6} \right) \right), \]
 and the associated recurrence relation \eqref{OP rec}. The coefficients \( r_0 \), \( r_2 \), and \( r_4 \) in the corresponding topological expansion \eqref{Asymp ga_n^2} are given by:

\begin{equation*}
    r_{0} = \sum_{j=1}^\infty \beta_{0,j} u^j, \qquad \text{and,} \qquad r_{2} = \sum_{j=1}^\infty \beta_{2,j} u^j, \qquad \text{and,} \qquad r_4 = \sum_{j=2}^\infty \beta_{4,j}u^j,
\end{equation*}
where
\begin{equation*}
    \beta_{0,j} = \left(-10\right)^j\frac{(3j)!}{j!(2j+1)!}x^{2j+1},
\end{equation*}
 \begin{eqnarray*}
  \beta_{2,1}  &=& -5x, \\
 \beta_{2,j\geq2}  &=&  \frac{(-10)^j}{2}x^{2j-1}\bigg( 10 {3j \choose j-2} 
    \pFq{2}{1}{3, 2-j}{3+2j}{-2} + {3j \choose j-1} \pFq{2}{1}{3, 1-j}{2+2j}{-2} \bigg),
\end{eqnarray*}
and,
\begin{flalign*}
 \beta_{4,2}  & =  295x, & \\
 \beta_{4,3}  &=  -274300 x^{3}, & \\
 \beta_{4,4}  &= 81777000x^{5}, & \\
 \beta_{4,j\geq5}  & =  \frac{(-10)^j}{20}x^{2j-3}\Bigg[ 59 { 3j \choose j-2}\pFq{2}{1}{8, 2-j}{3+2j}{-2} + 4011{ 3j \choose j-3} \pFq{2}{1}{8, 3-j}{4+2j}{-2}  \\
    &\, + 27528 { 3j \choose j-4} \pFq{2}{1}{8, 4-j}{5+2j}{-2}  + 34268{ 3j \choose j-5} \pFq{2}{1}{8, 5-j}{6+2j}{-2}\Bigg], & \nonumber
\end{flalign*}
\noindent and $\pFq{2}{1}{a, b}{c}{z}$ is Gauss' hypergeometric function \cite[Section 15.2]{NIST:DLMF}. \end{theorem}
\begin{remark} \normalfont
    Note that the integral representation in equation \eqref{cj sub2} is how one arrives at the results found in \cite{ELTMe}, where their results hold in a more general setting. Using \cite[Equation 15.6.2]{NIST:DLMF} one can explicitly compute integrals of this form (this is also how one derives the counts of graphs from the results presented in \cite[Section 5]{Ercolani-McLaughlin-Pierce}).
\end{remark}
\subsection{First Order Derivative of the Free Energy}
One can also derive a first-order differential equation for $F_{nN}(\sigma)$ introduced in \eqref{sigma free energy}. However, its form depends on $\nu$, and in this section we will detail the process for $\nu=3$. The first-order differential equation can be useful for deriving explicit expressions of $F_{nN}(\sigma)$ as solving an inhomogeneous first-order differential equation is considerably easier than solving an inhomogeneous second-order differential equation \eqref{the second f}. On the other hand, the first-order differential equation leads to a complicated expression which depends on $\nu$. Furthermore, if one is only concerned with representing $\mathcal{F}_{nN}(u)$ as a power series in $u$, the second-order differential equation is a much easier expression to use.
\begin{lemma}
The hexic $\sigma$-free energy satisfies the following first order differential equation in $\sigma$,
\begin{equation}\label{the first f}
     \frac{\ddd F_{nN}}{\ddd \sigma}
    = -\frac{N^2}{2n^2}\Big(\frac{n}{N}R_n+R_{n+1}R_nR_{n-1}(R_{n+2}+R_{n+1}+R_{n}+R_{n-1}+R_{n-2})\Big).
\end{equation} 
\end{lemma}
\begin{proof}
    Differentiating equation \eqref{F def} with respect to $\sigma$ we find
\begin{eqnarray*}
    \frac{\ddd F_{nN}}{\ddd \sigma} &=& \frac{1}{n^2} \left(\sum^{n-1}_{k=0}  \frac{1}{h_k}\frac{\partial h_k}{\partial \sigma} \right)=  -\frac{N}{2n^2} \left(\sum^{n-1}_{k=0}  \frac{1}{h_k}\int_\Gamma \zeta^2P_k(\zeta)^2e^{-NV(\zeta )}\dd \zeta  \right)\\
    &=& -\frac{N}{2n^2} \int_\Gamma \zeta ^2\left(\sum_{k=0}^{n-1}\frac{P_k(\zeta )^2}{h_k}\right)e^{-NV(\zeta )} \dd \zeta  \\
    &=& -\frac{N}{2n^2} \int_\Gamma \zeta ^2\left(\frac{P_n(\zeta )'P_{n-1}(\zeta ) - P_{n-1}(\zeta )'P_{n}(\zeta )}{h_{n-1}}\right)e^{-NV(\zeta )}\dd \zeta  ,
\end{eqnarray*}
where we have used the Christoffel-Darboux formula to arrive at the last equality and $f(\zeta )'$ is shorthand notation for $\frac{d f(\zeta )}{d\zeta }$. Using the orthogonality condition of the polynomials and repeated application of equation \eqref{three term} we find that
\begin{equation}\label{to be cont}
\begin{aligned}
\frac{\dd F_{nN}}{\dd \sigma}
&= \frac{N(n-1)}{2n^{2}}\,R_n
 - \frac{N}{2n^{2}} \int_{\Gamma} \frac{\zeta^{2}}{h_{n-1}}\,P_n' P_{n-1}\,e^{-NV}\,\dd\zeta \\
&= \frac{N(n-1)}{2n^{2}}\,R_n
  - \frac{N}{2n^{2} h_{n-1}}
   \int_{\Gamma} P_n'\!\bigg(P_{n+1} + (R_n+R_{n-1})P_{n-1} + R_{n-1}R_{n-2}P_{n-3}\bigg)
   e^{-NV}\,\dd\zeta \\
&= \frac{N(n-1)}{2n^{2}}\,R_n
 - \frac{N}{2n}(R_n+R_{n-1})
 - \frac{N}{2n^{2} h_{n-1}}\,R_{n-1}R_{n-2}
   \int_{\Gamma} P_n' P_{n-3}\,e^{-NV}\,\dd\zeta \\
&= -\frac{N}{2n^{2}}(R_n+nR_{n-1})
 - \frac{N}{2n^{2} h_{n-3}} \int_{\Gamma} P_n' P_{n-3}\,e^{-NV}\,\dd\zeta .
\end{aligned}
\end{equation}
We now take a brief detour to prove an identity required to proceed.  First observe that the derivative of $P_n(\zeta )$ with respect to $\zeta$ can be written as
\begin{equation}
    P_n' = nP_{n-1} + A_nP_{n-3} + \mathcal{O}(\zeta ^{n-5}),\label{diff identity}
\end{equation}
whenever the orthogonality weight is an even function. Let us differentiate both sides of equation \eqref{three term} with respect to $\zeta $ and apply \eqref{diff identity} and \eqref{three term} so that both sides are written only in terms of the basis of orthogonal polynomials. Comparing the coefficient of $P_{n-2}$ gives the identity
\begin{equation}
    nR_{n-1} + A_n = A_{n+1} + (n-1)R_n.\label{needed ident}
\end{equation}
Applying equation \eqref{needed ident} to  \eqref{to be cont} we find
\begin{eqnarray*}
    \frac{\ddd F_{nN}(\sigma)}{\ddd \sigma} &=& -\frac{N}{2n^2}\Big( R_n+nR_{n-1} + A_n \Big), \\
    &=& -\frac{N}{2n^2}\Big( nR_{n} + A_{n+1} \Big), \\
    &=& -\frac{N}{2n^2}\Big( nR_{n} + \frac{1}{h_{n-2}}\int_\Gamma P_{n+1}'P_{n-2} e^{-NV} \dd\zeta \Big).
\end{eqnarray*}
Applying integration by parts and using orthogonality we obtain
\begin{eqnarray*}
    \frac{\ddd F_{nN}(\sigma)}{\ddd \sigma} &=&  -\frac{N}{2n^2}\left(nR_n+\frac{N}{h_{n-2}}\int_\Gamma P_{n+1}P_{n-2}V' e^{-NV} \dd\zeta \right),\\
     &=& -\frac{N}{2n^2}\left(nR_n+\frac{N}{h_{n-2}}\int_\Gamma P_{n+1}P_{n-2}\zeta ^5 e^{-NV} \dd\zeta \right).
\end{eqnarray*}
Using equation \eqref{three term} we can express $\zeta ^5P_{n+1}$ in terms of the basis of orthogonal polynomials $\{P_{k}(\zeta )\}_{n-4}^{n+6}$. This leads us to equation \eqref{the first f}.
\end{proof}
\begin{remark}
\normalfont    The derivation detailed above holds for general $\nu$ up to \eqref{to be cont}. This simplifies the derivation of the first order Toda equation in the quartic case presented in \cite{BGM}.
\end{remark}

Equation \eqref{the first f} can be written in terms of $u$ using equations \eqref{uv relation}, \eqref{free relation}, \eqref{r relation} and \eqref{the second f} as,
\begin{equation}\nonumber
    \frac{\ddd\mathcal{F}_{nN}}{\ddd u} = \frac{\mathcal{R}_n}{6ux^2}\Big(x+u\mathcal{R}_{n+1}\mathcal{R}_{n-1}(\mathcal{R}_{n+2}+\mathcal{R}_{n+1}+\mathcal{R}_{n}+\mathcal{R}_{n-1}+\mathcal{R}_{n-2})\Big) - \frac{1}{6u}.
\end{equation}
In Theorem \ref{hexic free theorem}  we provide an explicit formula for the first three coefficients $N^{0}$, $N^{-2}$ and $N^{-4}$ of the free energy for hexic weights. The derivation follows from the same arguments as were used in the proof of Lemma \ref{explicit r2}.
\begin{theorem}\label{hexic free theorem}
Consider the eigenvalue partition function \eqref{u free energy} with respect to the weight
\[ \exp\left(-N \left( \frac{z^2}{2} + \frac{u z^6}{6} \right) \right), \]
 and the associated free energy  $\mathcal{F}_{nN}$ given by \eqref{free energy}. The coefficients \( f_0 \), \( f_2 \), and \( f_4 \) in the corresponding topological expansion \eqref{top exp free energy} are given by:
\begin{equation*}
    f_0 = \sum_{j=1}^\infty \alpha_{0,j}u^j, \qquad and, \qquad f_2 = \sum_{j=1}^\infty \alpha_{2,j}u^j, \qquad and, \qquad f_4 = \sum_{j=1}^\infty \alpha_{4,j}u^j,
\end{equation*}
where, {\small
\begin{equation*}
    \alpha_{0,j} = (-10)^j\frac{{ 3j+1\choose j}-2{3j+1 \choose j-1}}{6j(3j+1)}x^{2j},
\qquad
    \alpha_{2,j} = (-10)^j\frac{2{2+3j \choose j-1}\pFq{2}{1}{3, 1-j}{4+2j}{-2}}{3j(3j+1)}x^{2j-2},
\end{equation*}}
and,
{\small \begin{flalign*}
 \al_{4,2}  &= \frac{265}{4},& \\
 \al_{4,3}  &= -\frac{40025}{3}x^2,& \\
 \al_{4,4}  &= 1736625x^4, & \\
 \al_{4,5}  &= -187387500 x^6,& \\
 \al_{4,j}  &=  \frac{(-10)^jx^{2j-4}}{40j(3j+1)} \Bigg[  371 {3j\choose j-2}\pFq{2}{1}{8, 2-j}{3+2j}{-2}  + 6735 {3j\choose j-3}\pFq{2}{1}{8, 3-j}{4+2j}{-2}  \\& + 23496 {3j\choose j-4}\pFq{2}{1}{8, 4-j}{5+2j}{-2}  + 25004 {3j\choose j-5}\pFq{2}{1}{8, 5-j}{6+2j}{-2} \nonumber \\& +7872{3j\choose j-6}\pFq{2}{1}{8, 6-j}{7+2j}{-2} \Bigg],& \nonumber
\end{flalign*}} for all $j\geq6$.
\end{theorem}

\subsection{Explicit Formulae for \( \mathscr{N}_{g}(6,j) \) and \( \mathcal{N}_{g}(6,j) \) }\label{new results: Hexic Case}

Theorems \ref{thm hexic counts} and \ref{thm hexic 2-legged counts} provide combinatorial results for $6$-valent graphs embedded on Riemann surfaces and are essentially corollaries of Theorems \ref{hexic free theorem} and \ref{hexic rec theorem} respectively in view of the formulae \eqref{numbers and free energy} and \eqref{2legged counts}.

\begin{theorem}\label{thm hexic counts}
	Let $\mathscr{N}_g(6,j)$ be the number of connected labeled 6-valent graphs with $j$ vertices which are realizable on a closed Riemann surface of minimal genus $g$ , (as an example recall the graphs (a) and (b) in Figure \ref{fig:intro}). We have 

{\small \begin{flalign*}
    \mathscr{N}_0(6,j) &= 60^{j} \cdot \frac{(3j-1)!}{(2j+2)!}, \qquad j \in \N, & \\
    \mathscr{N}_1(6,j) &= \frac{40}{3j+1}(j-1)!(60)^{j-1}  {2+3j \choose j-1}\pFq{2}{1}{3, 1-j}{4+2j}{-2} , \qquad j \in \N, & \\
    \mathscr{N}_2(6,j)   &=  \frac{3}{2(3j+1)} (j-1)!(60)^{j-1}  \Bigg[  371 {3j\choose j-2}\pFq{2}{1}{8, 2-j}{3+2j}{-2} \\ &
    + 6735 {3j\choose j-3}\pFq{2}{1}{8, 3-j}{4+2j}{-2} + 23496 {3j\choose j-4}\pFq{2}{1}{8, 4-j}{5+2j}{-2} \\
     & + 25004 {3j\choose j-5}\pFq{2}{1}{8, 5-j}{6+2j}{-2} + 7872{3j\choose j-6}\pFq{2}{1}{8, 6-j}{7+2j}{-2} \Bigg], \nonumber &
\end{flalign*}} for all $j \geq 6$. For $g=2$ and $1\leq j \leq 5$ the counts are given by: $\mathscr{N}_2(6,1)=0$, $\mathscr{N}_2(6,2)=4770$, $\mathscr{N}_2(6,3) = 17290800$, $\mathscr{N}_2(6,4) = 54015984\times10^3$, and $\mathscr{N}_2(6,5)   = 174855024 \times 10^6$.
\end{theorem}

\begin{remark} \normalfont 
Note that the formulae for $\mathscr{N}_0(6,j)$ and $\mathscr{N}_1(6,j)$ agree with \eqref{numbers sphere} and \eqref{eq:e1_counts} for $\nu=3$. The formula for $\mathscr{N}_2(6,j)$ has not appeared before in the literature, but is simply an evaluation of the much more general formula provided in Theorem \ref{thm:main-N} with the explicit expressions for $b^{(2,\nu)}_{\ell}$, $\ell=0,1,2,3$.
\end{remark}
\begin{theorem}\label{thm hexic 2-legged counts}
	Let $\mathcal{N}_g(6,j)$ be the number of 2-legged connected labeled 6-valent graphs with $j$ vertices which are realizable on a closed Riemann surface of minimal genus $g$, (as an example recall the graphs (c) and (d) in Figure \ref{fig:intro}). We have
    {\small \begin{flalign*}
\mathcal{N}_0(6,j) &= \left(60\right)^j\frac{(3j)!}{(2j+1)!}, \qquad j \in \N, & \\
    \mathcal{N}_1(6,j) &= \frac{j!(60)^j}{2}\left[ 10 {3j \choose j-2} 
    \pFq{2}{1}{3, 2-j}{3+2j}{-2}\,+{3j \choose j-1} \pFq{2}{1}{3, 1-j}{2+2j}{-2} \right], & \\
    \mathcal{N}_2(6,j)   &=  \frac{j!(60)^j}{20}\Bigg[ 59 { 3j \choose j-2}\pFq{2}{1}{8, 2-j}{3+2j}{-2} + 4011{ 3j \choose j-3} \pFq{2}{1}{8, 3-j}{4+2j}{-2}  \\ &
      + 34268 { 3j \choose j-5} \pFq{2}{1}{8, 4-j}{5+2j}{-2}+27528 { 3j \choose j-4} \pFq{2}{1}{8, 5-j}{6+2j}{-2}\Bigg], \nonumber          
    \end{flalign*}}

\noindent where the expression for $\mathcal{N}_1(6,j)$ holds for all $j \geq 2$ and the expression for $\mathcal{N}_2(6,j)$ holds for all $j \geq 5$. For $g=1$ and $ j = 1$ we have $\mathcal{N}_1(6,1)=30$\footnote{See Figure \ref{fig2}.}, while for $g=2$ and $1\leq j \leq 4$ the counts are given by: $\mathcal{N}_2(6,1)=0$, $\mathcal{N}_2(6,2)=21240$, $\mathcal{N}_2(6,3) = 355492800$, and $\mathcal{N}_2(6,4) = 2543591808\times10^3$.
\end{theorem}

\section{Series Expansion of $r_0(x;u)$}\label{0 expansion}
For completeness we derive equation \eqref{beta 0 def} which is a known result in the literature \cite[Theorem 2.1]{Ercolani-McLaughlin-Pierce}. 
\begin{equation}\label{beta 0 def}
    \beta_{0,j} = \left(-{2\nu-1 \choose \nu}\right)^j\frac{(j\nu)!}{j!(j(\nu-1)+1)!}x^{j(\nu-1)+1}.
\end{equation}
As noted in Remark \ref{string remark} we are readily able to observe that $r_0(x;u)$ satisfies the algebraic equation
\begin{equation}\label{string appendix}
    r_0 + u{2\nu-1 \choose \nu}r_0^\nu = x.
\end{equation}
Taking the derivative of \eqref{string appendix} with respect to $u$ we find that
\begin{equation}
    \frac{\dd r_0}{\dd u} = - \frac{{2\nu-1 \choose \nu}r_0^\nu}{1+{2\nu-1 \choose \nu}u\nu r_{0}^{\nu-1}}.
\end{equation}
We now repeat the analysis carried out in \cite{BD} to determine a closed form expression for the coefficients of the power series \eqref{r_2g series} of $r_0$. By the Cauchy residue theorem
\begin{equation}
    \beta_{0,j} = \frac{1}{2\pi i}\oint \frac{r_0}{u^{j+1}} \dd u.
\end{equation}
Thus,
\begin{eqnarray*}
    \beta_{0,j} &=& \frac{1}{2\pi i} \oint \left( \frac{r_0}{u^{j+1}} \right) \left( \frac{\dd u}{\dd r_0} \right)\dd r_0,\\
    &=& \frac{1}{2\pi i} \oint \left( \frac{r_0\left({2\nu-1 \choose \nu}r_0^\nu \right)^{j+1}}{(x-r_0)^{j+1}}\right) \left( \frac{1+{2\nu-1 \choose \nu}u\nu r_0^{\nu-1}}{{2\nu-1 \choose \nu}r_0^\nu}\right)\dd r_0,\\
    &=& (-1)^j\frac{1}{2\pi i} \oint \frac{\left({2\nu-1 \choose \nu}r_0^\nu \right)^{j}}{(r_0-x)^{j+1}} \left( r_0+{2\nu-1 \choose \nu}u\nu r_0^{\nu}\right) \dd r_0.
\end{eqnarray*}
Note that the original contour integral was around $u=0$ and after the change of variables the integral is now around $r_0=x$. Next, we make the new change of variables $z=r_{0}-x$, in the variable $z$ the above integral becomes
\begin{eqnarray*}
    \beta_{0,j} &=& \left(-{2\nu-1 \choose \nu}\right)^j\frac{1}{2\pi i} \oint \frac{\left(x+z \right)^{j\nu}}{z^{j+1}} \left( x + z(1-\nu)\right) \dd z,\\
    &=& \frac{\left(-{2\nu-1 \choose \nu}\right)^j}{2\pi i} \left( \oint \frac{\left(x+z \right)^{j\nu}}{z^{j+1}}xdz + \oint \frac{\left(x+z \right)^{j\nu}}{z^{j}}(1-\nu) \dd z \right),\\
    &=&\left(-{2\nu-1 \choose \nu}\right)^j \left( {j\nu \choose j} + (1-\nu) {j\nu \choose j-1} \right)x^{j(\nu-1)+1},\\
    &=& \left(-{2\nu-1 \choose \nu}\right)^j\frac{(j\nu)!}{j!(j(\nu-1)+1)!}x^{j(\nu-1)+1}.
\end{eqnarray*}

\section{Freud Equations}\label{Freud section}
Lemma \ref{Freudian} is well known in the literature, however we include a proof for completeness. This proof of Lemma \ref{Freudian} follows from the binomial expansion of $(1+x)^k$. First, observe that in order to prove the two properties of Freud equations presented in Lemma \ref{Freudian} we are only interested in the number of terms and the degree of the product of recurrence coefficients. The three term recurrence relation is given by
\[ z \mathcal{P}_n(z) = \mathcal{P}_{n+1}(z) + \mathcal{R}_n\mathcal{P}_{n-1}(z).\]
We are interested in the coefficient of the $\mathcal{P}_{n-1}(z)$ term in the expansion of $z^{2\nu-1}\mathcal{P}_n(z)$, this is what constitutes the Freud equation \cite{Maggy}. By repeated application of the recurrence relation one sees that, concerning the two properties we are interested in, this is directly analogous to the $\nu$ coefficient of $(1+x)^{2\nu-1}$. The result follows immediately. 

\section{Complementary Graph Counts Necessary to Prove Theorems \ref{thm:main-N} and \ref{thm:main-Ncal}}\label{sec further eqns} 
Below we add to the results of Theorems \ref{thm 2legged combinatorics} and \ref{thm combinatorics} and include further graph counts which are necessary to prove Theorems \ref{thm:main-N} and \ref{thm:main-Ncal}.\\

For fixed small values of \( g \) and \( j \), closed-form expressions for \( \mathcal{N}_{g}(2\nu,j) \) are given by $\mathcal{N}_{g}(2\nu,j)=c_\nu^j Q_{g,j}(\nu) $ where the explicit polynomials \( Q_{g,j}(\nu) \)  are defined below.
{\small
\begin{flalign*}
     Q_{2,4}(\nu) &= \frac{1}{45}(2\nu - 3)\bigg(7148\nu^6 - 38626\nu^5 + 80669\nu^4 - 
     82165\nu^3 + 42170\nu^2  \\& - 10072\nu + 840\bigg)\prod_{i=0}^1(\nu-i), & \\
    Q_{2,5}(\nu) &= \frac{5}{288}(5\nu - 7)\bigg(112625\nu^7 - 635499\nu^6 + 1441299\nu^5 - 
     1686937\nu^4 + 1086700\nu^3 \\& - 379100\nu^2 + 64800\nu - 4032\bigg)\prod_{i=0}^1(\nu-i),&\\
      Q_{2,6}(\nu) &= \frac{3}{10}(3\nu - 4)\bigg(344260\nu^8 - 2051842\nu^7 + 
     5062412\nu^6 - 6707321\nu^5 + 5175010\nu^4 \\&- 2355053\nu^3 + 
     608238\nu^2 - 79744\nu + 3920\bigg)\prod_{i=0}^1(\nu-i), & \\
     Q_{3,4}(\nu) &= \frac{1}{5670}\bigg(2207696\nu^9 - 23059170\nu^8 + 
      103014219\nu^7 - 257038215\nu^6 + 392010135\nu^5 \\& - 
      375285093\nu^4 + 222463588\nu^3 - 77228952\nu^2 \\ & + 13855392\nu - 
      937440\bigg)\prod_{i=0}^2(\nu-i), &\\
           Q_{3,5}(\nu) &=\frac{5}{72576}(5\nu - 9)\bigg(62522399\nu^9 - 515187180\nu^8 + 
      1815830526\nu^7 - 3570372984\nu^6 \\&+ 4281265095\nu^5 - 
      3213153660\nu^4 + 1489031548\nu^3 \\ & - 403491072\nu^2 + 
      56546496\nu  - 2999808\bigg)\prod_{i=0}^2(\nu-i), &\\
     Q_{3,6}(\nu) &= \frac{1}{420}(3\nu - 5)\bigg(153801520\nu^{11} - 1577943896\nu^{10} + 
     7116498472\nu^9 - 18554100415\nu^8 \\&+ 30928752050\nu^7 - 
     34413210643\nu^6 + 25892235846\nu^5 - 13053109770\nu^4 \\&+ 
     4269785220\nu^3 - 849274416\nu^2 + 90319392\nu - 3749760\bigg)\prod_{i=0}^1(\nu-i), &\\
     Q_{3,7}(\nu) &=\frac{7}{51840}(7\nu - 11)\bigg(57762660809\nu^{12} - 601237736085\nu^{11} + 
     2780241259726\nu^{10} \\&- 7528766160606\nu^9 + 13246913167689\nu^8 - 
     15881960187189\nu^7 \\ & + 13230141322096\nu^6  - 7662897894984\nu^5  + 
     3036359472752\nu^4\\& - 793729924176\nu^3 + 127961180928\nu^2 - 
     11172591360\nu + 385689600\bigg)\prod_{i=0}^1(\nu-i), &
     \end{flalign*}
\begin{flalign*}
          Q_{3,8}(\nu) &= \frac{16}{2835}(2\nu - 3)\bigg(240990999704\nu^{13} - 
     2564120927116\nu^{12} + 12230680621318\nu^{11} \\&- 
     34537507809530\nu^{10} + 64216395779166\nu^9  - 
     82712176120473\nu^8 + 75592119041851\nu^7  \\&- 
     49368109659701\nu^6 + 22889376111695\nu^5 - 7380035573626\nu^4 + 
     1591149962856\nu^3 \\ & - 214064248464\nu^2 + 15766081920\nu - 
     464032800\bigg)\prod_{i=0}^1(\nu-i), & \\
     Q_{3,9}(\nu) &= \frac{9}{4480}(9\nu - 13)\bigg(7633080358851\nu^{14} - 
     83402337357060\nu^{13} + 411753316768359\nu^{12} \\& - 
     1214643242298940\nu^{11} + 2385589025005169\nu^{10} - 
     3289824665249788\nu^9 \\&+ 3273296349535789\nu^8 - 
     2376973084782212\nu^7 + 1259442599876392\nu^6 \\& - 
     481497757955712\nu^5  + 129710174087952\nu^4 - 
     23628195523008\nu^3 \\&+ 2712360722688\nu^2 - 172021294080\nu + 
     4399718400\bigg)\prod_{i=0}^1(\nu-i), & \\
 Q_{4,4}(\nu) &= \frac{1}{1360800}(2\nu - 5)\bigg(260145536\nu^{11} - 
     3852856336\nu^{10} + 25119085320\nu^9 \\&
     - 94893927618\nu^8 + 
     229949004225\nu^7 - 373436213661\nu^6 \\ & + 411954757417\nu^5  - 
     305856912485\nu^4 + 147851057610\nu^3 \\&- 43504612200\nu^2 + 
     6825425472\nu - 414771840\bigg)\prod_{i=0}^2(\nu-i), & \\
Q_{4,5}(\nu) &= \frac{1}{3483648}(5\nu - 11)\bigg(26696728923\nu^{12} - 
     370952050974\nu^{11} + 2294541589387\nu^{10}\\& - 8333238528990\nu^9 + 
     19725191345949\nu^8  - 31923036291330\nu^7 \\& + 
     36022272022041\nu^6 - 28353612535626\nu^5 + 
     15306244304900\nu^4 \\ &- 5457861243000\nu^3 + 1199435076000\nu^2 - 
     142315315200\nu \\&+ 6636349440\bigg)\prod_{i=0}^2(\nu-i), & \\
Q_{4,6}(\nu) &= \frac{1}{33600}\bigg(106291233600\nu^{14} - 1641228544800\nu^{13} + 
     11489170902012\nu^{12} \\&- 48230829311284\nu^{11} + 
     135310877873729\nu^{10} - 267575283754675\nu^9 \\ & + 
     383229663323086\nu^8 - 402055567761002\nu^7 + 
     308782996266697\nu^6 \\&- 171569608958355\nu^5 + 
     67291398444732\nu^4 - 17856610032924\nu^3 \\ &+ 2983143643344\nu^2 - 
     274552796160\nu + 10138867200\bigg)\prod_{i=0}^2(\nu-i),& \\    
     Q_{4,7}(\nu) &= \frac{7}{12441600}(7\nu - 13)\bigg(59827528284865\nu^{14} - 
     792755101620269\nu^{13} + 4762127989292963\nu^{12} \\&- 
     17148907697572141\nu^{11} + 41246386612822161\nu^{10} - 
     69867418438924707\nu^9 \\&+ 85625003322460889\nu^8 - 
     76771321915272223\nu^7 + 50320568155406698\nu^6 \\&- 
     23829127833023204\nu^5 + 7954999368562248\nu^4 - 
     1794871514727936\nu^3 \\&+ 254788084469376\nu^2 - 
     19924229560320\nu + 625712947200\bigg)\prod_{i=0}^2(\nu-i),& 
     \end{flalign*}
\begin{flalign*}
Q_{4,8}(\nu) &= \frac{1}{42525}(4\nu - 7)\bigg(176898841310688\nu^{16} - 
     2693497251490416\nu^{15}  + 18811324769190752\nu^{14} \\&- 
     79856217753482200\nu^{13} + 230183164994132056\nu^{12}  - 
     476637233352493472\nu^{11} \\&+ 731499443164185318\nu^{10} - 
     846140169372817956\nu^9 + 742806644356904671\nu^8 \\ & - 
     494377579762009877\nu^7  + 247310205535113371\nu^6- 
     91409482078529839\nu^5 \\&+ 24270884829919140\nu^4 - 
     4427423431851420\nu^3 + 515706668847504\nu^2  \\&- 
     33559861576320\nu + 889685596800\bigg)\prod_{i=0}^1(\nu-i), & \\
Q_{4,9}(\nu) &= \frac{9}{358400}(3\nu - 5)\bigg(15226246439849967\nu^{17} - 
     233157469299715206\nu^{16} \\& + 1645659359908858504\nu^{15}  - 
     7099851100561009676\nu^{14}+ 20933861891087217190\nu^{13} \\ & - 
     44678011405022416136\nu^{12} + 71310247178220382424\nu^{11} - 
     86715146850133389892\nu^{10} \\ & + 81086399883902809283\nu^9 - 
     58428598096113549618\nu^8+ 32304978976301131416\nu^7 \\& - 
     13556905700761904080\nu^6 + 4238229641600620080\nu^5 - 
     958623064855007520\nu^4\\& + 149921391481410816\nu^3 - 
     15061580882652672\nu^2 \\ &+ 850381337272320\nu - 
     19682920857600\bigg)\prod_{i=0}^1(\nu-i), & \\
Q_{4,10}(\nu) & =\frac{5}{27216}(5\nu - 8)\bigg(85562694562591904\nu^{18} - 
     1324327958855284160\nu^{17} \\&+ 9489582662393397880\nu^{16} - 
     41771977818193676192\nu^{15}+ 126386220800651384956\nu^{14} \\ & - 
     278635083361065737240\nu^{13} + 462965053710797027613\nu^{12} - 
     591451799686117800136\nu^{11} \\ &+ 587429197232582640311\nu^{10} - 
     455605936414839442102\nu^9 + 275628018515897406119\nu^8\\& - 
     129214946680960956196\nu^7 + 46358891319850977757\nu^6 - 
     12478492347822559134\nu^5 \\&+ 2445372448774902900\nu^4 - 
     333234184824131400\nu^3 + 29327442415777440\nu^2\\ & - 
     1458268956910080\nu + 29893436052480\bigg)\prod_{i=0}^1(\nu-i), & \\
Q_{4,11}(\nu) &=\frac{121}{87091200}(11\nu - 17)\bigg(354316216480761305925\nu^{19} - 
     5562857674691886437505\nu^{18} \\&+ 40594391707077586220794\nu^{17} - 
     182794537273475259575828\nu^{16} \\&+ 
     568651781959247875285078\nu^{15} - 
     1296530542543012381738790\nu^{14} \\ &+ 
     2242958540483603888345008\nu^{13} - 
     3006984730064022233375036\nu^{12} \\&+ 
     3163217844197388143476541\nu^{11} - 
     2627392544910792009694225\nu^{10}\\& + 
     1725172320938626212605822\nu^9 - 892393102172221303570984\nu^8 \\ & + 
     360687187175899070572064\nu^7 - 112357655330246151693520\nu^6 \\&+ 
     26421675360918138122976\nu^5 - 4548476866591685693952\nu^4  \\&+ 
     547259187982084756992\nu^3 - 42729584823703388160\nu^2 \\ &+ 
     1893977522609356800\nu - 34785089224704000\bigg)\prod_{i=0}^1(\nu-i), 
     \end{flalign*}
\begin{flalign*}
     Q_{4,12}(\nu) &=\frac{18}{175}(2\nu - 3)\bigg(1837389089069015040\nu^{20} - 
     29335825906712036736\nu^{19}  \\&+ 218482585015010236144\nu^{18} - 
     1008124838237212680068\nu^{17} \\ &+ 3228291236767125505622\nu^{16} - 
     7616016312301978575810\nu^{15} \\&+ 13713633029886554159198\nu^{14} - 
     19266820676488675236606\nu^{13} \\ & + 21409294239729359555824\nu^{12} - 
     18960352018928272710554\nu^{11}  \\&+ 13422333631931283792977\nu^{10} - 
     7586743646823016231884\nu^9   \\ & +3406451051822739448326\nu^8 - 
     1203540294803764991418\nu^7  \\&+ 329739294408248025601\nu^6 - 
     68567041663186648492\nu^5 \\ &+ 10488956102033829188\nu^4 - 
     1126520901427835232\nu^3 \\&+ 78856158406678080\nu^2 - 
     3147138574675200\nu  \\&+ 52282758528000\bigg)\prod_{i=0}^1(\nu-i). &
\end{flalign*}}

For fixed small values of \( g >0 \) and \( j \), closed-form expressions for \( \mathscr{N}_{g}(2\nu,j) \) are given by $\mathscr{N}_{g}(2\nu,j) = c_{\nu}^jS_{g,j}(\nu)$ where the explicit polynomials \( S_{g,j}(\nu) \)  are defined below. Note that in this section we are using the notation $\mathscr{N}_{g}(2\nu,j) = c_{\nu}^jS_{g,j}(\nu)$ not, $\mathscr{N}_{g}(2\nu,j) = C_{\nu}^jS_{g,j}(\nu)$ as used in Theorem \ref{thm 2legged combinatorics}, where the constants $c_\nu$ and $C_\nu$ (Catalan number) are related by $c_\nu = \nu(\nu+1)C_\nu$.
{\small
\begin{eqnarray*}
    S_{2,4}(\nu) &=& \frac{1}{360}(\nu - 1)\bigg(7148\nu^6 - 32946\nu^5 + 57857\nu^4 - 48477\nu^3  + 19778\nu^2 - 3504\nu  + 180\bigg),\\
    S_{3,4}(\nu) &=& \frac{1}{90720}(\nu - 1)(\nu - 2)\bigg(2207696\nu^8 - 16242050\nu^7 + 49364471\nu^6  - 79932137\nu^5 \\
    &&+ 74043341\nu^4 - 39060533\nu^3 + 10921512\nu^2  - 1335780\nu + 37800\bigg),\\
    S_{3,5}(\nu) &=& \frac{1}{72576}(\nu-1)\bigg(62522399\nu^{10} - 581103853\nu^9 + 2326946286\nu^8  - 5250945186\nu^7  \\
    && + 7329256599\nu^6 - 6532688373\nu^5 + 3701615836\nu^4  - 1282650908\nu^3  \\
    &&+ 248644320\nu^2 - 22098240\nu + 483840\bigg), \\ 
    S_{3,6}(\nu) &=& \frac{1}{5040}(\nu - 1)\bigg(153801520\nu^{11} - 1447813616\nu^{10} + 5955888280\nu^9 - 14058047545\nu^8  \\
    &&+ 21012908900\nu^7 - 20703187408\nu^6 + 13560491070\nu^5  - 5807949975\nu^4 \\&& + 1554253470\nu^3 - 236885256\nu^2 +  16835760\nu - 302400\bigg), \\
        S_{3,7}(\nu) &=& \frac{1}{51840}(\nu - 1)\bigg(57762660809\nu^{12} - 556670693418\nu^{11} + 2372309923585\nu^{10}  \\
    &&- 5886373358850\nu^9 + 9421948239807\nu^8  - 10182012470334\nu^7  \\
    &&+ 7553393392915\nu^6  - 3831963508110\nu^5 
    +1298418268004\nu^4 
    \\
    &&- 279586295688\nu^3  + 34789466880\nu^2   - 2047248000\nu + 31104000\bigg),
    \end{eqnarray*} }
{\small
\begin{eqnarray*} 
   S_{4,4}(\nu) &=& \frac{1}{10886400}(\nu - 1)(\nu - 2)\bigg(260145536\nu^{11} - 3499624976\nu^{10} + 
     20557992264\nu^9 \\ &&  - 69254891538\nu^8 + 147655647081\nu^7 - 
     207277108965\nu^6 + 192941777329\nu^5 \\&& - 116777265325\nu^4 + 
     43634464794\nu^3 - 9043717896\nu^2 + 813468096\nu - 
     11430720\bigg), \\
 S_{4,5}(\nu) &=&\frac{1}{17418240}(\nu - 1)(\nu - 2)\bigg(26696728923\nu^{12} - 339690851474\nu^{11} + 
     1912242628787\nu^{10} \\ && - 6271005358290\nu^9  + 13270755972549\nu^8  - 
     18956834778030\nu^7 + 18565070459121\nu^6 \\ &&- 
     12393974046406\nu^5 + 5490182452540\nu^4 - 1525755421800\nu^3 + 
     238554360000\nu^2 \\&&\,- 16429392000\nu + 182891520\bigg),\\
 S_{4,6}(\nu) &=&\frac{1}{403200}(\nu - 1)(\nu - 2)\bigg(35430411200\nu^{13} - 439302994400\nu^{12}  + 
     2436971579924\nu^{11}\\&&  - 7978872930452\nu^{10} + 
     17124254615635\nu^9  - 25296180149635\nu^8 \\&& + 
     26269123904332\nu^7 - 19231646818886\nu^6 + 9796221569255\nu^5\\&&  - 
     3364159070155\nu^4 + 734427963174\nu^3 - 91281579672\nu^2\\&& + 
     5059464480\nu - 46569600\bigg), \\
     S_{4,7}(\nu) &=&\frac{1}{12441600}(\nu - 1)\bigg(59827528284865\nu^{15} - 855802750597179\nu^{14} + 
     5563694303002489\nu^{13}\\&& - 21750962904597519\nu^{12} + 
     57013384440297127\nu^{11}  - 105748729961235033\nu^{10}\\&& + 
     142745418985420307\nu^9 - 142003639611680997\nu^8 + 
     104226950891832476\nu^7 \\&&- 55920917465621352\nu^6 + 
     21478178444606384\nu^5  - 5692881115155984\nu^4 \\&&+ 
     978640246256832\nu^3 - 97543773524736\nu^2  + 4417361464320\nu - 
     34488115200\bigg),\\
 S_{4,8}(\nu) &=&\frac{1}{680400}(\nu - 1)\bigg(176898841310688\nu^{16} - 2539263011679856\nu^{15} + 
     16666305507262944\nu^{14} \\&& - 66246753115795672\nu^{13} + 
     178029651854014296\nu^{12}  - 341938430331281392\nu^{11} \\&& + 
     483779297595483414\nu^{10} - 512027961013384676\nu^9  + 
     407465369199877959\nu^8 \\&& - 242914597181711693\nu^7 + 
     107143566154013235\nu^6  - 34166128376363207\nu^5 \\&& + 
     7581940010533812\nu^4 - 1099393946405244\nu^3 + 
     93081049823952\nu^2 - 3607697439360\nu \\&& + 24518894400\bigg),\\
     S_{4,9}(\nu) &=&\frac{1}{1075200}(\nu - 1)\bigg(15226246439849967\nu^{17} - 220791879118376586\nu^{16} \\&&+ 
     1471803789729997792\nu^{15} - 5978558980896826172\nu^{14} + 
     16537756892056748422\nu^{13} \\&& - 32974240831015598528\nu^{12} + 
     48923952330822382696\nu^{11} - 54973107177374273684\nu^{10} \\ && + 
     47152509137361769699\nu^9 - 30881852184541961174\nu^8  + 
     15337643160511165168\nu^7 \\&&- 5692202029948990128\nu^6 + 
     1540102222359466992\nu^5 - 292096179194032992\nu^4 \\&&+ 
     36436319876824704\nu^3 - 2670501911809536\nu^2 + 
     90208099215360\nu - 542442700800\bigg),
     \end{eqnarray*} }
{\small
\begin{eqnarray*} 
 S_{4,10}(\nu) &=&\frac{1}{108864}(\nu - 1)\bigg(85562694562591904\nu^{18} - 1259062740590420160\nu^{17} \\&& + 
     8557751365316924280\nu^{16} - 35638727342326058592\nu^{15}+ 101707256601030693756\nu^{14} \\&& - 210757074791817347640\nu^{13} + 
     327796914873391461053\nu^{12} - 390093042791014411536\nu^{11} \\&&+ 
     358792753565910136791\nu^{10} - 255848897862582292782\nu^9 + 
     141024019215380064999\nu^8\\&& - 59538088006286478876\nu^7 + 
     18940378557686606717\nu^6 - 4424281811064392454\nu^5 \\&&+ 
     729043441916539860\nu^4 - 79471448739520200\nu^3 \\&&+ 
     5118523141929120\nu^2 - 152866927866240\nu + 
     823834851840\bigg).
\end{eqnarray*}
}

\section{Number of Labeled Connected $4$-valent Graphs with One or Two Vertices on the Sphere and the Torus}\label{Appendix Number of graphs}

In this appendix we include some illustrations as examples of graphical interpretations for the formulae in Theorem \ref{thm combinatorics}\footnote{These illustrations  also appeared in \cite{BGM} to enhance the interpretation of \cite[Equation (6.140)]{BGM} and \cite[Equation (6.141)]{BGM} for planar and toroidal counts of regular 4-valent graphs.}. We specifically do this for four-valent graphs with one and two vertices and will focus on the four formulae in Theorem \ref{thm combinatorics} corresponding to the choices $(\nu,g,j)\in\{(2,0,1),(2,0,2),(2,1,1),(2,1,2)\}$. To this end, recall that
 $C_2=2$, $S_{0,1}(2)=1$, $S_{0,2}(2)=9$, $S_{1,1}(2)=1/2$, and $S_{1,2}(2)=15$. So, from Theorem \ref{thm combinatorics} we find

\begin{equation}\label{numbers j=1 and j=2} \begin{split}
    	&\mathscr{N}_0(4,1)=C_2S_{0,1}(2)=2, \qquad 	\mathscr{N}_1(4,1)=C_2S_{1,1}(2)=1, \\ 	&\mathscr{N}_0(4,2)=C^2_2S_{0,2}(2)=36, \qquad	\mathscr{N}_1(4,2)=C^2_2S_{1,2}(2)=60.        
\end{split}
\end{equation}

The first two members of \eqref{numbers j=1 and j=2} are easy to verify. Consider a labeled 4-valent graph with one vertex \(v\). There are two ways to connect adjacent edges on the sphere, giving \(\mathscr{N}_0(4,1)=2\). Connecting opposite edges yields a graph only realizable on the torus, giving \(\mathscr{N}_1(4,1)=1\).

To justify the third member of \eqref{numbers j=1 and j=2}, label the two vertices \(v_1\) and \(v_2\), with edges \(e_1\) through \(e_4\) at \(v_1\) and \(e_5\) through \(e_8\) at \(v_2\), each labeled counterclockwise. A connection \(e_j \leftrightarrow e_k\) links edges. Starting with \(e_1\), it cannot connect to \(e_3\), as that would leave either \(e_2\) or \(e_4\) unmatched. It can connect to \(e_2\) or \(e_4\) in 8 distinct graphs, and to any of \(e_5\)–\(e_8\) in 5 distinct graphs each, confirming \(\mathscr{N}_0(4,2) = 2\cdot8 + 4\cdot5 = 36\). Figures~\ref{fig: g=0,1 j=2, e1->e_2} and~\ref{fig: g=0,1 j=2, e1->e_6} illustrate these cases\footnote{ In Figures~\ref{fig: g=0,1 j=2, e1->e_2}–\ref{fig: g=1 j=2, e1->e_6}, each edge \(e_k\) is simply denoted by \(k\).}.

\begin{figure}[h]
	\centering
	\begin{minipage}{0.2\textwidth}
		\centering
		\includegraphics[width=\textwidth, alt={All eight labeled connected 4-valent graphs with two vertices, realizable on the sphere with edge $e_1$ joined to $e_2$.}]{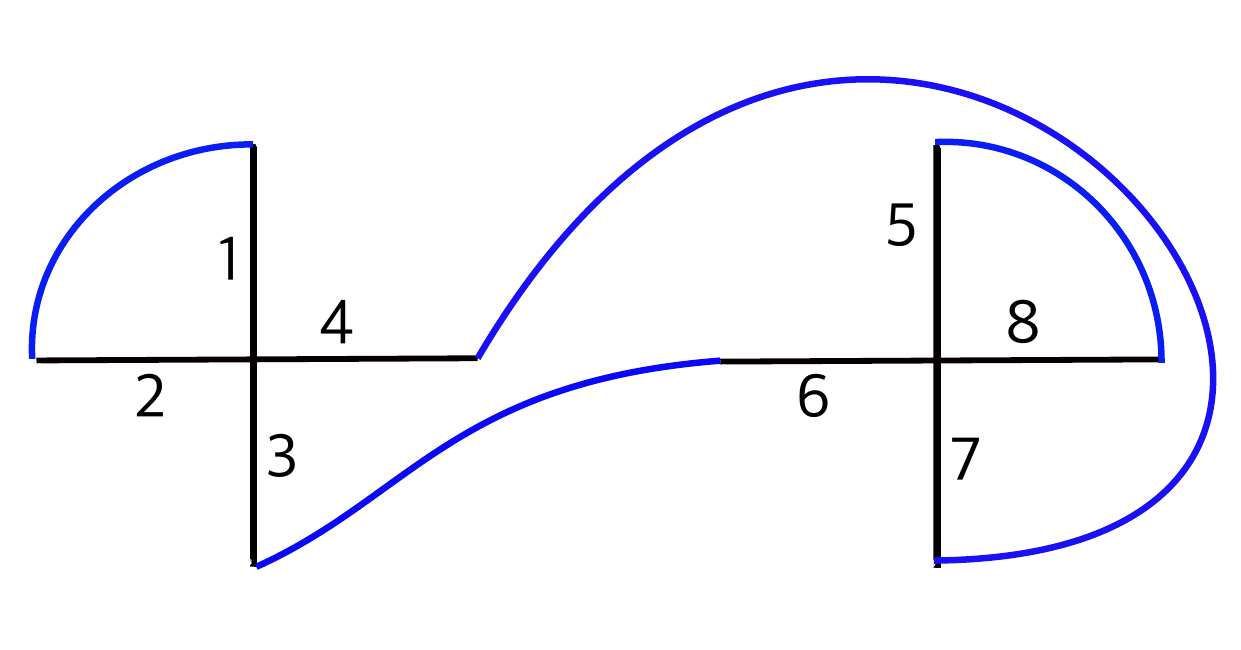}
        \subcaption{}
	\end{minipage} \hspace{0.3cm}
	\begin{minipage}{0.2\textwidth}
		\centering
		\includegraphics[width=\textwidth, alt={All eight labeled connected 4-valent graphs with two vertices, realizable on the sphere with edge $e_1$ joined to $e_2$.}]{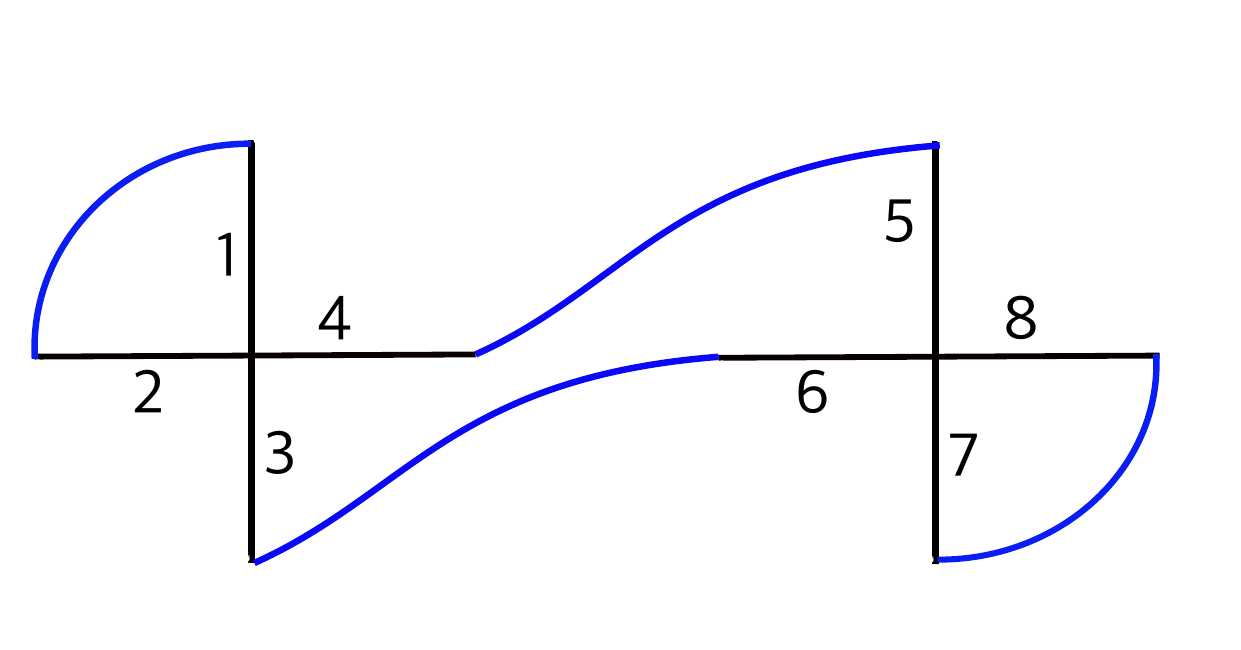}
        \subcaption{}
	\end{minipage} \hspace{0.3cm}
	\begin{minipage}{0.2\textwidth}
		\centering
		\includegraphics[width=\textwidth, alt={All eight labeled connected 4-valent graphs with two vertices, realizable on the sphere with edge $e_1$ joined to $e_2$.}]{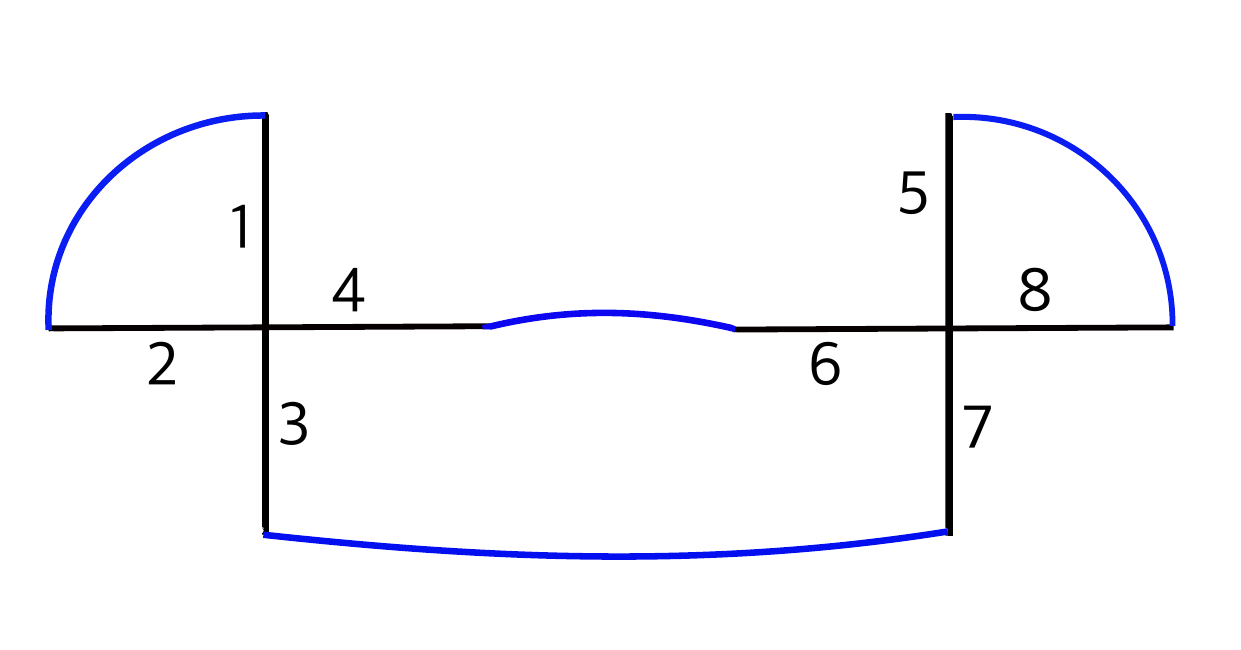}
        \subcaption{}
	\end{minipage} \hspace{0.3cm}
	\begin{minipage}{0.2\textwidth}
		\centering
		\includegraphics[width=\textwidth, alt={All eight labeled connected 4-valent graphs with two vertices, realizable on the sphere with edge $e_1$ joined to $e_2$.}]{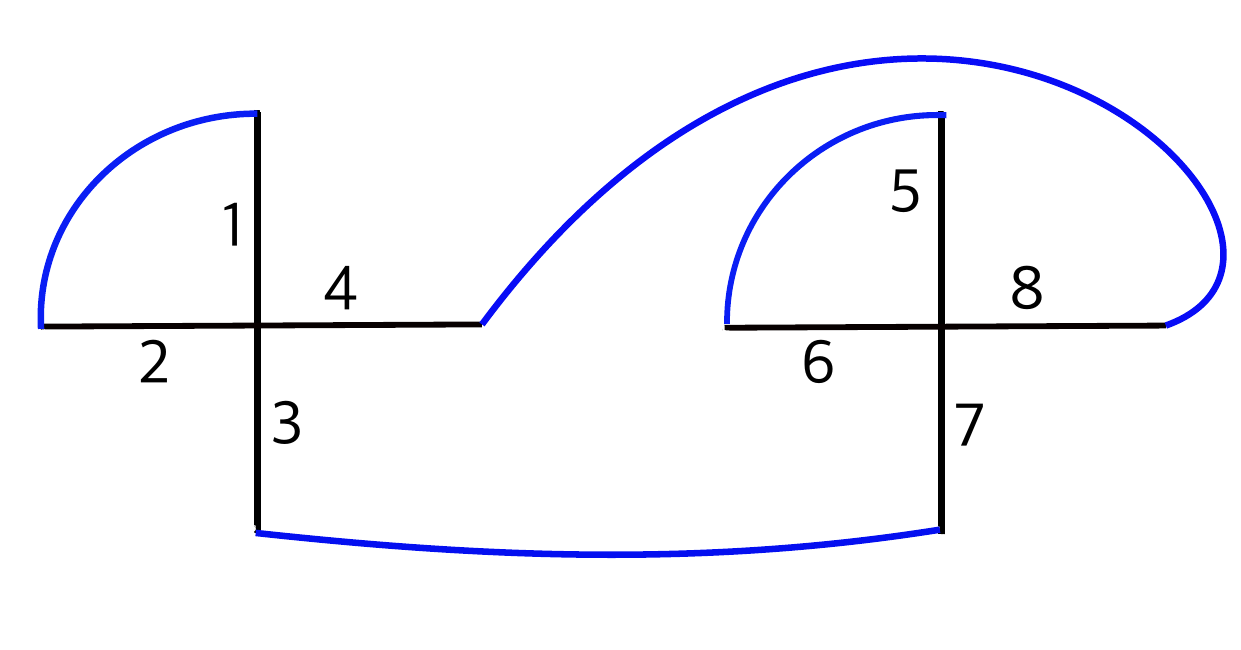}
        \subcaption{}
	\end{minipage}
	\centering
	\begin{minipage}{0.2\textwidth}
		\centering
		\includegraphics[width=\textwidth, alt={All eight labeled connected 4-valent graphs with two vertices, realizable on the sphere with edge $e_1$ joined to $e_2$.}]{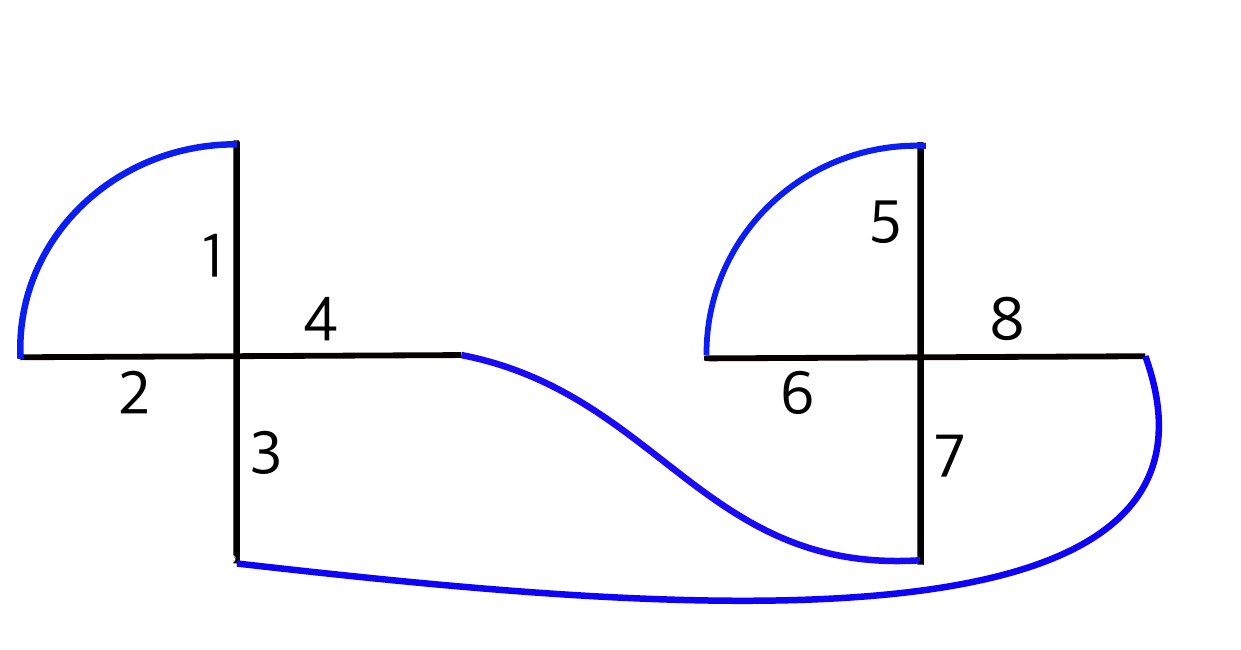}
        \subcaption{}
	\end{minipage} \hspace{0.3cm}
	\begin{minipage}{0.2\textwidth}
		\centering
		\includegraphics[width=\textwidth, alt={All eight labeled connected 4-valent graphs with two vertices, realizable on the sphere with edge $e_1$ joined to $e_2$.}]{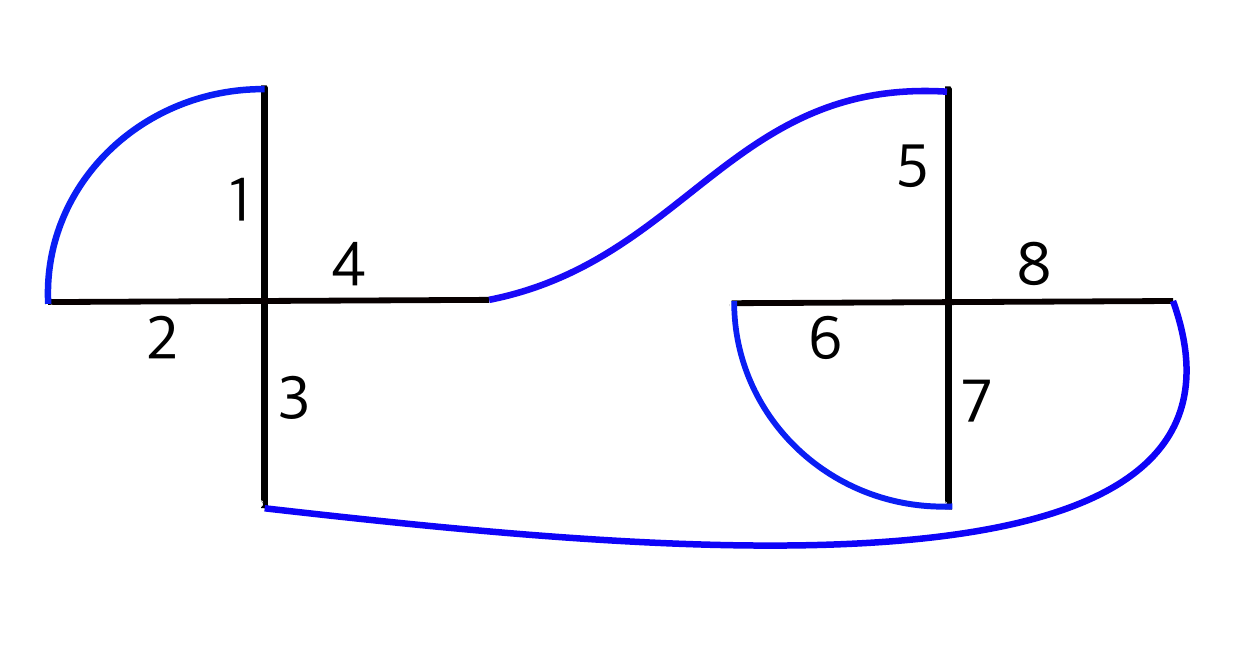}
        \subcaption{}
	\end{minipage} \hspace{0.3cm}
	\begin{minipage}{0.2\textwidth}
		\centering
		\includegraphics[width=\textwidth, alt={All eight labeled connected 4-valent graphs with two vertices, realizable on the sphere with edge $e_1$ joined to $e_2$.}]{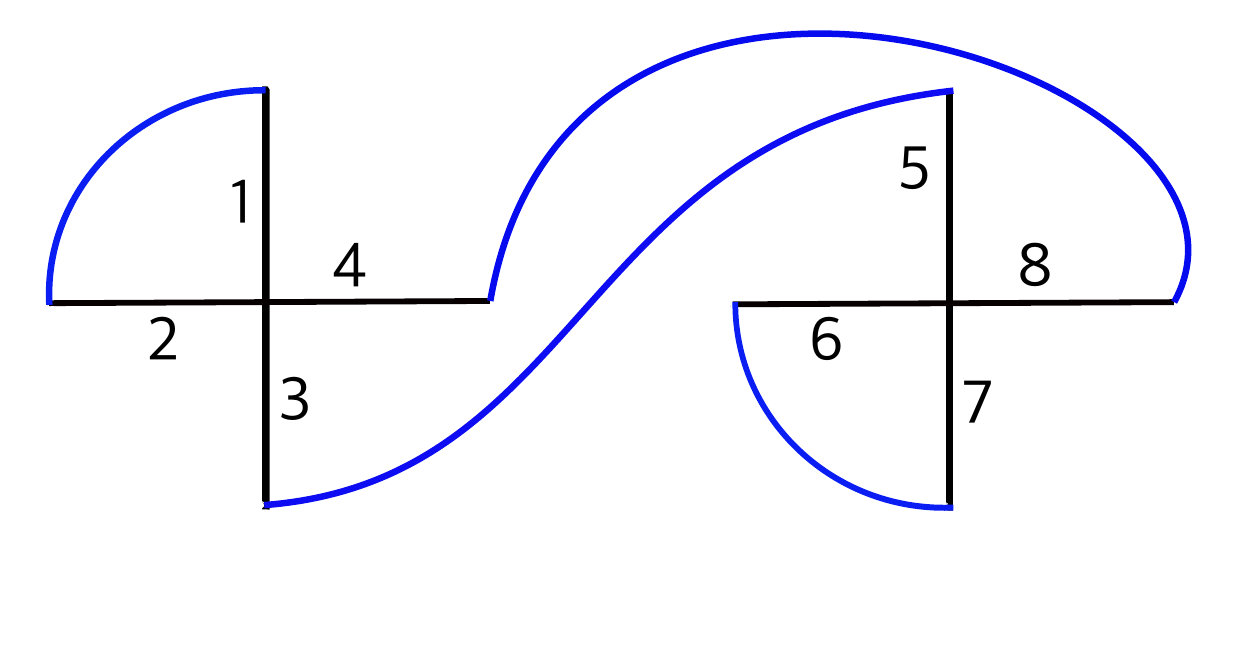}
        \subcaption{}
	\end{minipage} \hspace{0.3cm}
	\begin{minipage}{0.2\textwidth}
		\centering
		\includegraphics[width=\textwidth, alt={All eight labeled connected 4-valent graphs with two vertices, realizable on the sphere with edge $e_1$ joined to $e_2$.}]{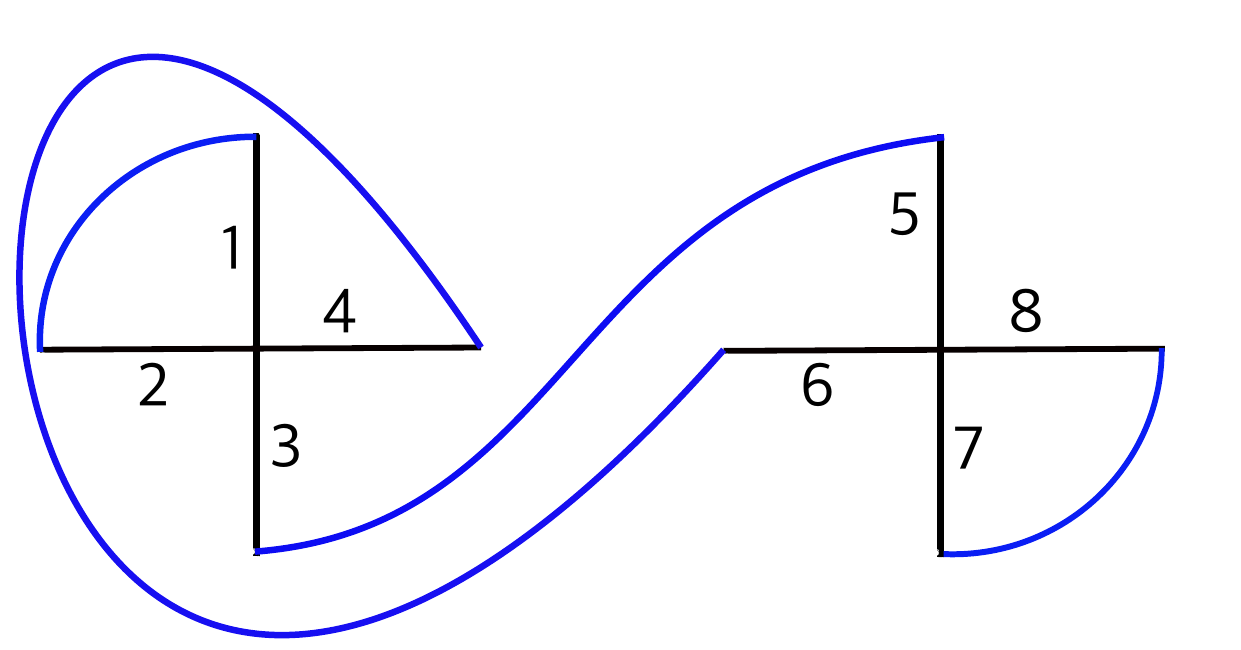}
        \subcaption{}
	\end{minipage}%
	\caption{All eight labeled connected 4-valent graphs with two vertices, where $e_1$ connects to $e_2$ and realizable on the sphere. Identically, for the case where $e_1$ connects to $e_4$, there are also eight distinct graphs. For the simplicity of the Figures \ref{fig: g=0,1 j=2, e1->e_2}, \ref{fig: g=0,1 j=2, e1->e_6}, \ref{fig: g=1 j=2, e1->e_2}, \ref{fig: g=1 j=2, e1->e_3}, and \ref{fig: g=1 j=2, e1->e_6} an edge $e_k$ will be simply denoted by $k$ on the graphs.}
	\label{fig: g=0,1 j=2, e1->e_2}
\end{figure}
\begin{figure}[h]
	\centering
	\begin{minipage}{0.2\textwidth}
		\centering
		\includegraphics[width=\textwidth, alt={All five labeled connected 4-valent graphs with two vertices, where $e_1$ connects to $e_6$ and realizable on the sphere.}]{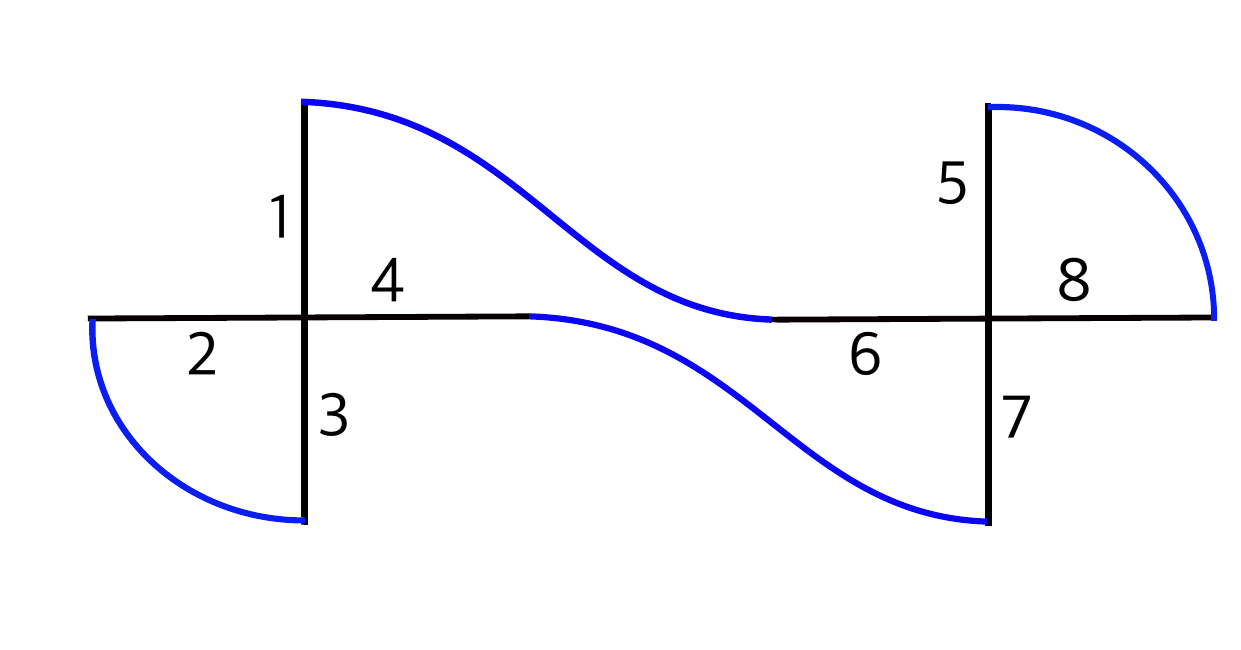}
        \subcaption{}
	\end{minipage} \hspace{0.5cm}
	\begin{minipage}{0.2\textwidth}
		\centering
		\includegraphics[width=\textwidth, alt={All five labeled connected 4-valent graphs with two vertices, where $e_1$ connects to $e_6$ and realizable on the sphere.}]{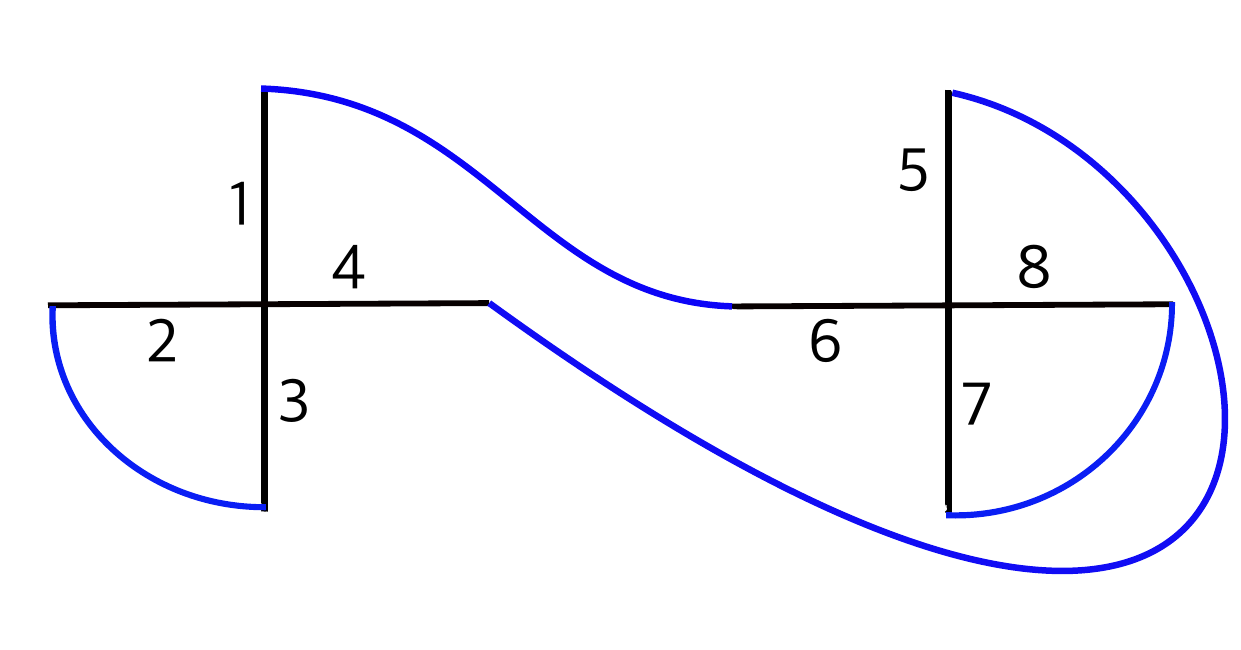}
        \subcaption{}
	\end{minipage} \hspace{0.5cm}
	\begin{minipage}{0.2\textwidth}
		\centering
		\includegraphics[width=\textwidth, alt={All five labeled connected 4-valent graphs with two vertices, where $e_1$ connects to $e_6$ and realizable on the sphere.}]{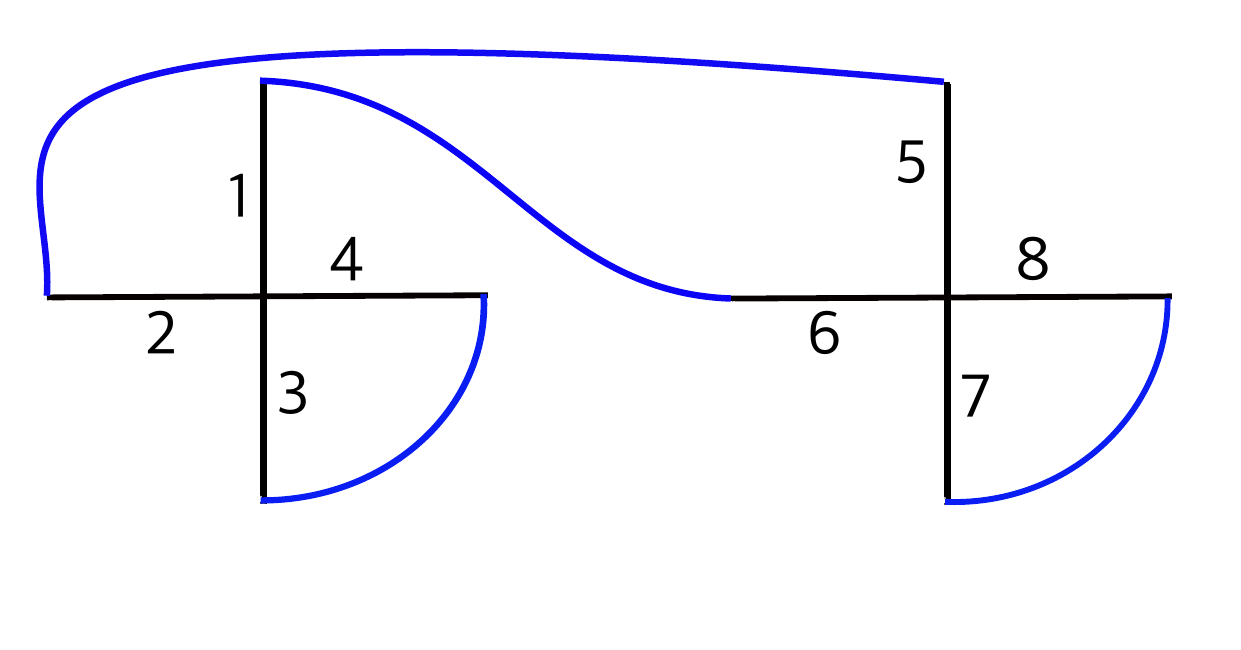}
        \subcaption{}
	\end{minipage}
	\centering
	\begin{minipage}{0.2\textwidth}
		\centering
		\includegraphics[width=\textwidth, alt={All five labeled connected 4-valent graphs with two vertices, where $e_1$ connects to $e_6$ and realizable on the sphere.}]{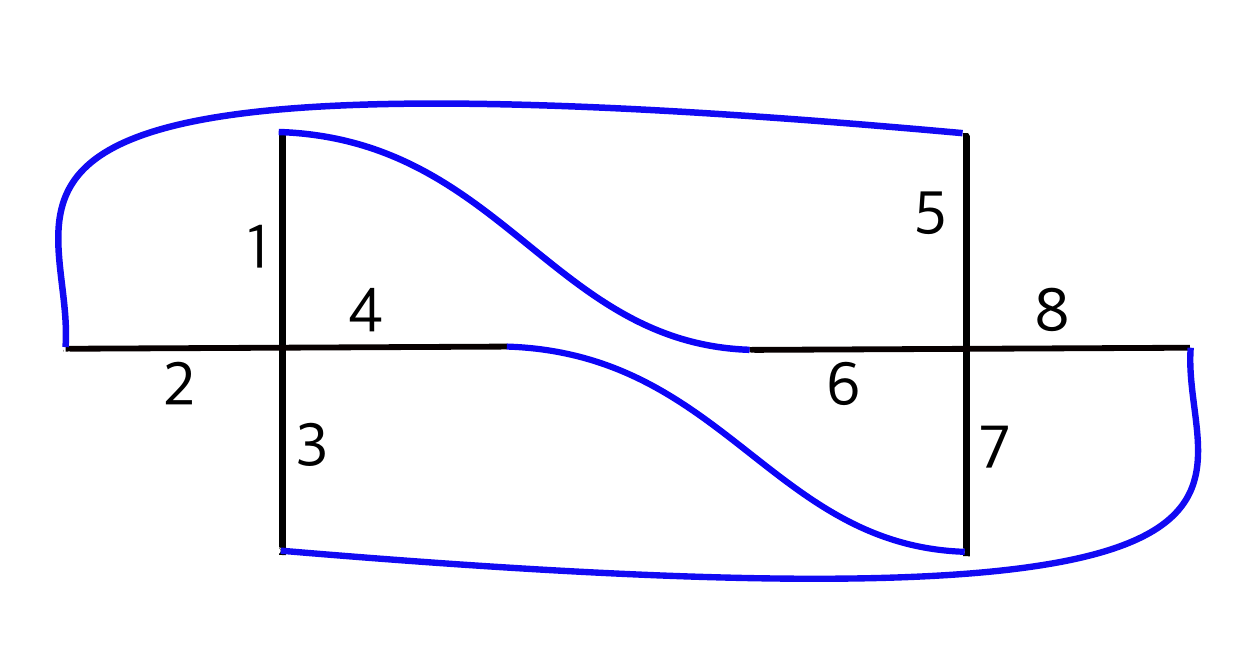}
        \subcaption{}
	\end{minipage} \hspace{0.5cm}
	\begin{minipage}{0.2\textwidth}
		\centering
		\includegraphics[width=\textwidth, alt={All five labeled connected 4-valent graphs with two vertices, where $e_1$ connects to $e_6$ and realizable on the sphere.}]{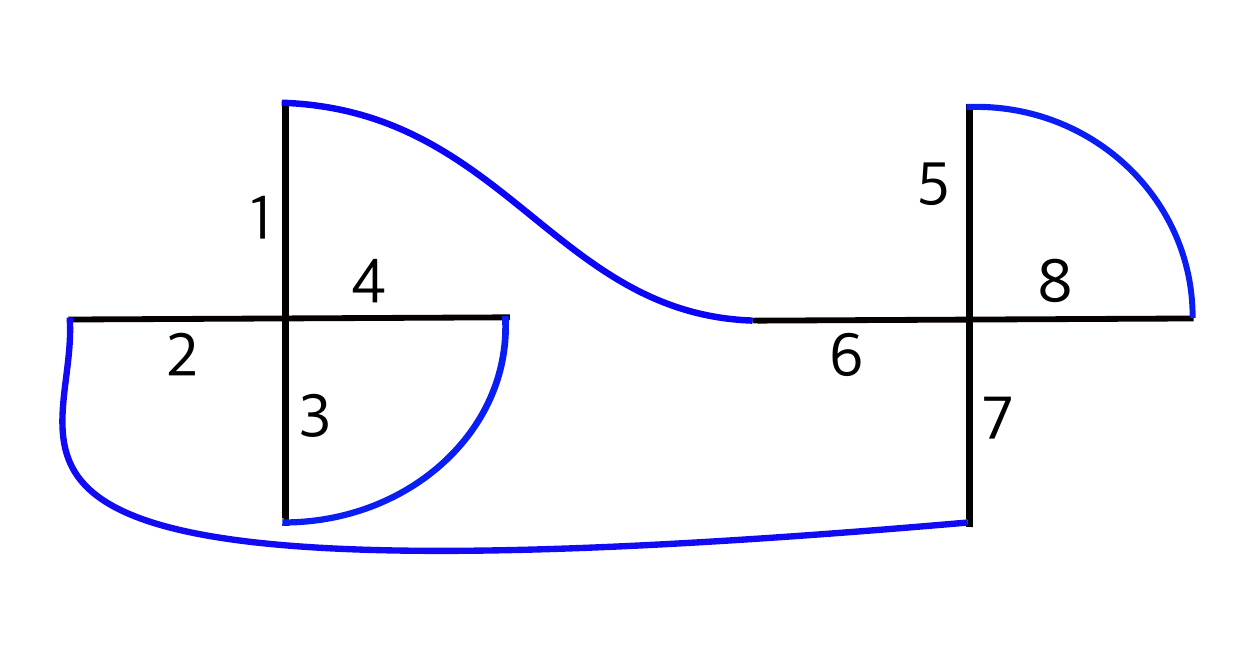}
        \subcaption{}
	\end{minipage} 
	\caption{All five labeled connected 4-valent graphs with two vertices, where $e_1$ connects to $e_6$ and realizable on the sphere. Identically, for each of the cases where $e_1$ connects to $e_5$, $e_7$, or $e_8$, there are also five distinct graphs. this Figure together with Figure \ref{fig: g=0,1 j=2, e1->e_2} confirm that $\mathscr{N}_2(0)=2\cdot8+4\cdot5=36$.}
	\label{fig: g=0,1 j=2, e1->e_6}
\end{figure}

We now interpret \(\mathscr{N}_1(4,2) = 60\) combinatorially. With two 4-valent vertices, it is easy to see that there are three distinct connected labeled graphs with two enforced connections. Since two graphs with connections \(e_1 \leftrightarrow e_2\) and \(e_3 \leftrightarrow e_6\) already appear on the sphere (see the first two graphs in Figure~\ref{fig: g=0,1 j=2, e1->e_2}), one more remains to be realized on the torus. The same holds for the combinations a) \(e_1 \leftrightarrow e_2\) \& \(e_3 \leftrightarrow e_7\), b) \(e_1 \leftrightarrow e_2\) \& \(e_3 \leftrightarrow e_8\), and c) \(e_1 \leftrightarrow e_2\) \& \(e_3 \leftrightarrow e_5\), giving four graphs in total with \(e_1 \leftrightarrow e_2\) on the torus. Additionally, four graphs with \(e_1 \leftrightarrow e_4\) are realizable only on the torus, totaling eight such graphs (see Figure~\ref{fig: g=1 j=2, e1->e_2}). 

\begin{figure}[!htp]
	\centering
	\begin{minipage}{0.24\textwidth}
		\centering
		\includegraphics[width=\textwidth, alt={All four labeled connected 4-valent graphs with two vertices, where $e_1$ connects to $e_2$ which are not realizable on the sphere.}]{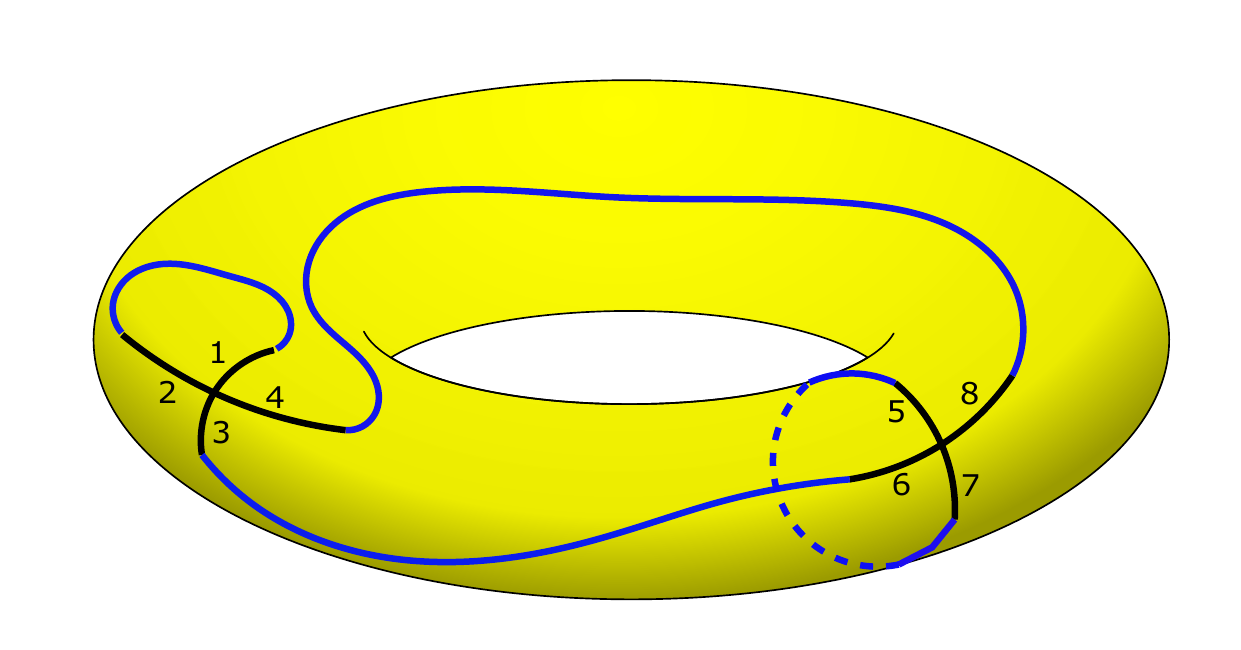}
        \subcaption{}
	\end{minipage} \hfil
	\begin{minipage}{0.24\textwidth}
		\centering
		\includegraphics[width=\textwidth, alt={All four labeled connected 4-valent graphs with two vertices, where $e_1$ connects to $e_2$ which are not realizable on the sphere.}]{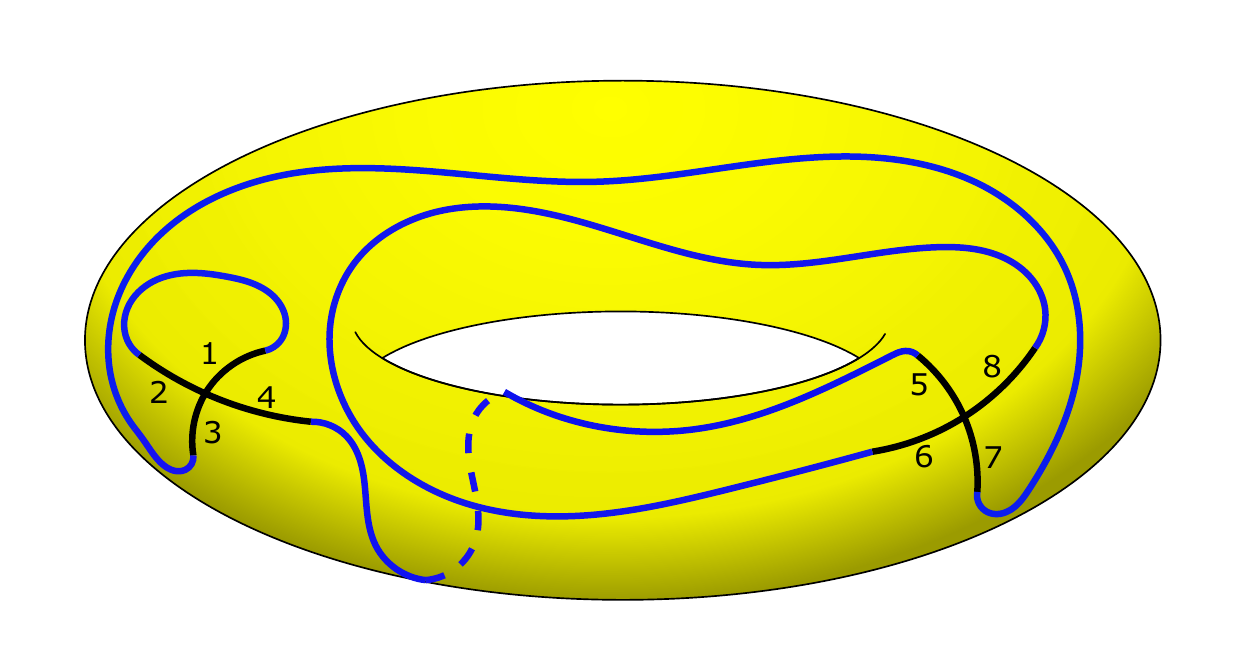}
        \subcaption{}
	\end{minipage} \hfil
	\begin{minipage}{0.24\textwidth}
		\centering
		\includegraphics[width=\textwidth, alt={All four labeled connected 4-valent graphs with two vertices, where $e_1$ connects to $e_2$ which are not realizable on the sphere.}]{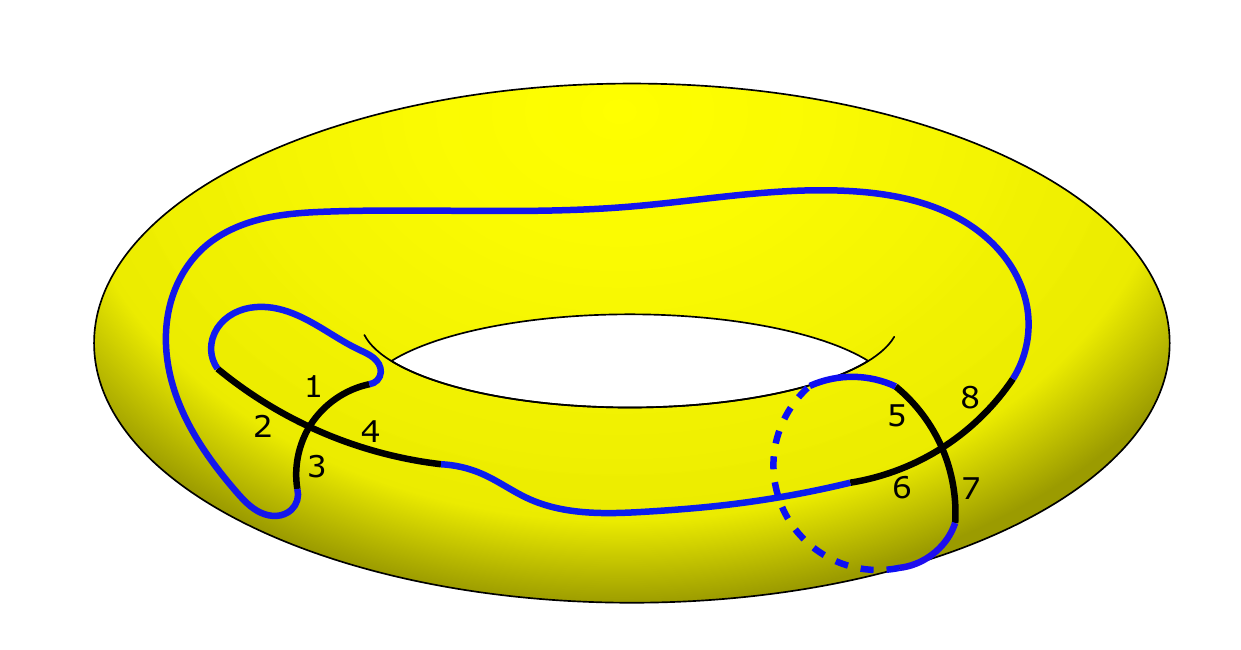}
        \subcaption{}
	\end{minipage}
	\hfil
	\begin{minipage}{0.24\textwidth}
		\centering
		\includegraphics[width=\textwidth, alt={All four labeled connected 4-valent graphs with two vertices, where $e_1$ connects to $e_2$ which are not realizable on the sphere.}]{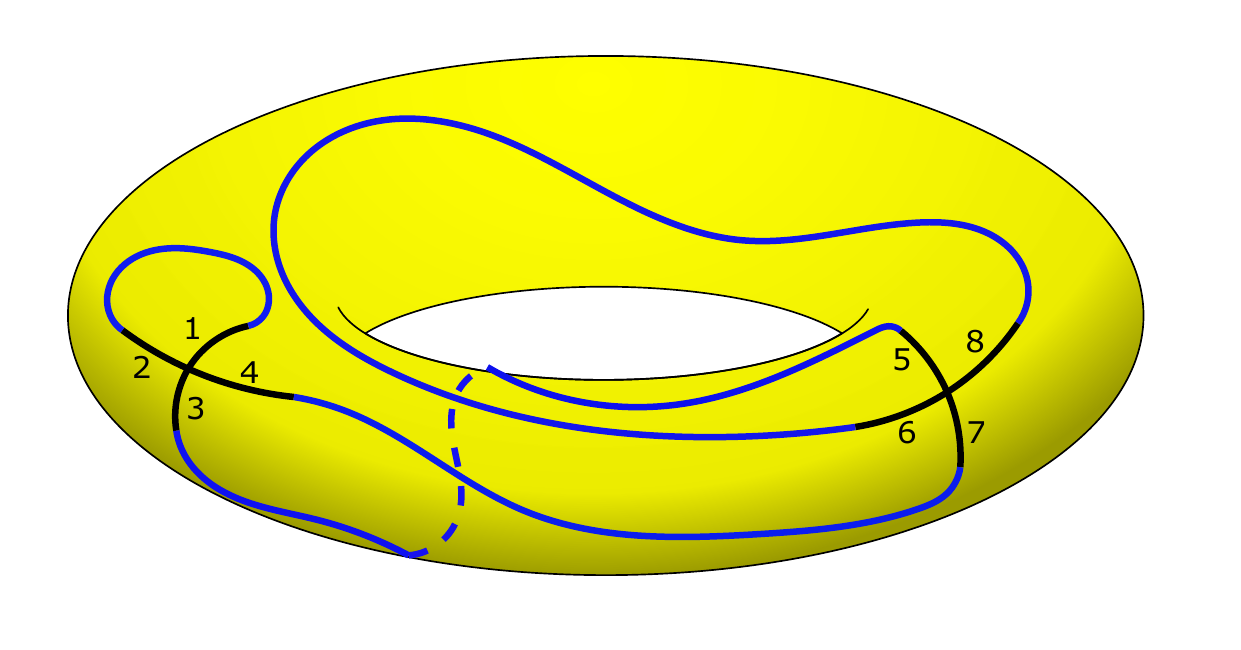}
        \subcaption{}
	\end{minipage}
	\caption{All four labeled connected 4-valent graphs with two vertices, where $e_1$ connects to $e_2$ which are not realizable on the sphere(compare with Figure \ref{fig: g=0,1 j=2, e1->e_2}). Identically, for the case where $e_1$ connects to $e_4$, there are also four distinct graphs.}
	\label{fig: g=1 j=2, e1->e_2}
\end{figure}

In Figure~\ref{fig: g=0,1 j=2, e1->e_6}, we already have two graphs each with the connections \(e_1 \leftrightarrow e_6\) and \(e_2 \leftrightarrow e_3\) or \(e_5\), and one with \(e_2 \leftrightarrow e_7\). Thus, having fixed the connection \(e_1 \leftrightarrow e_6\), one graph with \(e_2 \leftrightarrow e_3\), one with \(e_2 \leftrightarrow e_5\), and two with \(e_2 \leftrightarrow e_7\) remain to be realized on the torus (see the first four graphs in Figure~\ref{fig: g=1 j=2, e1->e_6}).

Although no graphs with \(e_2 \leftrightarrow e_8\) or \(e_4\) appear on the sphere, six such graphs can be realized on the torus (last six graphs in Figure~\ref{fig: g=1 j=2, e1->e_6}). This gives 10 graphs with \(e_1 \leftrightarrow e_6\) on the torus only. Similarly, for each of \(e_1 \leftrightarrow e_5\), \(e_1 \leftrightarrow e_7\), and \(e_1 \leftrightarrow e_8\), there are 10 graphs only realizable on the torus. In total, Figures~\ref{fig: g=1 j=2, e1->e_2} and~\ref{fig: g=1 j=2, e1->e_6} account for \(2 \cdot 4 + 4 \cdot 10 = 48\) such graphs.

Finally, we count the graphs with \(e_1 \leftrightarrow e_3\), which are not realizable on the sphere. Fixing \(e_1 \leftrightarrow e_3\), edge \(e_2\) can connect to any of \(e_5\)–\(e_8\) (but not \(e_4\)). With 3 distinct configurations per case, this yields \(4 \cdot 3 = 12\) additional graphs. Together, these give \(\mathscr{N}_1(4,2) = 48 + 12 = 60\).

\begin{figure}[!htp]
	\centering
	\begin{minipage}{0.24\textwidth}
		\centering
		\includegraphics[width=\textwidth, alt={All three labeled connected 4-valent graphs with two vertices, where $e_1 \leftrightarrow e_3$ \& $e_2 \leftrightarrow e_8$.}]{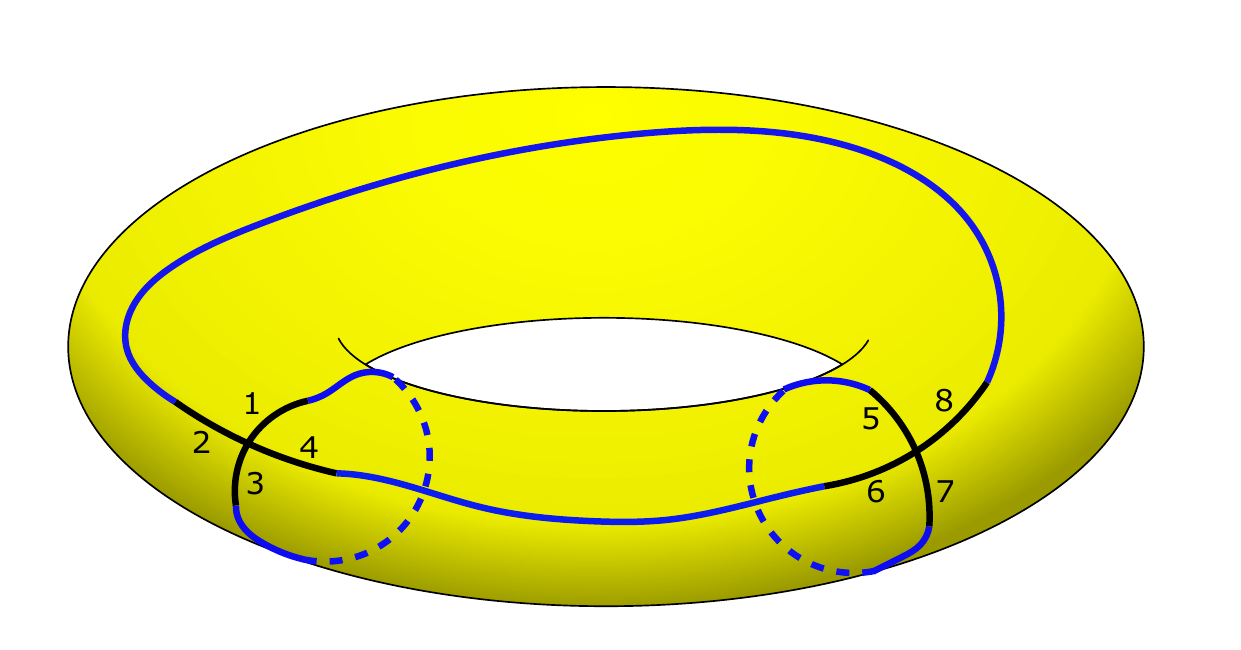}
        \subcaption{}
	\end{minipage} \hfil
	\begin{minipage}{0.24\textwidth}
		\centering
		\includegraphics[width=\textwidth, alt={All three labeled connected 4-valent graphs with two vertices, where $e_1 \leftrightarrow e_3$ \& $e_2 \leftrightarrow e_8$.}]{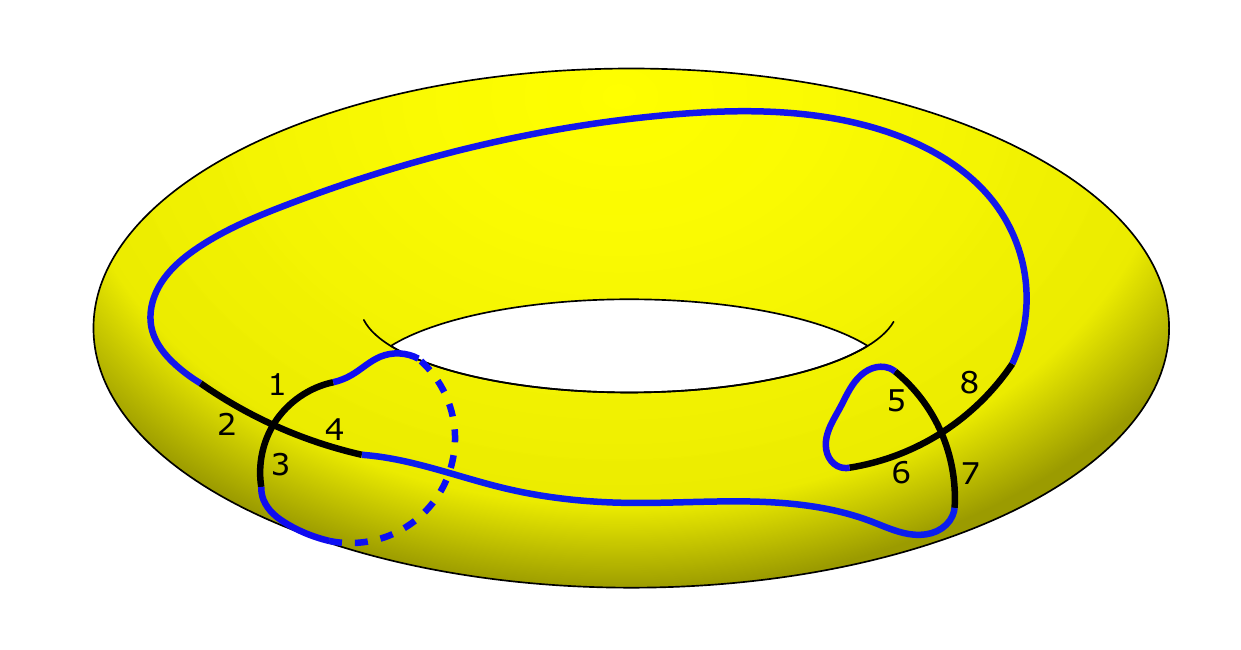}
        \subcaption{}
	\end{minipage} \hfil
	\begin{minipage}{0.24\textwidth}
		\centering
		\includegraphics[width=\textwidth, alt={All three labeled connected 4-valent graphs with two vertices, where $e_1 \leftrightarrow e_3$ \& $e_2 \leftrightarrow e_8$.}]{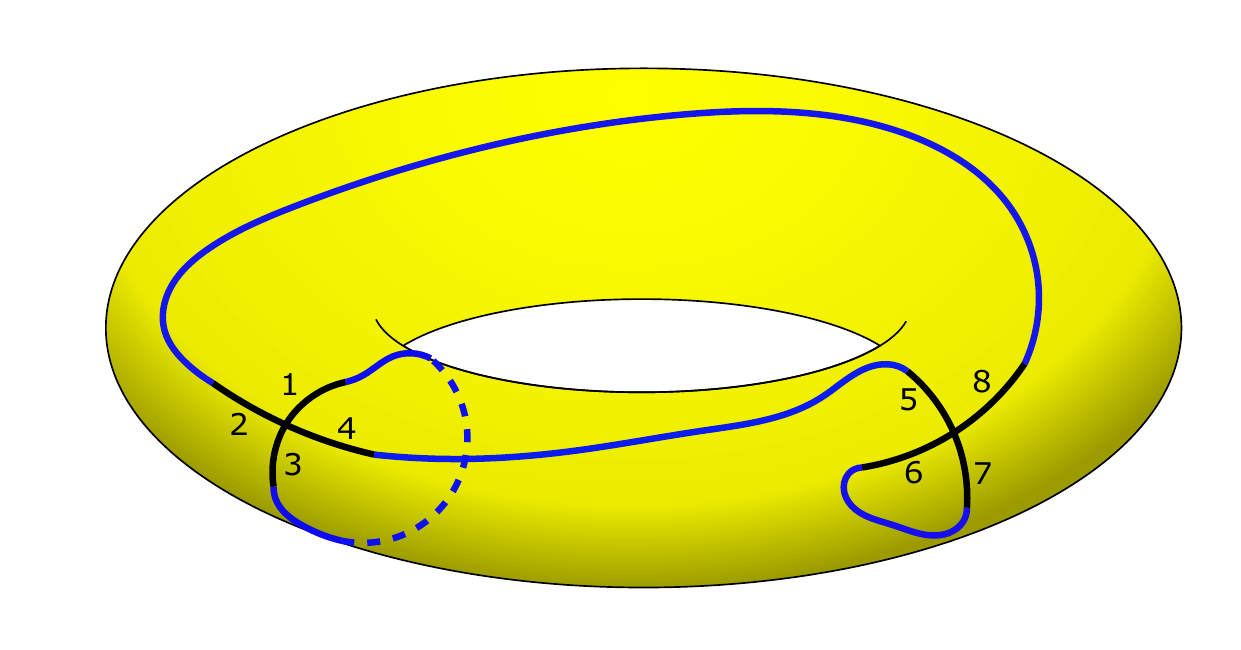}
        \subcaption{}
	\end{minipage}
	\caption{All three labeled connected 4-valent graphs with two vertices, where $e_1 \leftrightarrow e_3$ \& $e_2 \leftrightarrow e_8$. Identically, for each of the cases $e_1 \leftrightarrow e_3$ \& $e_2 \leftrightarrow e_5$, $e_1 \leftrightarrow e_3$ \& $e_2 \leftrightarrow e_7$, and $e_1 \leftrightarrow e_3$ \& $e_2 \leftrightarrow e_6$ there are three distinct graphs. Thus there exists $4 \cdot 3 =12$ distinct graphs with two vertices and $e_1 \leftrightarrow e_3$ realizable on the torus.}
	\label{fig: g=1 j=2, e1->e_3}
\end{figure}
\begin{figure}[!htp]
	\centering
	\begin{minipage}{0.24\textwidth}
		\centering
		\includegraphics[width=\textwidth, alt={All ten labeled connected 4-valent graphs with two vertices, where $e_1$ connects to $e_6$ which are not realizable on the sphere}]{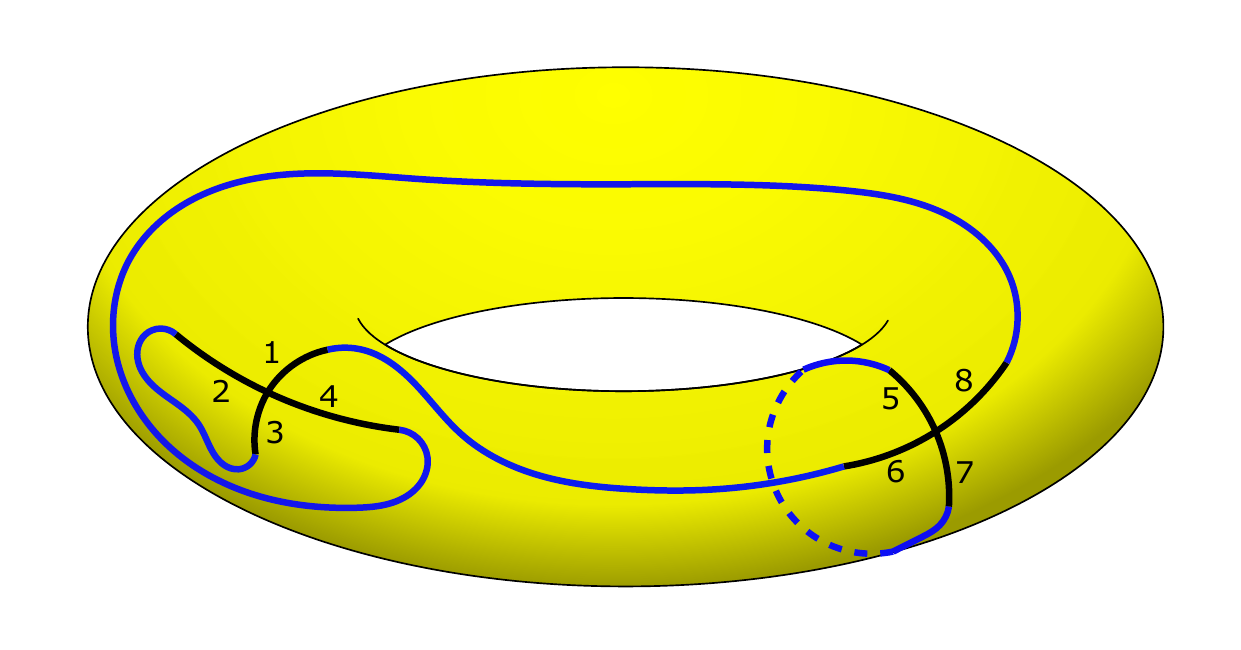}
        \subcaption{}
	\end{minipage} \hfil
	\begin{minipage}{0.24\textwidth}
		\centering
		\includegraphics[width=\textwidth, alt={All ten labeled connected 4-valent graphs with two vertices, where $e_1$ connects to $e_6$ which are not realizable on the sphere}]{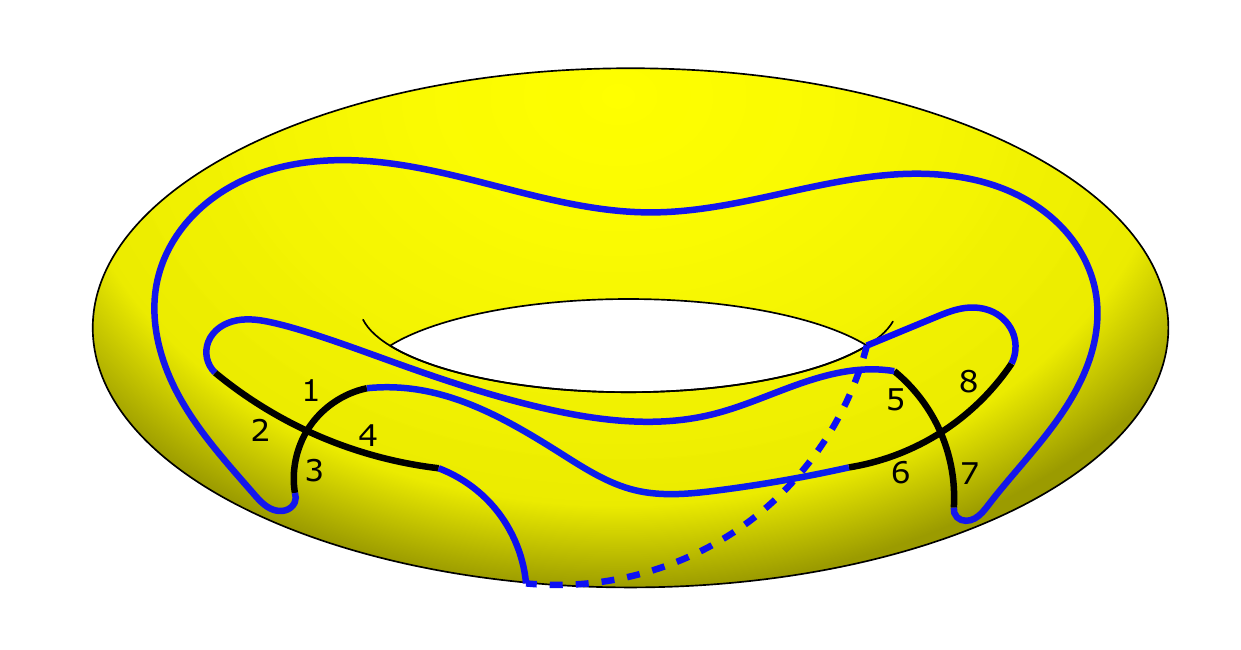}
        \subcaption{}
	\end{minipage} \hfil
	\begin{minipage}{0.24\textwidth}
		\centering
		\includegraphics[width=\textwidth, alt={All ten labeled connected 4-valent graphs with two vertices, where $e_1$ connects to $e_6$ which are not realizable on the sphere}]{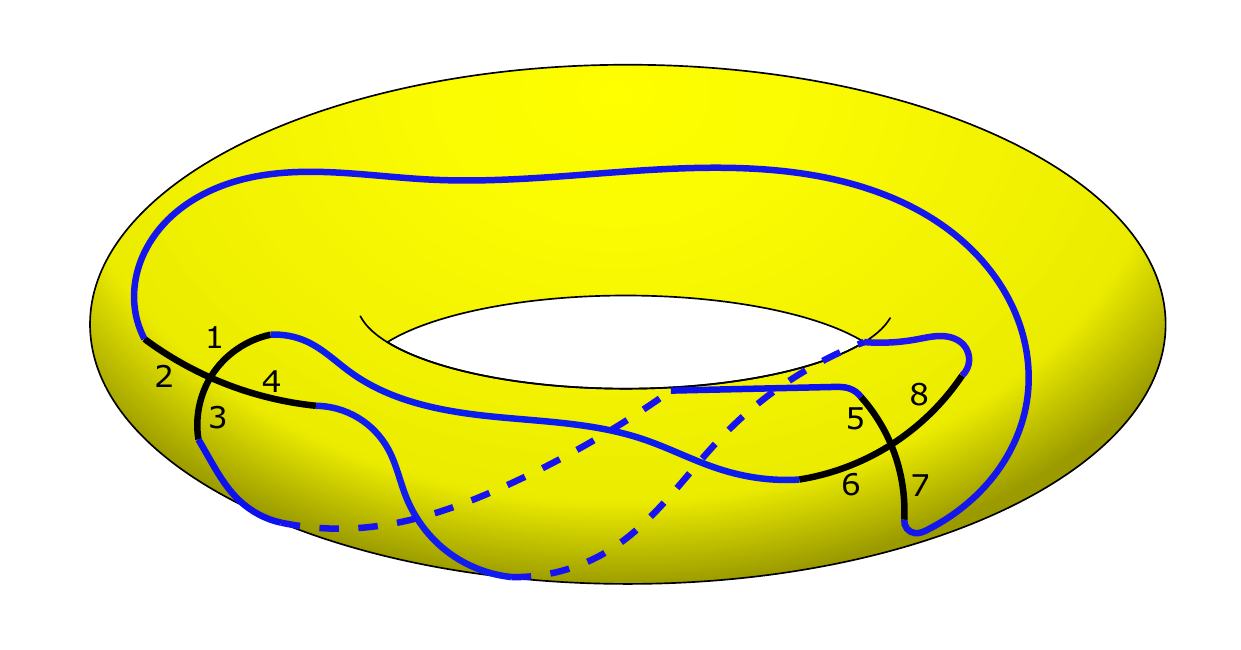}
        \subcaption{}
	\end{minipage}
	\hfil
	\begin{minipage}{0.24\textwidth}
		\centering
		\includegraphics[width=\textwidth, alt={All ten labeled connected 4-valent graphs with two vertices, where $e_1$ connects to $e_6$ which are not realizable on the sphere}]{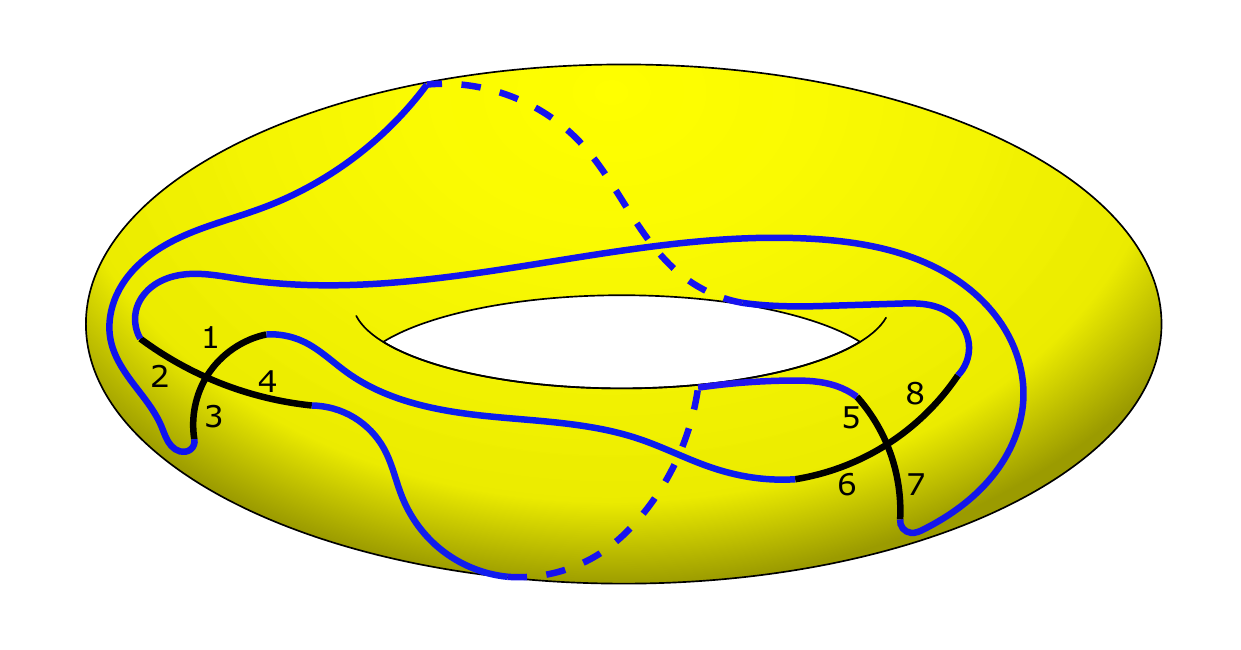}
        \subcaption{}
	\end{minipage}
	
	\medskip
	
	\centering
	\begin{minipage}{0.24\textwidth}
		\centering
		\includegraphics[width=\textwidth, alt={All ten labeled connected 4-valent graphs with two vertices, where $e_1$ connects to $e_6$ which are not realizable on the sphere}]{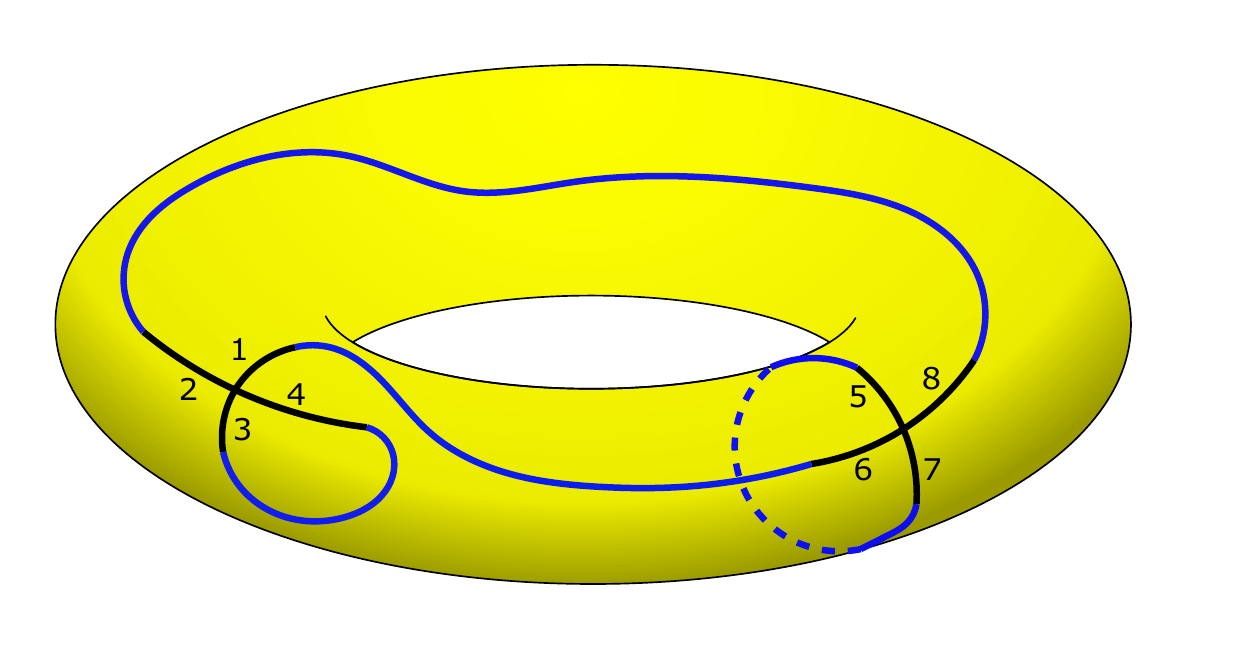}
        \subcaption{}
	\end{minipage} \hfil
	\begin{minipage}{0.24\textwidth}
		\centering
		\includegraphics[width=\textwidth, alt={All ten labeled connected 4-valent graphs with two vertices, where $e_1$ connects to $e_6$ which are not realizable on the sphere}]{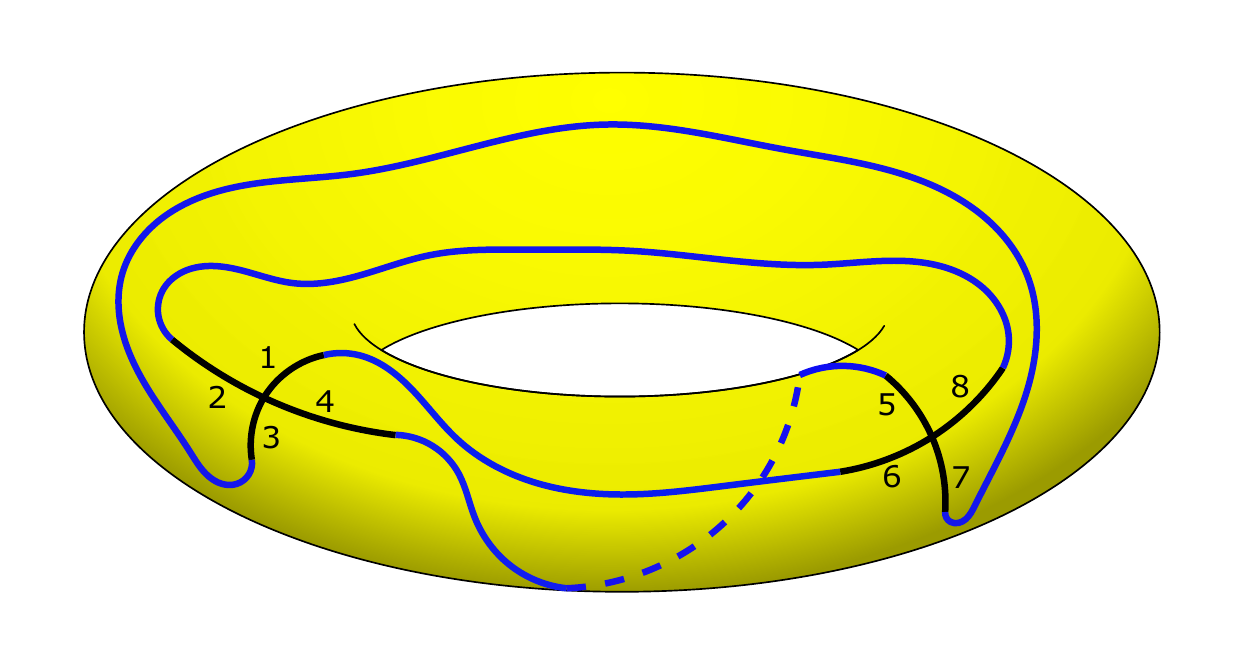}
        \subcaption{}
	\end{minipage} \hfil
	\begin{minipage}{0.24\textwidth}
		\centering
		\includegraphics[width=\textwidth, alt={All ten labeled connected 4-valent graphs with two vertices, where $e_1$ connects to $e_6$ which are not realizable on the sphere}]{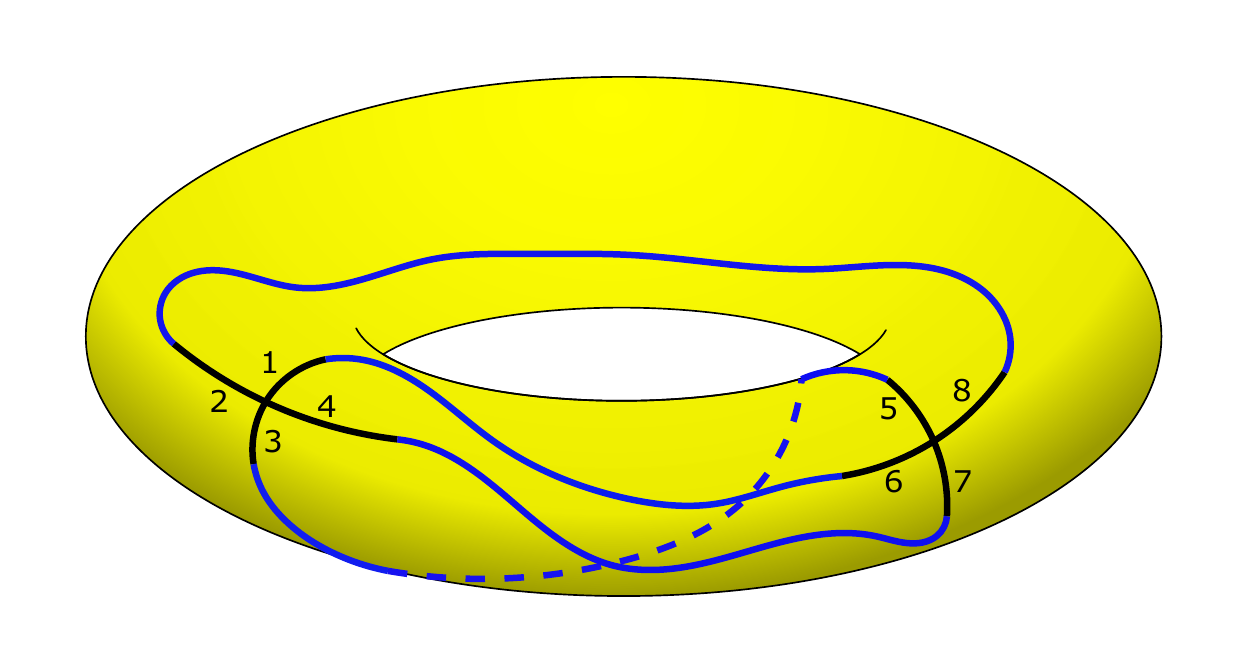}
        \subcaption{}
	\end{minipage}
	\hfil
	\begin{minipage}{0.24\textwidth}
		\centering
		\includegraphics[width=\textwidth, alt={All ten labeled connected 4-valent graphs with two vertices, where $e_1$ connects to $e_6$ which are not realizable on the sphere}]{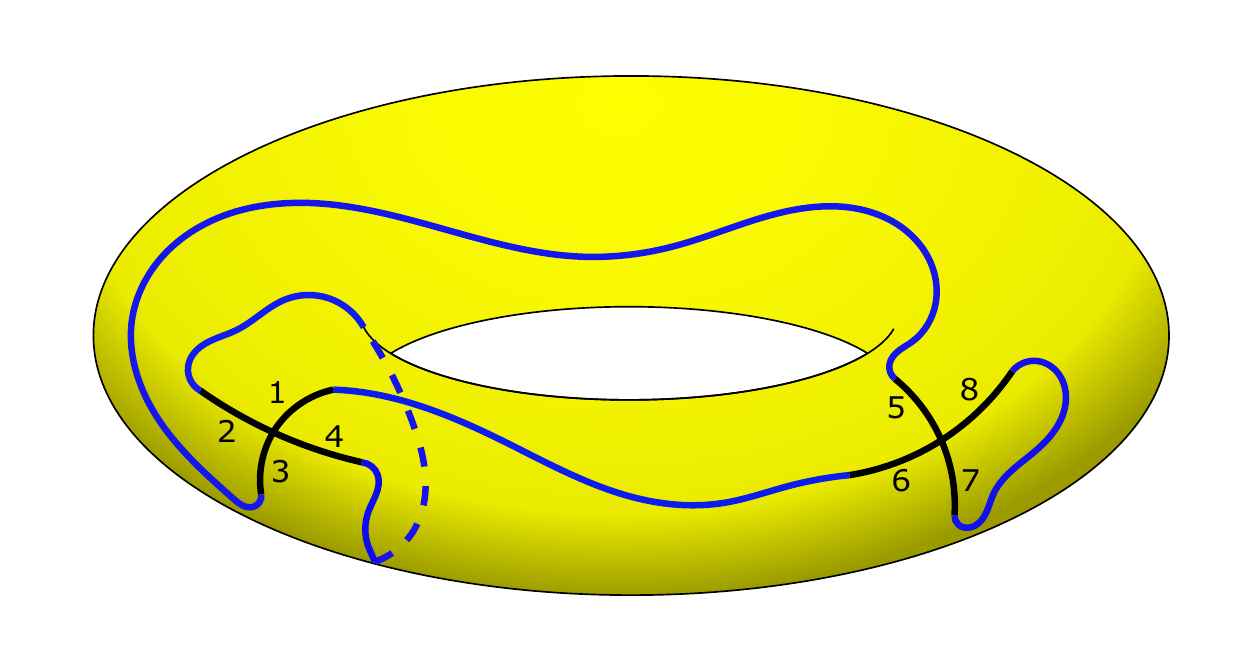}
        \subcaption{}
	\end{minipage}
	
	\medskip
	
	\centering
	\begin{minipage}{0.24\textwidth}
		\centering
		\includegraphics[width=\textwidth, alt={All ten labeled connected 4-valent graphs with two vertices, where $e_1$ connects to $e_6$ which are not realizable on the sphere}]{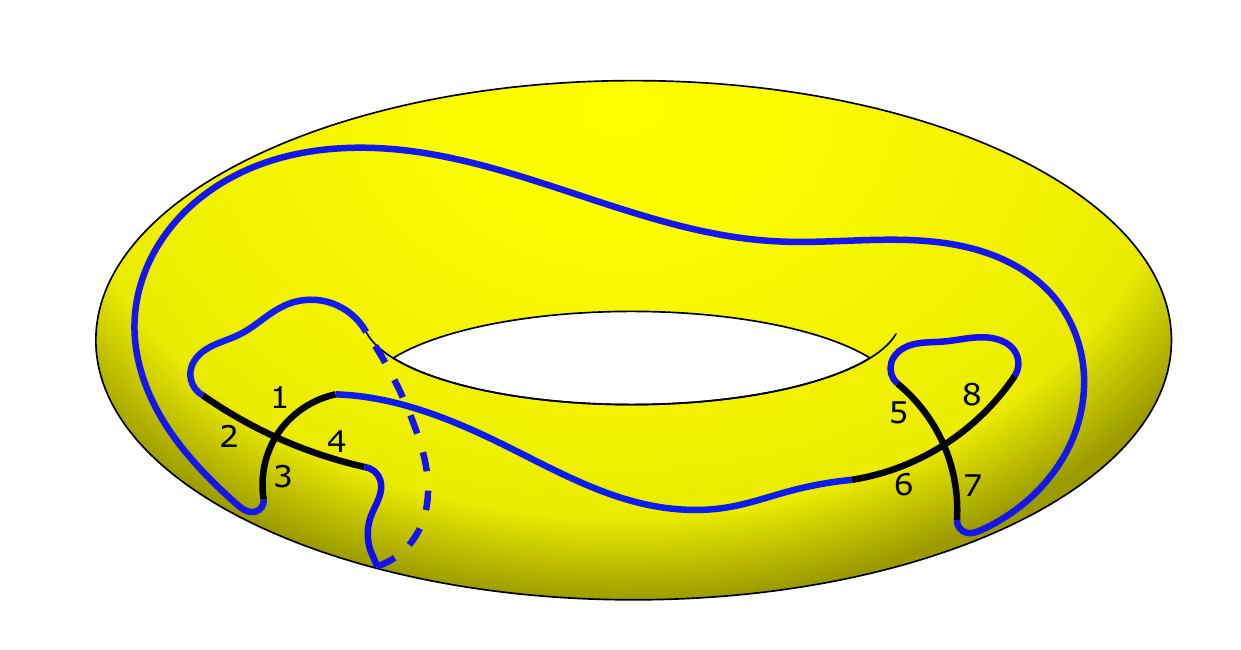}
        \subcaption{}
	\end{minipage} \hfil
	\begin{minipage}{0.24\textwidth}
		\centering
		\includegraphics[width=\textwidth, alt={All ten labeled connected 4-valent graphs with two vertices, where $e_1$ connects to $e_6$ which are not realizable on the sphere}]{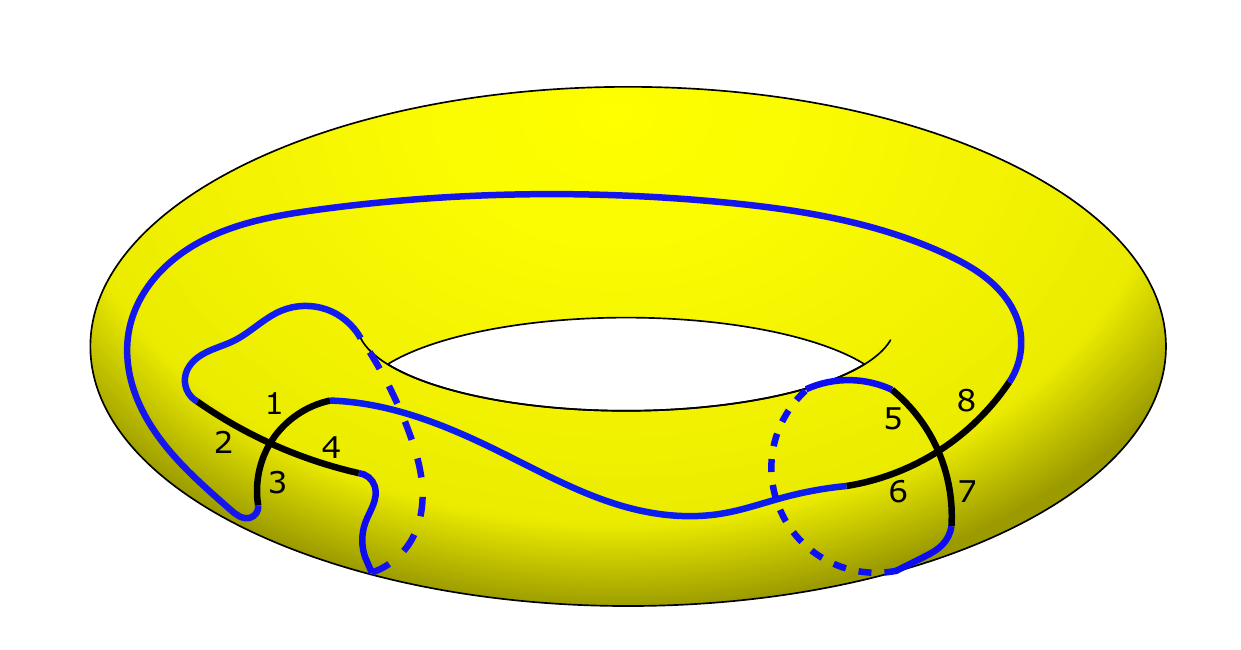}
        \subcaption{}
	\end{minipage}
	\caption{All ten labeled connected 4-valent graphs with two vertices, where $e_1$ connects to $e_6$ which are not realizable on the sphere (compare with Figure \ref{fig: g=0,1 j=2, e1->e_6}). Identically, for each one of the cases $e_1 \leftrightarrow e_5$, $e_1 \leftrightarrow e_7$, and $e_1 \leftrightarrow e_8$ there also exist 10 distinct graphs. This means that there are totally 40 distinct graphs, not realizable on the sphere, where $e_1$ connects to one of the edges emanating from $v_2$.}
	\label{fig: g=1 j=2, e1->e_6}
\end{figure}

\newpage

\end{appendices}

\bibliographystyle{amsplain}
\bibliography{main}
\end{document}